\documentclass[11pt]{amsart}
\usepackage{amsmath}
\usepackage[psamsfonts]{amssymb}
\usepackage[all]{xy}
\usepackage{graphicx}
\usepackage[T1]{fontenc}
\usepackage[bitstream-charter]{mathdesign}
\usepackage{hyperref}
\usepackage{verbatim}
\usepackage[english]{babel}
\usepackage{color}
\usepackage{mathtools}
\let\oldtocsection=\tocsection
\let\oldtocsubsection=\tocsubsection
\renewcommand{\tocsection}[2]{\hspace{0em}\oldtocsection{#1}{#2}}
\renewcommand{\tocsubsection}[2]{\hspace{1em}\oldtocsubsection{#1}{#2}}
\theoremstyle{definition}
\newtheorem{theorem}{Theorem}[section]
\newtheorem{prop}[theorem]{Proposition}
\newtheorem{lemma}[theorem]{Lemma}
\newtheorem{cor}[theorem]{Corollary}

\newtheorem{ex}[theorem]{Example}
\theoremstyle{remark}
\newtheorem{dfn}[theorem]{Definition}
\newtheorem{remark}[theorem]{Remark}

\def\co{\colon\thinspace}

\def\ep{\epsilon}

\def\R{\mathbb{R}}
\def\Z{\mathbb{Z}}
\def\N{\mathbb{N}}
\def\Q{\mathbb{Q}}
\def\C{\mathbb{C}}

\DeclareMathOperator{\Img}{Im}
\title{Symplectic Banach-Mazur distances between subsets of $\C^n$}
\author{Michael Usher}

\begin{document}
\maketitle
\begin{abstract}
Following proposals of Ostrover and Polterovich, we introduce and study ``coarse'' and ``fine'' versions of a symplectic Banach-Mazur distance on certain open subsets of $\mathbb{C}^n$ and other open Liouville domains.  The coarse version declares two such domains to be close to each other if each domain admits a Liouville embedding into a slight dilate of the other; the fine version, which is similar to the distance on subsets of cotangent bundles of surfaces recently studied by Stojisavljevi\'c and Zhang, imposes an additional requirement on the images of these embeddings that is motivated by the definition of the classical Banach-Mazur distance on convex bodies.  Our first main result is that the coarse and fine distances are quite different from each other, in that there are sequences that converge coarsely to an ellipsoid but diverge to infinity with respect to the fine distance.  Our other main result is that, with respect to the fine distance, the space of star-shaped domains in $\mathbb{C}^n$ admits quasi-isometric embeddings of $\R^D$ for every finite dimension $D$.  Our constructions are obtained from a general method of constructing $(2n+2)$-dimensional Liouville domains whose boundaries have Reeb dynamics determined by certain autonomous Hamiltonian flows on a given $2n$-dimensional Liouville domain.  The bounds underlying our main results are proven using filtered equivariant symplectic homology via methods from \cite{GU}. 
\end{abstract}
\tableofcontents
\section{Introduction}

The question of when one region in $\mathbb{C}^{n}$ symplectically embeds into another has motivated a great deal of work in symplectic topology, beginning in earnest with the non-squeezing theorem from \cite{Gr}.  The development of a wide variety of symplectic capacities, such as those surveyed in \cite{qsg}, has made it possible to obstruct many putative symplectic embeddings, while a broad array of both constructive and indirect methods have emerged to show that some quite non-obvious symplectic embeddings do exist.  

Given two open subsets $U,V\subset\mathbb{C}^{n}$, one can refine somewhat the question of whether $U$ symplectically embeds into $V$ by considering the real vector space structure of $\mathbb{C}^{n}$ and asking for what values $a\in\mathbb{R}_+$ it is the case that $U$ symplectically embeds into $a V$.  Following the standard convention in symplectic geometry that $a V=\{\sqrt{a}x|x\in V\}$, if $c$ is any symplectic capacity on open subsets of $\mathbb{C}^{n}$ a symplectic embedding of $U$ into $a V$ will imply that $c(U)\leq c(a V)=a c(V)$, and so a computation of $c(U)$ and $c(V)$ gives a lower bound on the infimal such $a$.

Motivated by the Banach-Mazur distance between convex bodies in $\mathbb{R}^n$, Ostrover and Polterovich  have proposed (see \cite{pvid},\cite{ptalk}) a notion of distance between star-shaped subsets of $\mathbb{C}^n$, loosely based on the idea that $U$ and $V$ should be considered close to each other if, for some $a$ close to $1$, there are symplectic embeddings both from $U$ into $a^{1/2} V$ and from $V$ into $a^{1/2} U$, perhaps satisfying compatibility requirements on the resulting compositions $U\hookrightarrow aU$ and $V\hookrightarrow a V$.  This paper will study such distances; one of our main theorems will imply that it in fact matters a great deal whether one imposes such compatibility requirements, affecting not just the precise values of the distance but even the topology that the distance defines.  Evidently any capacity will give rise to a lower bound for such a distance, but building on ideas from \cite{GU} we will see that one can often derive stronger lower bounds using information from the persistence module structure of equivariant symplectic homology that goes beyond capacities.

Although our main results concern star-shaped open subsets of $\C^n$, both the general problem and our methods are more naturally phrased in terms of a more general framework which we introduce now.  Let $(U,\lambda)$ be an exact symplectic manifold, \emph{i.e.} $U$ is a smooth manifold (without boundary) and $\lambda$ is a $1$-form on $U$ such that $d\lambda$ is nondegenerate. The choice of primitive $\lambda$ for the symplectic form $d\lambda$ on $U$ determines a Liouville vector field $\mathcal{L}_{\lambda}$ (usually just denoted $\mathcal{L}$ when $\lambda$ can be understood from context) via the prescription that $d\lambda(\mathcal{L}_{\lambda},\cdot)=\lambda$.  By the Cartan formula, if the time-$t$ flow $\mathcal{L}^{t}_{\lambda}\co U\to U$ of $\mathcal{L}_{\lambda}$ exists, then it will satisfy $\mathcal{L}^{t*}_{\lambda}\lambda=e^t\lambda$.  

\begin{dfn}\label{opendef}
An \textbf{open Liouville domain} is an exact symplectic manifold $(U,\lambda)$ with the properties that:
\begin{itemize} \item[(i)] For all $t\leq 0$, the time-$t$ flow $\mathcal{L}_{\lambda}^{t}\co U\to U$ of the Liouville field $\mathcal{L}_{\lambda}$ exists.
\item[(ii)] For all $t<0$, the image $\mathcal{L}_{\lambda}^{t}(U)$ of $U$ under the time-$t$ Liouville flow has compact closure inside of $U$. 
\end{itemize} In this case, if $0<a\leq 1$ we write \[ a U=\mathcal{L}^{\log a}_{\lambda}(U).\]  
\end{dfn}

\begin{ex}
The standard way of making any open subset of $\C^n=\{(x_1+iy_1,\ldots,x_n+iy_n)|x_j,y_j\in\R\}$ into an exact symplectic manifold is by using the one-form \[ \lambda_0=\frac{1}{2}\sum_{j=1}^{n}(x_jdy_j-y_jdx_j).\]  The corresponding Liouville vector field is then $\sum_j\mathcal{L}_{\lambda_0}=\frac{1}{2}(x_j\partial_{x_j}+y_j\partial_{y_j})$ and so for an open $U\subset\C^n$ the  definition of $a U$ in Definition \ref{opendef} coincides with the standard definition $a U=\{\sqrt{a}x|x\in U\}$.  Moreover $(U,\lambda_0)$ is an open Liouville domain if and only if $U$ is bounded with $\overline{aU}\subset U$ for all $a<1$.
\end{ex}

Because our definition of an open Liouville domain $(U,\lambda)$ is intrinsic, we obtain ``scalings'' $a U$ only for $a\leq 1$.  If $(U,\lambda)$ happens to be an open subset of an exact symplectic manifold $(\tilde{U},\lambda)$ with a complete Liouville flow\footnote{In fact such a $\tilde{U}$ can always be constructed---for any $b<1$ Proposition \ref{sstame} provides a Liouville domain $(W_{b},\lambda)$ with $b U\subset W_{b}\subset U$ and one could take $\tilde{U}$ equal to the Liouville completion of $W_{b}$---but the choice of this manifold is not entirely canonical and we do not use it.} then we could use the positive-time flow of the Liouville vector field on $\tilde{U}$ to make sense of $a U$ for $a>1$.  However we prefer to work inside $U$ itself; in particular instead of looking for embeddings $U\hookrightarrow a^{1/2}V$ when $a>1$ we will look for embeddings $a^{-1/2}U\hookrightarrow V$.

A \textbf{Liouville embedding} from an exact symplectic manifold $(U,\lambda)$ to another exact symplectic manifold $(V,\mu)$ is by definition an embedding $\phi\co U\hookrightarrow V$ such that $\phi^*\mu-\lambda$ is exact.  Thus Liouville embeddings are symplectic embeddings, and in the case that $H^1(U;\R)=\{0\}$ (in particular, in the case of star-shaped domains in $\C^n$) the two notions coincide.  Throughout the paper we use the symbol $\xhookrightarrow{L}$ to denote a Liouville embedding.

With this preparation we can now define the analogues of the Banach-Mazur distance that will be considered in this paper.  For context, 
recall (see, \emph{e.g.}, \cite{Rud})
that the \emph{Banach-Mazur distance} between two convex bodies $K,L\subset \R^n$ is defined to be \begin{equation}\label{bmdef} d_{BM}(K,L)=\inf\{a\geq 1|(\exists h\in GL(n), v,w\in \R^n)(a^{-1}(K+v)\subset h(L+w)\subset K+v)\}.\end{equation}   Thus $\log d_{BM}$ is a pseudometric on the set of convex bodies in $\mathbb{R}^n$, with $K$ and $L$ lying at distance zero away from each other iff they are equivalent under an affine transformation.

\begin{dfn}\label{maindef} Fix an even dimension $2n$ and let $\mathcal{OL}_{2n}$ denote the class of open Liouville domains of dimension $2n$.  We define the \textbf{coarse symplectic Banach-Mazur distance} $d_c$, the \textbf{fine symplectic Banach-Mazur hemidistance} $\delta_f$, and the \textbf{fine symplectic Banach-Mazur distance} $d_f$ on $\mathcal{OL}_{2n}$ by setting, for $(U,\lambda),(V,\mu)\in\mathcal{OL}_{2n}$:
\begin{itemize} \item[(i)]
\[ d_{c}\left((U,\lambda),(V,\mu)\right)=\inf\{a\geq 1|\exists \mbox{ embeddings }a^{-1/2}U\xhookrightarrow{L} V,\,a^{-1/2}V\xhookrightarrow{L} U\} \]
\item[(ii)] \[ \delta_f\left((U,\lambda),(V,\mu)\right) = \inf\left\{a\geq 1\left|\begin{array}{c}(\exists\mbox{ embedding } h\co a^{-1/2}U\xhookrightarrow{L} V)\\ (a^{-1}V\subset h(a^{-1/2}U)\subset V)\end{array}\right.\right\}.\]
\item[(iii)] \[ d_f\left((U,\lambda),(V,\mu)\right) = \max\left\{\delta_f\left((U,\lambda),(V,\mu)\right),\delta_f\left((V,\mu),(U,\lambda)\right)\right\} \]
\end{itemize}
\end{dfn}
Although the definition of $\delta_f$ bears the closest resemblance to the definition of the usual Banach-Mazur distance $d_{BM}$, we will see that unlike $d_{BM}$ (and also unlike $d_c$) $\delta_f$ is not symmetric (hence the term hemidistance rather than distance), necessitating the introduction of $d_f$ to obtain something that acts more like a distance.  Though it is not obvious, it does turn out to be true (see Proposition \ref{triangle}) that $\delta_f$ obeys a multiplicative triangle inequality $\delta_f((U,\lambda),(W,\nu))\leq \delta_f((U,\lambda),(V,\mu))\delta_f((V,\mu),(W,\nu))$; likewise $d_c$ and $d_f$ obey such inequalities.   Throughout the paper we conform to the convention standard in the Banach-Mazur distance literature of describing quantities like $d_{BM},d_c,d_f$ as distances when it is in fact their logarithms that are  pseudometrics in the usual sense.

Embedded within the class $\mathcal{OL}_{2n}$ is the set $\mathcal{S}_{2n}$ of bounded open subsets $U\subset \C^n$ that are strongly star-shaped with respect to the origin in the sense that for all $a<1$ it holds that $\overline{aU}\subset U$, equipped with the standard primitive $\lambda_0=\frac{1}{2}\sum_{j=1}^{n}(x_jdy_j-y_jdx_j)$.  Let $\mathcal{EL}_{2n}\subset \mathcal{S}_{2n}$ denote the subset consisting of interiors $E^{\circ}$ of ellipsoids \[E= E(a_1,\ldots,a_n)=\left\{(z_1,\ldots,z_n)\in\C^n\left|\sum_{j=1}^{n}\frac{\pi|z_j|^2}{a_j}\leq 1\right.\right\}.\]  Our two main results give examples of families of elements of strongly star-shaped open subsets of $\mathbb{C}^n$ that behave in interesting ways with respect to the symplectic Banach-Mazur distances. The first of these is the following:

\begin{theorem}\label{elldist} Let $2n\geq 4$ and fix any open ellipsoid $U\in\mathcal{EL}_{2n}$.  Then there is a sequence $\{U_m\}_{m=1}^{\infty}$ in $\mathcal{S}_{2n}$ such that $\delta_f((U_m,\lambda_0),(U,\lambda_0))\to 1$ but \[ \inf_{V\in\mathcal{EL}_{2n}}\delta_f((V,\lambda_0),(U_m,\lambda_0))\to\infty.\]   
\end{theorem}

Theorem \ref{elldist} is proven in Section \ref{trunc}; see Section \ref{examples} for an outline of the argument.
This result confirms in a rather strong way that $\delta_f$ is not symmetric. (This fact can also be inferred from some examples in \cite{GU}, though the failure of symmetry there is less dramatic.)  It is easy to check (see Proposition \ref{easyineq}) that one has  inequalities \begin{equation}\label{obvineq} d_c((U,\lambda),(V,\mu))\leq \min\{\delta_f((U,\lambda),(V,\mu)),\delta_f((V,\mu),(U,\lambda))\}\leq d_f((U,\lambda),(V,\mu));\end{equation} the asymmetry of $\delta_f$ can be phrased as saying that the second inequality here is sometimes strict.  In view of the first inequality, Theorem \ref{elldist} shows that any open ellipsoid $E^{\circ}$ arises as the limit with respect to the topology induced by the coarse distance $d_c$ of a sequence in $\mathcal{S}_{2n}$ which, with respect to fine distance, diverges arbitrarily far from $E^{\circ}$ (and indeed from the entire set of ellipsoids). In particular the pseudometrics $\log d_c,\log d_f$ induce different topologies on $\mathcal{S}_{2n}$, and are not quasi-isometric.

We also use the constructions of Section \ref{trunc} to show in Corollary \ref{v34} that the first inequality in (\ref{obvineq}) is sometimes strict.

Our second main result shows that the fine symplectic Banach-Mazur distance makes the space of star-shaped domains in $\R^{2n}$ into a large space from a coarse geometric viewpoint.  A similar result in the context of symplectic Banach-Mazur distances between subsets of cotangent bundles of surfaces has recently been established in \cite[Theorem 1.10]{SZ} via rather different constructions.

\begin{theorem}\label{quasiembed}
Let $2n\geq 4$, $D\in \N$, and write $\triangle_D=\{(x_1,\ldots,x_D)\in[0,\infty)^D|x_1\geq x_2\geq\cdots\geq x_D\}$.  Then there is a map $\mathcal{G}\co \triangle_D\to\mathcal{S}_{2n}$ such that, for $\vec{x},\vec{y}\in \triangle_D$, \begin{equation}\label{quasiineq}\|\vec{x}-\vec{y}\|_{\infty}\leq \log d_f\left((\mathcal{G}(\vec{x}),\lambda_0),(\mathcal{G}(\vec{y}),\lambda_0)\right)\leq 2\|\vec{x}-\vec{y}\|_{\infty}.\end{equation}
\end{theorem}

We outline the proof of this theorem at the end of Section \ref{examples}, based on the constructions and arguments of Section \ref{sink}.

\begin{remark}
As is observed in \cite[Lemma 5.8]{SZ}, for all $D$ the Euclidean space $\R^D$ quasi-isometrically embeds into $\triangle_{2D}$.  Thus Theorem \ref{quasiembed} implies that $\mathbb{R}^D$ likewise quasi-isometrically embeds into $(\mathcal{S}_{2n},\log d_f)$.  
\end{remark}

Let us make a couple of comparisons to the standard Banach-Mazur distance $d_{BM}$ on convex bodies in $\R^{n}$. Rather in contrast to Theorem \ref{quasiembed}, $d_{BM}$ is bounded, for a combination of two reasons: first,  modulo translation,  by \cite[Theorem III]{J} any convex body $K\subset \R^n$ has an associated John ellipsoid $E_K$ which obeys\footnote{We are continuing to use the symplectic convention for scaling a domain; in the standard convex geometry notation the second inclusion would read $K\subset nE_K$.}   $E_K\subset K\subset n^2 E_K$;second, the group from which $h$ is drawn in (\ref{bmdef}) is $GL(n)$ (rather than, say, $O(n)$ or, if $n$ is even, $Sp(n)$), which acts transitively on the set of ellipsoids.   In our symplectic context, it is no longer true that all ellipsoids are equivalent---in particular they can often be distinguished by their volumes or their capacities, which makes it possible to quasi-isometrically embed the half-plane  $\{x_1\leq x_2\}\subset \mathbb{R}^2$ into the space $\mathcal{EL}_{2n}$ of ellipsoids with either pseudometric $\log d_c$ or $\log d_f$ provided that $n\geq 2$, for instance via the map $(x_1,x_2)\mapsto E(e^{x_1},e^{x_2},\ldots,e^{x_2})^{\circ}$.  The properties of the John ellipsoid (together with Williamson's theorem on standard forms for symplectic ellipsoids) show that any convex body lies at distance at most $(2n)^2$ from $\mathcal{EL}_{2n}$ with respect to either $d_c$ or $d_f$; it would be interesting to know if this upper bound  can be lowered.   Since elements $E(a_1,\ldots,a_n)^{\circ}$ of $\mathcal{EL}_{2n}$ are Lipschitz-parametrized by just $n$ parameters, Theorem \ref{quasiembed} shows that the space $\mathcal{S}_{2n}$ of strongly star-shaped domains is much larger, coarse-geometrically speaking, than its subset consisting of open convex domains, at least with respect to the fine distance $d_f$.

We do not know whether an analogue of Theorem \ref{quasiembed} holds for the coarse distance $d_c$.  \cite[Example 8]{qsg} (due to D. Hermann) does imply that there are star-shaped domains that lie arbitrarily far from $\mathcal{EL}_{2n}$ with respect to $d_c$: indeed on $\mathcal{EL}_{2n}$ the Gromov width $c_B$ and the cylindrical capacity $c^Z$ coincide, so an upper bound on $\inf_{V\in\mathcal{EL}_{2n}}d_c((V,\lambda_0),(U,\lambda_0))$ for general $U\in \mathcal{S}_{2n}$ would give an upper bound on $\frac{c^Z(U)}{c_B(U)}$, which by Hermann's example does not exist.  Evidently Theorem \ref{elldist} shows that there are ways of going arbitrarily far from $\mathcal{EL}_{2n}$ in $\mathcal{S}_{2n}$ with respect to the fine distance $d_f$ that behave differently than Hermann's example.

\subsection{Open Liouville domains versus Liouville domains}  Some readers might prefer to work with Liouville domains in the usual sense, instead of what we are calling open Liouville domains; let us say a bit to relate these notions and motivate our choice. Recall that a Liouville domain is typically defined to be a compact exact symplectic manifold with boundary whose Liouville field points outward along the boundary.  Thus one class of examples of what we are calling open Liouville domains consists of interiors of Liouville domains.  We prefer to work with open sets instead of compact sets mainly in order to avoid worrying about smoothness of boundaries.  For example with our definition a product of open Liouville domains is clearly an open Liouville domain, but the analogous statement does not apply to Liouville domains with the standard definition because the product of two smooth manifolds with nonempty boundary is not canonically a smooth manifold with boundary but rather a manifold with corners.  

That said, Proposition \ref{sstame} implies that any open Liouville domain is arbitrarily well-approximated with respect to the fine distance $d_f$ by interiors of Liouville domains.  By passing to such approximations, it is straightforward to obtain versions of Theorems \ref{elldist} and \ref{quasiembed} for the obvious analogues of $d_c,\delta_f,d_f$ defined using Liouville domains instead of open Liouville domains.  Indeed  such approximations play a significant role in the way that these theorems are proved, cf. Lemmas \ref{exhaust} and \ref{betaorbits} and Proposition \ref{hepper}.

Corollary \ref{sympinv} implies that $\log d_c$ and $\log d_f$ descend to  extended (\emph{i.e.} valued in $[0,\infty]$) pseudometrics on the set of equivalence classes of open Liouville domains modulo Liouville diffeomorphism.\footnote{By definition, a ``Liouville diffeomorphism'' is simply a Liouville embedding that is also a diffeomorphism.}  It would be interesting to know whether one or both of these pseudometrics are nondegenerate when restricted to the equivalence classes of elements of $\mathcal{S}_{2n}$, \emph{i.e.} whether two strongly star-shaped domains $U$, $V$ having $d_c((U,\lambda_0),(V,\lambda_0))=1$ and/or $d_f((U,\lambda_0),(V,\lambda_0))=1$ are necessarily symplectomorphic. The analogous question for (compact) Liouville domains is readily seen to have a negative answer in view of the construction in \cite{EH} of non-symplectomorphic convex domains having symplectomorphic interiors.  It would also be interesting to know whether $\log    d_f$ is complete.  If we had defined $d_f$ on Liouville domains instead of open Liouville domains it certainly would not be complete, due to the existence of sequences of Liouville domains that approximate sets with non-smooth boundary. 

\begin{remark}
Our main results concern the collection $\mathcal{S}_{2n}$ of strongly star-shaped open sets in $\mathbb{R}^{2n}$ for $2n\geq 4$; the case of $\mathcal{S}_2$ is rather different.  Indeed any two elements of $\mathcal{S}_{2}$ having the same area are symplectomorphic: this can be seen by combining the fact that they are diffeomorphic (this is a nontrivial folk theorem; a proof appears in \cite[Theorem 2.6]{nlab}) with the extension of the Moser method to the noncompact setting in \cite{GS}.  Thus the quotient of $\mathcal{S}_2$ by symplectomorphisms can be identified with the space of open disks around the origin, and either $d_c$ or $d_f$ just measures (the square of) the ratio of the areas. In particular one cannot quasi-isometrically embed $\triangle_D$ into $(\mathcal{S}_2,d_f)$ unless $D=1$. 
\end{remark}

\subsection{$\delta_f$ and unknottedness}

Essentially following \cite{GU}, if $(M,\lambda)$ is an exact symplectic manifold and $U\subset M$ is an open subset, we say that a Liouville embedding $\psi\co U\xhookrightarrow{L}M$ is \emph{unknotted} provided that there is a Liouville diffeomorphism $\Psi\co M\to M$ such that $\Psi(U)=\psi(U)$.  Then the hemidistance $\delta_f$ can be equivalently defined as \begin{equation}\label{du} \delta_f\left((U,\lambda),(V,\mu)\right)=\inf\left\{a\geq 1\left|\begin{array}{c}(\exists f\co a^{-1}V\xhookrightarrow{L} a^{-1/2}U,\,g\co a^{-1/2}U\xhookrightarrow{L}V)\\ (g\circ f\co a^{-1}V\to V\mbox{ is unknotted})\end{array}\right.\right\}.\end{equation} Indeed if  $f\co a^{-1}V\xhookrightarrow{L} a^{-1/2}U$ and $g\co a^{-1/2}U\xhookrightarrow{L}V$ with $g\circ f(a^{-1}V)=\Psi(a^{-1}V)$ for a Liouville diffeomorphism $\Psi\co V\to V$, then $h=\Psi^{-1}\circ g$ has $a^{-1}V\subset h(a^{-1/2}U)\subset V$.  Conversely if $a^{-1}V\subset h(a^{-1/2}U)\subset V$ then $f=h^{-1}|_{a^{-1}V}$ is a Liouville embedding from $a^{-1}V$ to $a^{-1/2}U$ such that $h\circ f$ is the inclusion of $a^{-1}V$ into $V$ and hence is unknotted.

Clearly the coarse distance $d_c$ can be defined by a formula that differs from (\ref{du}) only in not requiring $g\circ f$ to be unknotted.  In this language, \cite{GU} finds many knotted (\emph{i.e.} not unknotted) Liouville (equivalently symplectic) embeddings between various star-shaped domains in $\mathbb{R}^4$ by showing that $d_c$ is strictly smaller than $\delta_f$ in many examples.  Evidently Theorem \ref{elldist} provides further
such examples.  Also Corollary \ref{v34} yields examples of embeddings $a^{-1/2}U\xhookrightarrow{L} V$ and $a^{-1/2}V\xhookrightarrow{L} U$ such that \emph{both} resulting compositions $a^{-1}U\xhookrightarrow{L} U$ and $a^{-1}V\xhookrightarrow{L} V$ are knotted. 

The strategy in \cite{GU} is to consider quantities $\delta_{\mathrm{ell}}(U)$ and $\delta_{\mathrm{ell}}^{\mathrm{u}}(U)$ associated to any $U\in\mathcal{S}_{2n}$ that can be written\footnote{Strictly speaking \cite{GU} sometimes uses closed sets in place of open sets here, but in view of Proposition \ref{sstame} this does not affect the conclusions.} in our notation as \[ \delta_{\mathrm{ell}}(U)=\inf_{V\in\mathcal{EL}_{2n}}d_c((V,\lambda_0),(U,\lambda_0)) \qquad \delta_{\mathrm{ell}}^{\mathrm{u}}(U)=\inf_{V\in\mathcal{EL}_{2n}}\delta_f((V,\lambda_0),(U,\lambda_0)).\] Inequalities $\delta_{\mathrm{ell}}(U)<a<\delta_{\mathrm{ell}}^{\mathrm{u}}(U)$ imply the existence of a knotted embedding $U\xhookrightarrow{L} aU$.  All of the examples in \cite{GU} had $\delta_{\mathrm{ell}}^{\mathrm{u}}(U)\leq 2$, which led to \cite[Question 1.9]{GU} asking whether one could ever have a knotted embedding $U\xhookrightarrow{L} aU$ with $a>2$.  Since Theorem \ref{elldist} produces examples with $\delta_{\mathrm{ell}}^{\mathrm{u}}$ arbitrarily large and $\delta_{\mathrm{ell}}$ small, it implies an affirmative answer to this question.

\subsection{Sketch of the proofs of the main theorems}\label{examples}


Throughout this section we shift the value of $n$ by $1$ with respect to the statements of Theorems \ref{elldist} and \ref{quasiembed}, so that our domains are subsets of $\mathbb{C}^{n+1}$ rather than $\mathbb{C}^n$.

In the proof of Theorem \ref{elldist}, carried out in Section \ref{trunc}, we associate to an open ellipsoid $\hat{E}^{\circ}=E(a_1,\ldots,a_{n+1})^{\circ}$ a family of open Liouville domains $E_{H_{\ep,\beta}}^{\circ}$.  When $\beta$ is large these are close to $\hat{E}^{\circ}$ with respect to the coarse distance $d_c$, but, if $\ep$ is small and the large number $\beta$ is chosen carefully,  $E_{H_{\ep,\beta}}^{\circ}$ is far from $\hat{E}^{\circ}$ (and indeed far from all ellipsoids) with respect to the fine distance.  We call the domains $E_{H_{\ep,\beta}}^{\circ}$ ``truncated ellipsoids.''  To describe them, note that we have \[ \hat{E}^{\circ}=\left\{(z_1,\ldots,z_{n+1})\left|\frac{\pi|z_{n+1}|^2}{a_{n+1}}<1-u\right.\right\}\quad\mbox{where }u=\sum_{j=1}^{n}\frac{\pi |z_{j}|^2}{a_j}.\]  Then $E_{H_{\ep,\beta}}^{\circ}$ is defined by replacing the expression $1-u$ on the right-hand side of the above inequality by $\min\{\ep+\beta u,1-u\}$.  See Figure \ref{truncfig} for the case $a_1=\cdots =a_{n+1}=1$.  Thus if $\beta$ is large, $E_{H_{\ep,\beta}}^{\circ}$  coincides with the original ellipsoid $\hat{E}^{\circ}$ everywhere that $|z_1|,\ldots,|z_n|$ are not all small, but a large proportion of the intersection of $\hat{E}^{\circ}$ with a small neighborhood of $\{\vec{0}\}\times \C\subset \C^{n+1}$ has been cut away to obtain $E_{H_{\ep,\beta}}^{\circ}$.

Roughly because a large part of the region in Figure \ref{truncfig} is filled by a right triangle, it is possible to adapt an argument from \cite[Section 5]{T} to show that $\alpha\hat{E}^{\circ}$ symplectically embeds in $E_{H_{\ep,\beta}}^{\circ}$  for a value of $\alpha$ that is not much smaller than one; this is done in Lemma \ref{translate}, leading to Corollary \ref{coarsecvg} which shows that $\delta_f\left((E_{H_{\ep,\beta}}^{\circ},\lambda_0),(\hat{E}^{\circ},\lambda_0)\right)<(1+1/\beta)^2$.  

\begin{figure}
\includegraphics[height=3.2 in]{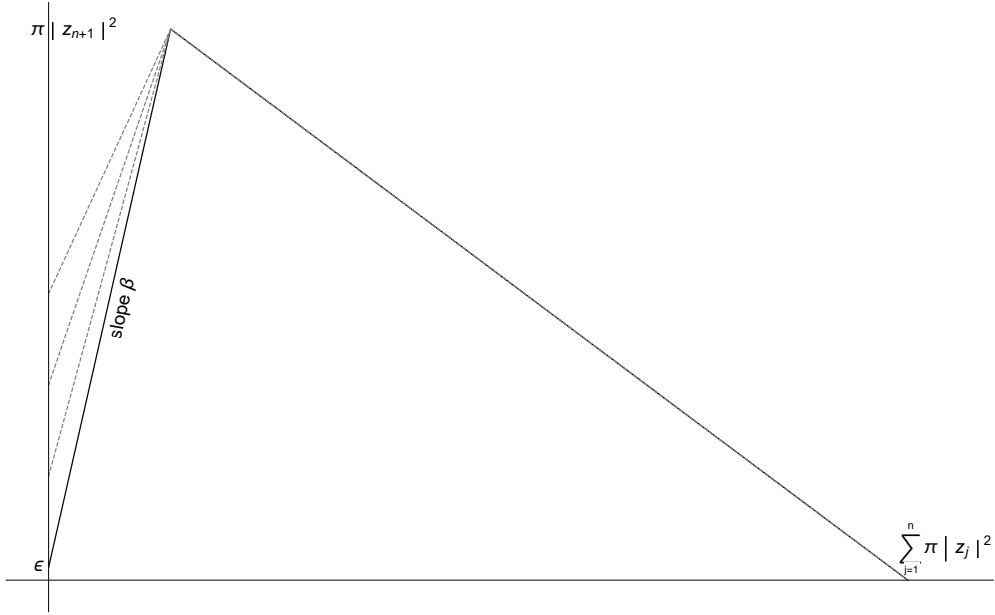}
\caption{The image under a map to $\R^2$ of the boundary of the domain $E_{H_{\ep,\beta}}$ in the case that $a_1=\cdots=a_{n+1}=1$, $\beta\approx 5.9$, and $\ep=0.02$.  The dashed lines have integer slope, and their intersections with the vertical axis correspond to periods of Reeb orbits that wind once around $\{z_{n+1}=0\}$. These periods are large compared to $\epsilon$, which is the period of the orbit where $z_1=\cdots=z_n=0$.}
\label{truncfig}
\end{figure}

The proof that $\delta_f\left((\hat{E}^{\circ},\lambda_0),(E_{H_{\ep,\beta}}^{\circ},\lambda_0)\right)$ can be arranged to be large is based on filtered equivariant symplectic homology, using an analysis of the periods and Conley--Zehnder indices of the closed Reeb orbits on the boundary of (smooth Liouville approximations to) $E_{H_{\ep,\beta}}^{\circ}$.  The key is to find a closed Reeb orbit $\gamma_0$ of index $k\leq 1$ such that every closed Reeb orbit $\gamma$ having index $k\pm 1$ has period which is a large multiple of the period of $\gamma_0$.  The orbit $\gamma_0$ that we use is the one contained in the $z_{n+1}$-plane, whose period is the parameter $\ep$ that we are taking very small.  We then establish (for suitably chosen large $\beta$) that Reeb orbits of index $k\pm 1$ have periods bounded below by $M_{\beta}\ep$ where $M_{\beta}\gg 1$, and this leads to a large lower bound on $\delta_f\left((\hat{E}^{\circ},\lambda_0),(E_{H_{\ep,\beta}}^{\circ},\lambda_0)\right)$.

Part of the basis for this lower bound on periods, at least in the case that $a_1=\cdots=a_{n+1}=1$, is suggested in Figure \ref{truncfig}.  The relevant orbits have positive winding numbers around the hyperplane $\{z_{n+1}=0\}$.  Those orbits with winding number one have period equal to the vertical intercept of a line with \emph{integer} slope $m<\beta$ that passes through the  vertex of the graph in Figure \ref{truncfig}. If $\beta$ has been chosen to be slightly less than an integer, then these intercepts will be bounded below by a multiple of $\frac{1}{\beta}$ independently of $\ep$.  More generally, one must consider orbits of arbitrary winding number $N$, which have periods equal to $N$ \emph{times} the intercept of a line with slope $\frac{p}{N}<\beta$; a more careful analysis  (a more general version of which is done in the proof of Lemma \ref{strongTbound}) shows that these periods are  bounded below by a multiple of $\beta\ep$ provided that $\beta$ is slightly less than an integer and $\ep<\beta^{-2}$. 

In the case that the $a_j$ are not all equal the argument is somewhat more delicate, as it requires $\beta$ to be chosen to be \emph{simultaneously} slightly smaller than integer multiples of all of the values $\frac{a_j}{a_{n+1}}$ for $j=1,\ldots,n$.  However by using Dirichlet's theorem on simultaneous Diophantine approximations we are able to find an unbounded open set of values $\beta$ for which  the required type of period bound holds, see Lemma \ref{strongTbound}. From this we complete the proof of Theorem \ref{elldist} using properties of filtered equivariant symplectic homology. More specifically,\footnote{These results require the assumption that each $\frac{a_{n+1}}{a_j}\notin \mathbb{Q}$, but continuity considerations show that it is sufficient to prove the theorem in this case.} Proposition \ref{eme} and Corollary \ref{dellu} show that there is an unbounded open set $B\subset (1,\infty)$ such that if $\beta$ is chosen from $B$ and has each $\frac{\beta a_n}{a_j}\notin\Q$, and if $0<\ep<\beta^{-2}$, then $\delta_f((V,\lambda_0),(E_{H_{\ep,\beta}}^{\circ},\lambda_0))$ is bounded below by a positive constant times a positive power of $\beta$ for every ellipsoid $V$ (with the constants independent of the choices of $V,\ep$).  So for the sequence $U_m$ whose existence is asserted by Theorem \ref{elldist} we may take the domains $E_{H_{\beta^{-3}_{m},\beta_m}}^{\circ}$ where $\beta_m$ is a sequence in $B$ such that $\beta_m\to\infty$ and each $\frac{\beta_ma_n}{a_j}$ is irrational.

The examples giving rise to the quasi-isometric embedding of $\Delta_D$ from Theorem \ref{quasiembed} are constructed in Section \ref{sink} by arranging for a somewhat similar picture to the one just described (with $a_1=\cdots=a_{n+1}=1$ and $\beta$ slightly less than $2$) to occur in $D$ distinct locations (``sinkholes'') in the domain.  The resulting domains in $\C^{n+1}$ are not star-shaped, but we adapt a Moser-type argument from \cite{CE} to show that they are symplectomorphic to star-shaped domains.  Associated to each of the sinkholes is a depth parameter $\ep_j$ ($j=1,\ldots,D$), and we associate to $(x_1,\ldots,x_D)\in \Delta_D$ the domain with depth parameters $\ep_j=\frac{1}{2}e^{-x_j}$.  Filtered equivariant symplectic homology is used to give lower bounds on the fine symplectic Banach-Mazur distance between the resulting domains in terms of the $\ell^{\infty}$ distance on $\Delta_D$.  This is somewhat similar to---and partially influenced by---the approach used in \cite{SZ} in the context of cotangent bundles of surfaces, though the constructions of the domains and the ingredients in the proofs of the necessary constraints on periods of Reeb orbits are different in the two contexts.  

To outline the proof of Theorem \ref{quasiembed} in more detail, in Section \ref{sink} we associate to each $\vec{\ep}=(\ep_1,\ldots,\ep_D)\in (0,2]^D$ a $(2n+2)$-dimensional open Liouville domain $(W_{H_{\vec{\ep}}}^{\circ},\hat{\lambda})$ where $W_{H_{\vec{\ep}}}^{\circ}\subset \C^{n+1}$ and $d\hat{\lambda}$ is the standard symplectic form on $\mathbb{C}^{n+1}$.   Corollary \ref{quasicor} gives lower bounds for the fine Banach-Mazur distance restricted to those $W_{H_{\vec{\ep}}}^{\circ}$ with $0<\ep_1\leq\cdots\leq\ep_D\leq\frac{1}{2}$.  Although $(W_{H_{\vec{\ep}}}^{\circ},\hat{\lambda})$ does not belong to $\mathcal{S}_{2n+2}$ (since $\hat{\lambda}$ differs from $\lambda_0$),  Corollary \ref{makess} shows that there is a symplectomorphism $F_{\vec{\ep}}\co \C^{n+1}\to \C^{n+1}$ such that $F_{\vec{\ep}}(W_{H_{\vec{\ep}}}^{\circ})\in \mathcal{S}_{2n+2}$. 

Given $\vec{x}=(x_1,\ldots,x_D)\in \triangle_D$ define $\vec{\ep}(\vec{x})=\left(\frac{1}{2}e^{-x_1},\ldots,\frac{1}{2}e^{-x_D}\right)$.  Thus the condition that $\vec{x}\in\triangle_D$ translates to the condition that the coordinates $\ep_1,\ldots,\ep_D$ of $\vec{\ep}(\vec{x})$ obey $0<\ep_1\leq\cdots\leq \ep_D\leq\frac{1}{2}$.  The map $\mathcal{G}$ promised in the theorem is then given by \[ \mathcal{G}(\vec{x})=F_{\vec{\ep}(\vec{x})}(W_{H_{\vec{\ep}(\vec{x})}}^{\circ}).\]  To prove the desired inequalities first note that Corollary \ref{sympinv} shows that \[ d_f\left((\mathcal{G}(\vec{x}),\lambda_0),(\mathcal{G}(\vec{y}),\lambda_0)\right) = d_f\left((W_{H_{\vec{\ep}(\vec{x})}}^{\circ},\hat{\lambda}),(W_{H_{\vec{\ep}(\vec{y})}}^{\circ},\hat{\lambda})\right).\]  Then Corollary \ref{uppersink} implies the second inequality in (\ref{quasiineq}) and Corollary \ref{quasicor} (together with the fact that $\frac{1}{\ep_m}>\frac{\zeta_m}{\ep_m}$ when $0<\ep_m,\zeta_m\leq \frac{1}{2}$) implies the first inequality.

\subsection{Organization of the paper}
The upcoming Section \ref{genl} establishes a number of mostly elementary---though in some cases not so straightforward---results about open Liouville domains and applies these to establish basic results about our Banach-Mazur-type distances.  As has already been mentioned, Proposition \ref{sstame} plays a key role throughout the paper in that it shows that any open Liouville domain is exhausted by (compact) Liouville subdomains, allowing us to apply results about Liouville domains in the usual sense to open Liouville domains.   Proposition \ref{triangle} establishes multiplicative triangle inequalities for $d_c,\delta_f,$ and $d_f$.  For $d_c$ this is straightforward once one notices (see Lemma \ref{rescale}) that a Liouville embedding $\phi\co U\xhookrightarrow{s} V$ gives rise, for any $a\geq 1$, to a Liouville embedding $\phi_a\co a^{-1}U\xhookrightarrow{s} a^{-1}V$ by conjugating by the flows of the respective Liouville vector fields.  That $\delta_f$ (and therefore also $d_f$) obeys a triangle inequality is more delicate, and depends on the non-obvious Proposition \ref{opencon}; roughly speaking the subtlety here in comparison to the classical Banach-Mazur distance $d_{BM}$ lies in the fact that Liouville embeddings do not commute with Liouville scalings, whereas the linear maps that are used to define $d_{BM}$ do of course commute with linear scalings.  Section \ref{genl} concludes with Lemma \ref{symplecto}, which is used to prove that the domains constructed in Section \ref{sink} are symplectomorphic to star-shaped domains, allowing them to be used in the proof of Theorem \ref{quasiembed}

Section \ref{CHSect} explains how filtered equivariant symplectic homology can be used to provide lower bounds for $\delta_f$ and hence for $d_f$, using results from \cite{G},\cite{GH},\cite{GU}.  To each open Liouville domain $(U,\lambda)$ this theory associates a persistence module consisting of $\mathbb{Q}[T]$-modules $CH^{L}(U,\lambda)$ as $L$ varies through $\R$; under topological hypotheses on $U$ these are naturally $\Z$-graded.  (In this paper just their structure as graded $\Q$-vector spaces is used.)  Moreover this assignment is contravariant with respect to Liouville embeddings.  We show in Proposition \ref{implantdelta} that an upper bound on $\delta_f$ implies the existence of what we call an \emph{implantation} between the corresponding persistence modules.  For readers familiar with persistence theory we remark that an implantation is a partial, asymmetric version of an interleaving; whereas an interleaving gives rise to an approximate matching between the corresponding barcodes, an implantation should be expected to give rise to an approximate injection.  However we do not use the theory of barcodes in the rest of the paper and instead work directly with our notion of implantations.  Lemma \ref{exhaust} gives some information about the structure of the persistence module associated to $(U,\lambda)$ under certain hypotheses that hold in our later examples; together with Proposition \ref{implantdelta} this becomes our main tool for bounding $\delta_f$ from below.

Section \ref{tubesect} describes a general mechanism for constructing Liouville domains of dimension $2n+2$ from certain autonomous Hamiltonian flows on Liouville domains of dimension $2n$.  We relate the Reeb dynamics on the boundary of the new domain $W_H$ to the discrete dynamics of the time-one map of the Hamiltonian flow, in particular deriving expressions for the periods (Corollary \ref{neworbits}) and Conley--Zehnder indices (Proposition \ref{czwn}) of the closed Reeb orbits on $\partial W_H$ in terms of lower-dimensional information.  The domains that serve as a basis for the proofs of both of our main theorems are constructed and analyzed  using this technique. This approach was motivated in part by the construction in \cite[Section 3]{ABHS} of contact forms on open books having monodromy equal to a prescribed area-preserving disk map.

Sections \ref{trunc} and \ref{sink} provide the proofs of Theorems \ref{elldist} and \ref{quasiembed}, respectively, using the examples that have already been introduced in Section \ref{examples}.  A variation on the main family of examples from Section \ref{trunc} is also used in Corollary \ref{v34} to show that the first inequality in (\ref{obvineq}) can be strict.

\subsection*{Acknowledgements}
I am very grateful to J. Gutt, R. Hind, Y. Ostrover, L. Polterovich, V. Stojisavljevi\'c, and J. Zhang for stimulating discussions, to L. Polterovich and J. Zhang for comments on a preliminary version, and to an anonymous referee for an exceptionally thorough and thoughtful report.  This work was partially supported by NSF grant DMS-1509213.

\section{(Open) Liouville domains} \label{genl}
This section will be concerned with some basic properties of Liouville domains and open Liouville domains, with implications for our Banach-Mazur type (hemi)distances $d_c,\delta_f,d_f$.  We begin by proving the simple inequalities (\ref{obvineq}) that appeared in the introduction. Recall that for an open Liouville domain $(U,\lambda)$ the Liouville flow $\mathcal{L}_{\lambda}^{t}\co U\to U$ is defined for all $t\leq 0$, and that when $a\geq 1$ we write $a^{-1}U$ for the image of $U$ under $\mathcal{L}_{\lambda}^{-\log a}$.   Evidently if $a>1$ then $\mathcal{L}_{\lambda}^{t}$ is well-defined on $a^{-1}U$ (as a map to $U$) for all $t\leq \log a$.

\begin{prop}\label{easyineq}
For any $2n$-dimensional open Liouville domains $(U,\lambda),(V,\mu)$ we have \[ d_c\left((U,\lambda),(V,\mu)\right)\leq \delta_f\left((U,\lambda),(V,\mu)\right)\leq d_f\left((U,\lambda),(V,\mu)\right).\]
\end{prop}

\begin{proof}
The second inequality holds by definition.  For the first, if $h\co a^{-1/2}U\xhookrightarrow{L} V$ has $a^{-1}V\subset h(a^{-1/2}U)$, then $h^{-1}$ gives a Liouville embedding of $a^{-1}V$ into $a^{-1/2}U$, and conjugating this embedding by the respective Liouville flows gives a Liouville embedding $\mathcal{L}_{\lambda}^{\log(a^{1/2})}\circ h^{-1}\circ \mathcal{L}_{\mu}^{-\log(a^{1/2})}\co a^{-1/2}V\xhookrightarrow{L} U$.  (Indeed, writing $h^{-1*}\lambda=\mu+df$, the map  $\mathcal{L}_{\lambda}^{\log(a^{1/2})}\circ h^{-1}\circ \mathcal{L}_{\mu}^{-\log(a^{1/2})}$ pulls $\lambda$ back to \[ \mathcal{L}_{\mu}^{-\log(a^{1/2})*}h^{-1*} (a^{1/2}\lambda)=\mathcal{L}_{\mu}^{-\log(a^{1/2})*} (a^{1/2}\mu+a^{1/2}df) \] which does indeed differ from $\mu$ by an exact one-form.)
\end{proof}
\begin{lemma}\label{rescale}
If $\phi\co (U,\lambda)\xhookrightarrow{L}(V,\mu)$ is a Liouville embedding between two open Liouville domains and $a\geq 1$ then there is a Liouville embedding $\phi_a\co (a^{-1}U,\lambda)\xhookrightarrow{L}(a^{-1}V,\mu)$.  Moreover $\phi_a$ can be chosen such that, for any $b>1$ such that $b^{-1}V\subset \phi(U)$, we have $a^{-1}b^{-1}V\subset \phi_a(a^{-1}U)$.
\end{lemma}

\begin{proof}
Let $g_U$, $g_V$ denote the time-$(\log a^{-1})$ flows of the Liouville vector fields of $U$ and $V$ respectively, so that $g_U\co U\to a^{-1}U$ is a diffeomorphism with $g_{U}^{*}\lambda=a^{-1}\lambda$ and similarly for $g_V$.  Then $\phi_a=g_{V}\circ\phi\circ g_{U}^{-1}$ is easily seen to satisfy the required properties.  
\end{proof}

\begin{cor}\label{infint} If $(U,\lambda),(V,\mu)$ are open Liouville domains and $a>\delta_f\left((U,\lambda),(V,\mu)\right)$ then there is a Liouville embedding $h\co a^{-1/2}U\xhookrightarrow{L} V$ with $a^{-1}V\subset h(a^{-1/2}U)\subset V$.
\end{cor}

\begin{proof}
The immediate implication of the definition of $\delta_f$ is that there is some $b<a$ and a Liouville embedding $\phi\co b^{-1/2}U\xhookrightarrow{L} V$ with $b^{-1}V\subset \phi(b^{-1/2}U)\subset V$. With notation as in Lemma \ref{rescale}, take $h=\phi_{\sqrt{a/b}}$, so $h$ is a Liouville embedding of $a^{-1/2}U$ into $(a/b)^{-1/2}V\subset V$ and we have\[ a^{-1} V\subset a^{-1/2}b^{-1/2} V\subset h(a^{-1/2}U)\subset (a/b)^{-1/2}V\subset V.\]
\end{proof}

Before we can prove some other properties of $d_c$ and $\delta_f$ (such as triangle inequalities and Liouville diffeomorphism invariance) we will need to consider a couple other types of domains.
Recall the standard definition of a Liouville domain (without the adjective ``open'') as a pair $(W,\lambda)$ where $W$ is a compact manifold with boundary and $\lambda\in\Omega^1(W)$ has the properties that $d\lambda$ is symplectic and the Liouville vector field $\mathcal{L}_{\lambda}$ points outward along $\partial W$.  We also consider the following notion, generalized from \cite{GU} in which a similar condition was imposed in order to have a suitable setting for considering filtered symplectic homology on certain open subsets of $\mathbb{R}^{2n}$.
\begin{dfn}\label{tamedef}
An exact symplectic manifold $(U,\lambda)$ is said to be \textbf{tamely exhausted} if for every compact subset $K\subset U$ there is a closed subset $X\subset U$ such that $(X,\lambda|_X)$ is a Liouville domain, $K\subset X$, and $X$ is a deformation retract of  $U$.
\end{dfn}

Note that there is no loss of generality in assuming that $K\subset X^{\circ}$ in Definition \ref{tamedef}, since one could always replace $X$ by $\mathcal{L}_{\lambda}^{\ep}(X)$ for some $\ep>0$ that is small enough for $\mathcal{L}_{\lambda}^{\ep}|_X$ to be defined.

The following fact will be crucial:

\begin{prop}\label{sstame}
Any open Liouville domain is tamely exhausted.
\end{prop}

\begin{proof} Let $(U,\lambda)$ be a $2n$-dimensional open Liouville domain with Liouville vector field $\mathcal{L}=\mathcal{L}_{\lambda}$.  While the flow of $\mathcal{L}$ does not exist globally on $U$ for any positive time, standard ODE existence results show that given any $x\in U$ there is a neighborhood $V_x$ of $x$ and a value $\ep_x>0$ such that for $0\leq \ep \leq \ep_x$ we have a well-defined flow $\mathcal{L}^{\ep}\co V_x\to U$, \emph{i.e.} such that $V_x\subset \mathcal{L}^{-\ep}(U)$.  So if $K\subset U$ is a compact subset, covering $K$ by finitely many such $V_x$ shows that $K\subset \mathcal{L}^{-\ep}(U)$ for some $\ep>0$.

We will show that for all $\ep>0$ there is a deformation retract $X\subset U$ such that $(X,\lambda|_X)$ is a Liouville domain and $\mathcal{L}^{-\ep}(U)\subset X$. In view of the previous paragraph this will  be sufficient to prove the proposition. So let $\ep>0$ be fixed.

If $x\in U$ let us consider the set $I_x=\{t\in [0,\infty)|x\in \mathcal{L}^{-t}(U)\}$.  Clearly if $t\in I_x$ and $0\leq s<t$ then also $s\in I_x$ because $\mathcal{L}^{-t}=\mathcal{L}^{-s}\circ \mathcal{L}^{s-t}$.  Also if $t\in I_x$, say $x=\mathcal{L}^{-t}(y)$, then letting $\ep_y>0$ be as in the first paragraph of the proof we see that $x=\mathcal{L}^{-(t+\ep_y)}(\mathcal{L}^{\ep_y}(y))$, so we have $t+\delta\in I_x$ for all sufficiently small $\delta>0$.  This suffices to show that $I_x$ is an interval of the form $[0,t_x)$ for some $t_x\in (0,\infty]$.  

Accordingly define $g\co U\to \R$ by \[ g(x)=\min\{\ep,t_x\}=\sup\{t\in (0,\ep]|x\in\mathcal{L}^{-t}(U)\}.\]
Naively we would like to set $X=g^{-1}([\frac{\ep}{2},\infty))$, which certainly contains $\mathcal{L}^{-\ep}(U)$, but this choice of $X$ typically will not be a smooth manifold. We work around this difficulty as follows.  

First,  we claim that the function $g$ is continuous.  For this purpose it suffices to show that, for any $a\in \R$, the set $g^{-1}((a,\infty))$ is open and the set $g^{-1}([a,\infty))$ is closed (as these two statements respectively imply the lower semicontinuity and the upper semicontinuity of $g$).  Now if $a\geq \ep$, then $g^{-1}((a,\infty))$ is empty and if $a<0$ then $g^{-1}((a,\infty))=U$. For the remaining case that $0\leq a<\ep$ we have \[ g^{-1}((a,\infty))=\{x\in U|t_x>a\}=\{x\in U|a\in I_x\}=\mathcal{L}^{-a}(U) \] since as noted earlier $I_x=[0,t_x)$.  So since the time-$(-a)$ flow $\mathcal{L}^{-a}\co U\to U$ is a local diffeomorphism $g^{-1}((a,\infty))$ is indeed open.

As for $g^{-1}([a,\infty))$, if $a>\ep$ this is the empty set and if $a\leq 0$ it is all of $U$, so in both of these cases it is closed in $U$. For the other cases we claim that \begin{equation}\label{upset} g^{-1}([a,\infty))=\overline{\mathcal{L}^{-a}(U)} \mbox{ for }0<a\leq \ep \end{equation} (here and elsewhere closures are taken relative to $U$).  To see this, observe that for $0<a\leq\ep$ we have $g^{-1}([a,\infty))=\{x\in U|t_x\geq a\}$, so (\ref{upset}) is equivalent to the statement that $t_x\geq a$ if and only if $x\in \overline{\mathcal{L}^{-a}(U)}$.  The forward implication is clear since if $t_x\geq a$ then we can find $s_n\nearrow a$ and $y_n\in U$  with $x=\mathcal{L}^{-s_n}(y_n)$, and then the sequence $\mathcal{L}^{-a}(y_n)=\mathcal{L}^{-(a-s_n)}(x)$ converges to $x$ since $s_n\nearrow a$.  As for the reverse implication, if $y_n\in U$ with $\mathcal{L}^{-a}(y_n)\to x$ and if $0<\delta<a$, then by the hypothesis that $\mathcal{L}^{-\delta}(U)$ has compact closure in $U$ we can pass to a subsequence (still denoted $\{y_n\}$) such that $\mathcal{L}^{-\delta}(y_n)$ converges in $U$, say $\mathcal{L}^{-\delta}(y_n)\to z$.  But then \[ \mathcal{L}^{-(a-\delta)}(z)=\lim_{n\to\infty}\mathcal{L}^{-(a-\delta)}(\mathcal{L}^{-\delta}(y_n))=\lim_{n\to\infty}\mathcal{L}^{-a}(y_n)=x,\] whence $t_x>a-\delta$.  This holds for all $\ep>0$, so $t_x\geq a$.  This completes the proof of (\ref{upset}), thus establishing the continuity of $g$.

We now consider the restriction of $g$ to $A:=g^{-1}([\frac{\ep}{4},\frac{3\ep}{4}])$, which is (by (\ref{upset})) a closed subset of the compact set $\overline{\mathcal{L}^{-\ep/4}(U)}$ and so is compact.  Our intention is to perturb this function to a smooth function $\tilde{g}$ whose level set at $\ep/2$ can serve as the boundary of our desired Liouville domain $X$.  Note that, everywhere on $A$, even though $g$ might not be differentiable we have an identity $g(\mathcal{L}^tx)-g(x)=-t$ for all sufficiently small $t$, so that the directional derivative $\mathcal{L}g$ of $g$ along $\mathcal{L}$ does exist and is equal to $-1$ throughout $A$. By the flow box theorem, we may cover $A$ by finitely many coordinate charts $\psi_\alpha\co V_{\alpha}\to I_{\alpha}\times W_{\alpha}$ where $I_{\alpha}\subset \R$ is an open interval, $W_{\alpha}\subset \R^{2n-1}$ is open, and $\psi_{\alpha*}\mathcal{L}=\frac{\partial}{\partial x_1}$.  In each such coordinate chart, the identity $g(\mathcal{L}^tx)-g(x)=-t$ implies that $g\circ\psi_{\alpha}^{-1}(x_1,\ldots,x_n)=-x_1+h_{\alpha}(x_2,\ldots,x_n)$ for some continuous function $h_{\alpha}\co W_{\alpha}\to\R$.  

Given $\delta>0$ let $g_{\alpha}^{\delta}\co V_{\alpha}\to \R$ be a function obeying $g_{\alpha}^{\delta}\circ \psi_{\alpha}^{-1}(x_1,\ldots,x_n)=-x_1+h_{\alpha}^{\delta}(x_2,\ldots,x_n)$ where $h_{\alpha}^{\delta}$ is smooth and $\|h_{\alpha}^{\delta}-h_{\alpha}\|_{L^{\infty}}<\delta$.  Thus $g_{\alpha}^{\delta}$ is smooth, $\|g_{\alpha}^{\delta}-g\|_{L^{\infty}(V_{\alpha})}<\delta$, and $g_{\alpha}^{\delta}-g$ has directional derivative zero along $\mathcal{L}$.  

Choose (independently of $\delta$) a partition of unity $\{\chi_{\alpha}\}$ subordinate to our finite cover $\{V_{\alpha}\}$ and let $ \tilde{g}^{\delta}=\sum_{\alpha}\chi_{\alpha}g_{\alpha}^{\delta}$.  Thus $\tilde{g}^{\delta}$ is a smooth function on a neighborhood of $A$, and everywhere on $A$ we compute the directional derivative \[ \mathcal{L}\tilde{g}^{\delta}=\mathcal{L}g+\sum_{\alpha}\mathcal{L}\left(\chi_{\alpha}(g_{\alpha}^{\delta}-g)\right) = -1+\sum_{\alpha}(\mathcal{L}\chi_{\alpha})(g_{\alpha}-g) \] since $\mathcal{L}g=-1$ and $\mathcal{L}(g_{\alpha}^{\delta}-g)=0$. If $\delta$ is chosen smaller than $\frac{1}{2\sum_{\alpha}\max_A|\mathcal{L}\chi_{\alpha}|}$ we will thus have \begin{equation}\label{Lgtilde}\mathcal{L}\tilde{g}^{\delta}<-\frac{1}{2} \mbox{ everywhere on }A=g^{-1}\left(\left[\frac{\ep}{4},\frac{3\ep}{4}\right]\right). \end{equation} 

For the rest of the proof let $\tilde{g}=\tilde{g}^{\delta}$ for some value of $\delta$ that has been chosen so that (\ref{Lgtilde}) holds and that also obeys $\delta<\frac{\ep}{8}$.  By construction $\|\tilde{g}-g\|_{L^{\infty}(A)}<\delta$.  Let \[ X=\left\{x\in A\left|\tilde{g}(x)\geq \frac{\ep}{2} \right.\right\}\cup \left\{x\in M\left|g(x)>\frac{5\ep}{8}\right.\right\}.\]  This is a closed subset of $U$ since if a convergent sequence $x_n\in X$ has $g(x_n)\to \frac{5\ep}{8}$ then its limit $x$ has $\tilde{g}(x)>\frac{5\ep}{8}-\delta>\frac{\ep}{2}$.  Also since $\|\tilde{g}-g\|_{L^{\infty}(A)}<\frac{\ep}{8}$ we have $X\subset g^{-1}\left(\left[\frac{3\ep}{8},\infty\right)\right)=\overline{\mathcal{L}^{-\frac{3\ep}{8}}(U)}$; thus $X$ is a closed subset of a compact space and is compact. The (topological) boundary of $X$ is evidently $\tilde{g}^{-1}(\{\ep/2\})$ (which is entirely contained in $A$ and hence in $X$ since $\|\tilde{g}-g\|_{L^{\infty}(A)}<\frac{\ep}{4}$).  Moreover the fact that $\mathcal{L}\tilde{g}<-\frac{1}{2}$ on $A$ and in particular on $\tilde{g}^{-1}(\{\ep/2\})$  implies that this boundary is a regular level set, and hence that $X$ is a submanifold with boundary of $U$, with the Liouville vector field $\mathcal{L}$ pointing outward along $\partial X$.  Also $X$ obviously contains $\mathcal{L}^{-\ep}(U)\subset g^{-1}(\{\ep\})$.

It remains to check that $X$ is a deformation retract of $U$.  For this purpose observe that the function $r\co U\to [0,\ep)$ defined by \[ r(x)=\inf\{t\geq 0|\mathcal{L}^{-t}(x)\in X\} \] is continuous (for essentially the same reason that our function $g$ from earlier in the proof is continuous; in particular we are using that $\mathcal{L}$ points outward along $\partial X$), vanishes identically on $X$ and has $\mathcal{L}^{-r(x)}(x)\in X$ for all $x\in U$.  So $(s,x)\mapsto \mathcal{L}^{-sr(x)}(x)$ gives a homotopy from the identity to a retraction $U\to X$.
\end{proof}

Proposition \ref{sstame} allows us to apply results about Liouville domains to obtain results about open Liouville domains.  For example the following will play a role in proving the multiplicative triangle inequality for $\delta_f$ (and hence also $d_f$).
 
\begin{lemma}\label{outin}
Let $(X,\lambda)$ be a Liouville domain and let $Y\subset X^{\circ}$ and $\mu\in \Omega^1(Y)$ be such that $(Y,\mu)$ is a Liouville domain and $\lambda|_Y-\mu$ is exact.  Choose $c>1$.  Then there is a Liouville diffeomorphism $\zeta\co X\to X$ such that  $\zeta(c^{-1}Y)\subset c^{-1}X^{\circ}$ and $\zeta$ is equal to the identity on a neighborhood of $\partial X$.
\end{lemma}
(Here $c^{-1}Y$ is the image of $Y$ under the time-$(\log c^{-1})$ flow of the Liouville vector field of $\mu$ whereas $c^{-1}X^{\circ}$ is the image of $X^{\circ}$ under the time-$(\log c^{-1})$ flow of the Liouville vector field of $\lambda$. Since these vector fields are different there is no reason to expect that $c^{-1}Y\subset c^{-1}X^{\circ}$, so the lemma is not completely trivial.)

\begin{proof}
Write $\mu=\lambda|_Y+df_Y$ where $f_Y\co Y\to \R$ is smooth, and (using a collar for $\partial Y$ in $X$ and a cutoff function) extend $f_Y$ to a function $f\co X\to \R$ which has support contained in $X^{\circ}$ and which coincides with $f_Y$ on $Y$.  This results in a one-form $\tilde{\mu}=\lambda+df\in \Omega^1(X)$ such that $\tilde{\mu}|_Y=\mu$ and such that $\tilde{\mu}$ coincides with $\lambda$ on a neighborhood of $\partial X$.  

Now form the Liouville completion $\hat{X}$ of $X$, so that $\hat{X}=X\cup_{\partial X}\left(\partial X\times [1,\infty)\right)$ and $\lambda$ extends to $\hat{X}$ as a one-form $\hat{\lambda}$ that is equal to $s\lambda|_{\partial X}$ on $\partial X\times[1,\infty)$ where $s$ is the $[1,\infty)$ variable. Let us extend the function $f\co X\to \R$ (which has support in $X^{\circ}$) to a smooth function on $\hat{X}$ (still denoted $f$) by setting it equal to zero on $\hat{X}\setminus X$.  Then $\tilde{\mu}$ likewise extends to a one-form $\hat{\mu}=\hat{\lambda}+df$ on $\hat{X}$, and we have $\hat{\mu}|_{\hat{X}\setminus X}=\hat{\lambda}|_{\hat{X}\setminus X}$. 

On $\hat{X}$ the Liouville flows $\mathcal{L}_{\hat{\lambda}}^{t},\mathcal{L}_{\hat{\mu}}^{t}$ associated to $\hat{\lambda}$ and $\hat{\mu}$ are defined for all $t\in\R$.  On $\partial X\times[1,\infty)$ we have, for all $t\geq 0$ and $(x,s)\in \partial X\times [1,\infty)$, $\mathcal{L}^{t}_{\hat{\lambda}}(x,s)=\mathcal{L}^{t}_{\hat{\mu}}(x,s)=(x,se^t)$.  

Now, letting $c>1$ as in the statement of the proposition, define \[ \zeta=\mathcal{L}^{-\log c}_{\hat{\lambda}}\circ \mathcal{L}^{\log c}_{\hat{\mu}}\co \hat{X}\to\hat{X}.\]  Obviously $\zeta$ acts as the identity on $\partial X\times [1,\infty)$, so $\zeta$ restricts to a diffeomorphism of $X$. Moreover since on a neighborhood of $\partial X$ the vector fields $\mathcal{L}_{\hat{\lambda}}$ and $\mathcal{L}_{\hat{\mu}}$ coincide and point outward toward $\partial X\times[1,\infty)$, $\zeta$ will be equal to the identity throughout this neighborhood of $\partial X$.

By definition, $\mathcal{L}^{\log c}_{\hat{\mu}}(c^{-1}Y)=Y$, so since $Y\subset X^{\circ}$ we have \[ \zeta(c^{-1}Y)\subset \mathcal{L}^{-\log c}_{\hat{\lambda}}(X^{\circ})=c^{-1}X^{\circ}.\]  Finally since $\hat{\mu}=\hat{\lambda}+df$ we find \begin{align*} \zeta^{*}\hat{\lambda}&=\mathcal{L}_{\hat{\mu}}^{\log c*}\mathcal{L}_{\hat{\lambda}}^{\log c^{-1}*}\hat{\lambda}=c^{-1}\mathcal{L}_{\hat{\mu}}^{\log c*}\left(\hat{\mu}-df\right)
\\ &= c^{-1}\left(c\hat{\mu}-d(f\circ\mathcal{L}_{\hat{\mu}}^{\log c})\right)=\hat{\lambda}+d\left(f-c^{-1}f\circ\mathcal{L}_{\hat{\mu}}^{\log c}\right).\end{align*} 
So upon restricting from $\hat{X}$ to $X$ we see that $\zeta^*\lambda-\lambda$ is indeed exact.
\end{proof}

\begin{prop}\label{opencon}
Suppose that $(U,\lambda)$ and $(V,\mu)$ are open Liouville domains and suppose that $\alpha>1$ and that there is a Liouville embedding $\phi\co V\xhookrightarrow{L} U$ such that $\alpha^{-1}U\subset \phi(V)$.  Then if $C>C'>1$ there is also a Liouville embedding $\hat{\phi}\co V\xhookrightarrow{L} U$ such that both $\alpha^{-1}U\subset \hat{\phi}(V)$ and $C^{-1}\alpha^{-1} U\subset \hat{\phi}(C'^{-1}V)$.
\end{prop}

\begin{proof}
By Proposition \ref{sstame} there is a compact subset $Y\subset \alpha^{-1}U$ such that $(Y,\lambda)$ is a Liouville domain and $\frac{C'}{C}\alpha^{-1}U\subset Y$.  So since $\alpha^{-1}U\subset \phi(V)$ we obtain a compact subset $\phi^{-1}(Y)\subset V$, which is a Liouville domain with respect to the form $\phi^*\lambda$.  Again by Proposition \ref{sstame}, there is a compact subset $X\subset V$ with $(X,\mu)$ a Liouville domain such that $\phi^{-1}(Y)\subset X^{\circ}$.   Applying Lemma \ref{outin} gives a Liouville diffeomorphism $\zeta\co X\to X$ which is equal to the identity near $\partial X$, such that $\zeta(C'^{-1}\phi^{-1}(Y))\subset C'^{-1}X$.  Here $C'^{-1}\phi^{-1}(Y)$ is defined using the Liouville flow of $\phi^*\lambda$; since this flow pushes forward via $\phi$ to the Liouville flow of $\lambda$ on $Y$ we have $C'^{-1}\phi^{-1}(Y)=\phi^{-1}(C'^{-1}Y)$. On the other hand $C'^{-1}X$ is defined using the Liouville flow of $\mu\in \Omega^1(V)$, so in particular $C'^{-1}X\subset C'^{-1}V$.

Because $\zeta$ is the identity near $\partial X$, $\zeta$ extends to a Liouville diffeomorphism $\zeta\co V\to V$. Now let $\hat{\phi}=\phi\circ\zeta^{-1}$.  Then $\hat{\phi}$ is a Liouville embedding of $V$ into $U$ having the same image as $\phi$; in particular this image contains $\alpha^{-1}U$.  Moreover we have \begin{align*}
\hat{\phi}(C'^{-1}V)&\supset \hat{\phi}(C'^{-1}X)=\phi(\zeta^{-1}(C'^{-1}X))\supset \phi(C'^{-1}\phi^{-1}(Y))
\\ &= C'^{-1}Y\supset C'^{-1}\left(\frac{C'}{C}\alpha^{-1}U\right)=C^{-1}\alpha^{-1}U,
\end{align*} as desired.
\end{proof}

We can now finally prove the multiplicative triangle inequalities for our distances.

\begin{prop}\label{triangle}
For open Liouville domains $(U,\lambda),(V,\mu),(W,\nu)$ of the same dimension we have inequalities \begin{itemize} \item[(i)] $d_c\left((U,\lambda),(W,\nu)\right)\leq d_c\left((U,\lambda),(V,\mu)\right)d_c\left((V,\mu),(W,\nu)\right)$, \item[(ii)] $\delta_f\left((U,\lambda),(W,\nu)\right)\leq \delta_f\left((U,\lambda),(V,\mu)\right)\delta_f\left((V,\mu),(W,\nu)\right)$, and \item[(iii)] $d_f\left((U,\lambda),(W,\nu)\right)\leq d_f\left((U,\lambda),(V,\mu)\right)d_f\left((V,\mu),(W,\nu)\right)$.\end{itemize}
\end{prop}

\begin{proof}
By Lemma \ref{rescale}, Liouville embeddings $\phi\co a^{-1/2}U\xhookrightarrow{L} V$ and $\psi\co b^{-1/2}V\xhookrightarrow{L} W$ give rise to a composition of Liouville embeddings $b^{-1/2}a^{-1/2}U\xhookrightarrow{L} b^{-1/2}V\xhookrightarrow{L} W$, which immediately implies (i).  

(ii) is more subtle as it relies on Proposition \ref{opencon}.  Suppose that $a>\delta_f((U,\lambda),(V,\mu))$ and $b>\delta_f((V,\mu),(W,\nu))$.  Choose $z$ with $a>z>\delta_f((U,\lambda),(V,\mu))$, so Corollary \ref{infint} gives embeddings $\phi\co z^{-1/2}U\xhookrightarrow{L} V$ and $\psi\co b^{-1/2}V\xhookrightarrow{L} W$ with $\Img(\phi)\supset z^{-1}V$ and $\Img(\psi)\supset b^{-1}W$.  Then Lemma \ref{rescale} gives embeddings \[ \phi'\co a^{-1/2}b^{-1/2} U\xhookrightarrow{L} z^{1/2}a^{-1/2}b^{-1/2} V\quad \mbox{with }\Img(\phi')\supset (zab)^{-1/2}V \] and \[ \psi'\co z^{1/2}a^{-1/2}b^{-1/2} V\xhookrightarrow{L} (z/a)^{1/2}W\subset W\quad \mbox{with }\Img(\psi')\supset (z/a)^{1/2}b^{-1}W.\]  Moreover by applying Proposition \ref{opencon} with $\alpha=b$, $C'=z$, and $C=(az)^{1/2}$ we see that $\psi'$ can be chosen to  have the additional property that \[ \psi'\left((zab)^{-1/2}V\right)\supset a^{-1}b^{-1}W.\]  Putting these together we see that $\psi'\circ\phi'\co a^{-1/2}b^{-1/2}U\xhookrightarrow{L} W$ is a Liouville embedding whose image contains $a^{-1}b^{-1}W$, whence $\delta_f((U,\lambda),(W,\nu))\leq ab$.  Since $a$ and $b$ were arbitrary subject to the requirements that $a>\delta_f((U,\lambda),(V,\mu))$ and $b>\delta_f((V,\mu),(W,\nu))$, this suffices to prove (ii).  

Given (ii), (iii) follows immediately from the definition of $d_f$.
\end{proof}

\begin{cor}\label{sympinv}
If $(U,\lambda),(U',\lambda'),(V,\mu)$ are open Liouville domains of the same dimension such that $(U,\lambda)$ and $(U',\lambda')$ are Liouville diffeomorphic, then we have: \[ d_c((U,\lambda),(V,\mu))=d_c((U',\lambda'),(V,\mu)),\quad \delta_f((U,\lambda),(V,\mu))=\delta_f((U',\lambda'),(V,\mu)),\] and\[ \delta_f((V,\mu),(U,\lambda))=\delta_f((V,\mu),(U',\lambda')).\]
\end{cor}

\begin{proof}
It is clear from the definitions that we have \[ d_c((U,\lambda),(U',\lambda'))=\delta_f((U,\lambda),(U',\lambda'))=\delta_f((U',\lambda'),(U,\lambda))=1,\] so the corollary is an immediate consequence of the multiplicative triangle inequalities: for instance we have \begin{align*} &\delta_f ((U,\lambda),(V,\mu))\leq \delta_f((U,\lambda),(U',\lambda'))\delta_f((U',\lambda'),(V,\mu))\\&=\delta_f((U',\lambda'),(V,\mu))\leq \delta_f((U',\lambda'),(U,\lambda))\delta_f((U,\lambda),(V,\mu))=\delta_f((U,\lambda),(V,\mu)).\end{align*}
\end{proof}

In a somewhat different direction, the following  result will be helpful in proving Theorem \ref{quasiembed}.

\begin{lemma}\label{symplecto}
Let $X$ be a manifold without boundary equipped with a smooth family of $1$-forms $\lambda_t$ ($0\leq t\leq 1$) such that $d\lambda_t$ is symplectic and is independent of $t$. Assume furthermore that the Liouville vector fields $\mathcal{L}_{\lambda_t}$ of $\lambda_t$ are each complete.  Let $W$ be a compact codimension-zero submanifold of $X$ with boundary $\partial W$, having the properties that each $\mathcal{L}_{\lambda_t}$ is positively transverse to $\partial W$, and that every point of $X$ lies on a flowline of $\mathcal{L}_{\lambda_t}$ that intersects $W$.  Then there is a smooth family of symplectomorphisms $F_t\co X\to X$ such that $F_0=1_X$ and the support of $F_{t}^{*}\lambda_t-\lambda_0$ is contained in $W^{\circ}$ for all $t$.
\end{lemma}
(This is similar to a special case of \cite[Proposition 11.8]{CE}; we give details because we need more control over the support of $F_{t}^{*}\lambda_t-\lambda_0$ than can be read off from that result.)

\begin{proof}
Write $Y=\partial W$, and $\alpha_t=\lambda_t|_Y$.  By the Gray stability theorem, there is a diffeotopy $\psi_t\co Y\to Y$ with $\psi_0=1_Y$ and $\psi_{t}^{*}\alpha_t=e^{g_t}\alpha_0$ for some $g_t\co Y\to\R$. Define $G_t\co \R\times Y\hookrightarrow X$ by \[ G_t(s,y)=\mathcal{L}_{\lambda_t}^{s-g_t(y)}(\psi_t(y)).\]  Based on the facts that $\lambda_t(\mathcal{L}_{\lambda_t})=0$ and that $\mathcal{L}_{\lambda_t}^{r*}\lambda_t=e^r\lambda_t$ it is easy to check that we have \[ G_{t}^{*}\lambda_t=e^s\alpha_0, \] independently of $t$. (Here $s$ denotes the coordinate on $\R$.)

Let $0<\delta<\frac{1}{2}$, choose $C>1$ such that $g_t(y)\leq C-1$ for all $t\in [0,1]$ and $y\in  Y$, and choose a smooth function $\chi\co \R\to [0,1]$ such that $\chi(s)=0$ for $s<-C+\delta$, $\chi(s)=1$ for $s>-\delta$, and $\chi'(s)<\frac{1}{C-1}$ for all $s$. The latter condition implies that, for all $(t,y)\in [0,1]\times Y$, the function $s\mapsto s-\chi(s)g_t(y)$ has positive derivative everywhere.   Also let $\nu\co\R\to [0,1]$ be smooth with $\nu(s)=0$ for $s<-C+\frac{\delta}{2}$ and $\nu(s)=1$ for $s>-C+\delta$.  Thus the map $\tilde{G}_t\co [-C,\infty)\times Y\to X$ defined by \[ \tilde{G}_t(s,y)=\mathcal{L}_{\lambda_t}^{s-\chi(s)g_t(y)}(\psi_{\nu(s)t}(y)) \] is an embedding which coincides with $G_t$ on $[-\delta,\infty)\times Y$ but is given on $[-C,-C+\delta/2]\times Y$ by $(s,y)\mapsto \mathcal{L}_{\lambda_t}^{s}(y)$. For $t=0$, since $g_0\equiv 0$ we have $\tilde{G}_0(s,y)=G_0(s,y)=\mathcal{L}_{\lambda_0}^{s}(y)$ for all $s,y$.

Moreover the image of $\tilde{G}_t$ is $\cup_{s\geq -C}\mathcal{L}_{\lambda_t}^{s}(Y)$.  The hypotheses on the behavior of $\mathcal{L}_{\lambda_t}$ with respect to $W$ imply that $X\setminus W^{\circ}=\cup_{s\geq 0}\mathcal{L}^{s}_{\lambda_t}(Y)$, so the image of $\tilde{G}_{t}$ can equivalently be written as $X\setminus\mathcal{L}_{\lambda_t}^{-C}(W^{\circ})$.  For each $t$ we thus have a diffeomorphism \[ \tilde{G}_{t}\circ \tilde{G}_{0}^{-1}\co X\setminus \mathcal{L}_{\lambda_0}^{-C}(W^{\circ})\to X\setminus \mathcal{L}_{\lambda_t}^{-C}(W^{\circ}). \] The restriction of this diffeomorphism to a neighborhood of $X\setminus W^{\circ}=\tilde{G}_0([0,\infty)\times Y)$ coincides with $G_{t}\circ G_{0}^{-1}$ and hence pulls back $\lambda_t$ to $\lambda_0$ since $G_{t}^{*}\lambda_t$ is independent of $t$.  On the other hand the restriction of $\tilde{G}_{t}\circ \tilde{G}_{0}^{-1}$ to a neighborhood of the boundary $\mathcal{L}_{\lambda_0}^{-C}(Y)$ of its domain sends $w$ to $\mathcal{L}^{s_w}_{\lambda_t}\mathcal{L}^{-s_w}_{\lambda_0}(w)$ where $s_w$ is the unique number with $w\in \mathcal{L}_{\lambda_0}^{s_w}(Y)$.  We can then extend $\tilde{G}_t\circ\tilde{G}_{0}^{-1}$ to a diffeomorphism $\tilde{F}_t\co X\to X$ by taking $\tilde{F}_t|_{\mathcal{L}_{\lambda_0}^{-C}(W)}$ to be given by $w\mapsto \mathcal{L}^{f(w)}_{\lambda_t}\circ \mathcal{L}^{-f(w)}_{\lambda_0}(w)$ for a suitable smooth function $f\co \mathcal{L}_{\lambda_0}^{-C}(W)\to \R$ that is constant on $\mathcal{L}_{\lambda_0}^{-C-\delta}(W)$ and obeys $f(w)=s_w$ for $w$ near $\mathcal{L}_{\lambda_0}^{-C}(Y)$.  (Concretely, $f$ can be constructed by first choosing a monotone smooth fuction $\beta\co [-C-\delta,-C]\to \R$ that is the identity on a neighborhood of $-C$ and is equal to a constant $\beta_0$ on a neighborhood of $-C-\delta/2$, and then setting $f\co \mathcal{L}_{\lambda_0}^{-C}(W)\to \R$ equal to $\beta(s)$ on $\mathcal{L}_{\lambda_0}^{s}(Y)$ for $s\in [-C-\delta,-C]$, and to $\beta_0$ on $\mathcal{L}_{\lambda_0}^{-C-\delta}(W)$.)

We thus have a diffeotopy $\tilde{F}_t\co X\to X$ such that $\tilde{F}_{t}^{*}\lambda_t-\lambda_0$ has support contained in $W^{\circ}$.  We finally use a standard Moser argument, introducing the new diffeotopy $F_t=\tilde{F}_t\circ \phi_t$ where $\{\phi_t\}$ is the flow of the time-dependent vector field $\{V_t\}$ given as the solution to $\iota_{V_t}d\tilde{F}_{t}^{*}\lambda=-\left.\frac{d}{du}\right|_{u=t}\left(F_{u}^{*}\lambda_u\right)$.  In particular $V_t$ and hence $\phi_t$ has support contained in $W^{\circ}$ and a standard calculation with Cartan's formula shows that $\frac{d}{dt}F_{t}^{*}d\lambda_t=\frac{d}{dt}\phi_{t}^{*}\tilde{F}_{t}^{*}d\lambda_t=0$.  Thus the $F_t$ are  symplectomorphisms that obey the required properties.
 \end{proof}


\section{$\delta_f$ and filtered equivariant symplectic homology}\label{CHSect}

In the present section we explain, following \cite{V},\cite{BO}, \cite{G}, \cite{GH}, and \cite{GU}, how to associate filtered symplectic homology groups to open Liouville domains, giving rise to constraints on the hemidistance $\delta_f$.  This is most naturally explained using the language of persistence modules (see \cite{CSGO}), though we will not require any deep results concerning these.  It is convenient to parametrize our persistence modules by the positive reals $\mathbb{R}_+$, rather than by $\mathbb{R}$ as is more common in the literature, as $\mathbb{R}_+$ parametrizes scalings of domains in $\mathbb{R}^{2n}$ (or more generally in an exact symplectic manifold with complete Liouville flow).  Of course there is no essential difference here since one can move from scalings in $\mathbb{R}_+$ to translations in $\R$ by taking logarithms.  We begin with the following definition.

\begin{dfn} Let $\mathcal{C}$ be a category.  A $\mathbb{R}_+$-persistence module in $\mathcal{C}$ consists of a collection
of objects $V_s$ for all $s\in \R_+$ together with morphisms (``structure maps'') $\sigma_{ts}\co V_s\to V_t$ whenever $s\leq t$ such that $\sigma_{ss}$ is the identity and $\sigma_{us}=\sigma_{ut}\circ \sigma_{ts}$ whenever $s\leq t\leq u$.

If $\mathbb{V}=\{V_s,\sigma_{ts}\}$ and $\mathbb{W}=\{W_s,\tau_{ts}\}$ are $\R_+$-persistence modules in $\mathcal{C}$ a \textbf{morphism} $F\co\mathbb{V}\to\mathbb{W}$ consists of morphisms (in $\mathcal{C}$) $F_s\co V_s\to W_s$ for all $s$ such that $F_t\circ\sigma_{ts}=\tau_{ts}\circ F_s$ whenever $s\leq t$.
\end{dfn}

In other words, viewing $\R_+$ as a category with a single morphism from $s$ to $t$ when $s\leq t$, the category of $\mathbb{R}_+$-persistence modules in $\mathcal{C}$ is just the category whose objects are functors $\mathbb{R}_+\to \mathcal{C}$ and whose morphisms are natural transformations.  Let us denote this category by $\mathcal{C}^{\R_+}$.

If $(X,\lambda)$ is a Liouville domain, then $\lambda|_{\partial X}$ defines a contact form on $\partial X$. Let us call  $(X,\lambda)$ \emph{nondegenerate} if the Reeb vector field $R_{\lambda}$ of $\lambda|_{\partial X}$ has the property that the linearized return map of each closed orbit of $R_{\lambda}$, acting on $\ker(\lambda|_{\partial X})$, does not have one as an eigenvalue.  Let $\mathbf{LDom}_{2n}$ denote the category whose objects are nondegenerate $2n$-dimensional Liouville domains, and whose morphisms $(X,\lambda)\to (Y,\mu)$ are maps $\phi\co X\to Y$ such that \emph{either} $\phi$ is a diffeomorphism with $\phi^*\mu=\lambda$, \emph{or} $\phi$ is a Liouville embedding whose image is contained in $Y^{\circ}$.

Some of the results in \cite[Section 3]{GH} can be summarized by the statement that filtered positive $S^1$-equivariant symplectic homology defines a functor \[ \mathbb{CH}\co \mathbf{LDom}_{2n}^{op}\to (\mathbf{Mod}_{\mathbb{Q}[T]})^{\mathbb{R_+}}.\] (Here the $\mathbf{Mod}_{\mathbb{Q}[T]}$ refers to the category of modules over the polynomial ring $\mathbb{Q}[T]$, and the superscript  $op$ refers to the opposite category; in other words $\mathbb{CH}$ is a contravariant functor from $\mathbf{LDom}_{2n}$ to  $(\mathbf{Mod}_{\mathbb{Q}[T]})^{\mathbb{R_+}}$.)   Namely, $\mathbb{CH}$ sends the object $(X,\lambda)$ to a persistence module over $\mathbb{R}_+$ whose value at $L\in \R_+$ is the $\mathbb{Q}$-vector space denoted $CH^L(X,\lambda)$ in \cite{GH}, which is made into a $\mathbb{Q}[T]$-module by the map $U^L$ from the ``(U map)'' statement of \cite[Proposition 3.1]{GH}.\footnote{We use $T$ rather than $U$ as our formal variable so that $U$ can refer to an open Liouville domain later on.  For the main applications of this paper the fact that $CH^{L}(X,\lambda)$ is a $\mathbb{Q}[T]$-module as opposed to just a $\mathbb{Q}$-vector space will not be important.}  The structure maps of this persistence module are the maps $\imath_{L_2,L_1}$ from \cite[(3.1)]{GH}.  As for the action on morphisms, for a  Liouville embedding $\phi\co (X,\lambda)\hookrightarrow (Y,\mu)$, $\mathbb{CH}$ sends $\phi$ to the $L$-parametrized family of ``transfer maps'' $\Phi^L\co CH^L(Y,\mu)\to CH^L(X,\lambda)$; that these assemble into a morphism in the category of $\R_+$-persistence modules in $\mathbf{Mod}_{\mathbb{Q}[T]}$is the content of \cite[(3.4) and (3.6)]{GH}. If instead $\phi$ is a Liouville isomorphism then the associated map $\mathbb{CH}(\phi)$ is instead formed by pulling back all of the ingredients in the construction of $CH^L(Y,\mu)$ in the obvious way as in \cite[Lemma 4.14]{G}.  The functoriality property $\mathbb{CH}(\psi\circ\phi)=\mathbb{CH}(\phi)\circ\mathbb{CH}(\psi)$ holds as in \cite[Theorem 4.12]{G}.  We say a little more about how these structures are constructed in Section \ref{scalesect} below; for details one should consult references such as \cite{BO},\cite{GH}.

\subsection{Gradings}
In our applications it will be useful to appeal to an absolute $\mathbb{Z}$-grading on equivariant symplectic homology, which requires imposing topological hypotheses on the Liouville domains in question. A $\Z$-grading on the $\mathbb{Q}[T]$-modules $CH^L(X,\lambda)$ (with the formal variable $T$ having degree $-2$) requires systematically choosing homotopy classes of trivializations of the symplectic vector bundles $\gamma^*TX$ as $\gamma$ varies through loops $\gamma\co S^1\to X$, in such a way that whenever $\Gamma\co [0,1]\times S^1\to X$ is a homotopy between two loops $\gamma_0,\gamma_1$, the chosen trivializations of $\gamma_{0}^{*}TX$ and $\gamma_{1}^{*}TX$ simultaneously extend over $\Gamma^*TX$.  Consideration of the case that $\gamma_0=\gamma_1$ shows that a necessary condition for such a grading is that $c_1(TX)$ vanish on homology classes represented by tori.  If one just wants a $\Z$-grading on $CH^L(X,\lambda)$ this is also sufficient, as one can see by choosing trivializations of $\gamma^*TX$ for one choice of $\gamma$ in each component of the free loop space of $X$ and then extending these via homotopies.  However this involves a non-canonical choice which cannot be expected to behave well with respect to transfer maps.

To get a more canonical grading we can impose the conditions that $H_1(X;\Z)=\{0\}$ and that $c_1(TX)$ represent a torsion class in $H^2(X;\Z)$.  Then any loop $\gamma\co S^1\to X$ has the form $\gamma=u|_{\partial \Sigma}$ for some map $u\co \Sigma\to X$ where $\Sigma$ is an oriented surface with boundary $S^1$, and a unique homotopy class of trivializations of $\gamma^*TX$ is prescribed by requiring the trivialization to extend to $u^*TX$.  Moreover the assumption that $c_1(TX)$ is torsion  implies that this homotopy class is independent of the choice of $\Sigma$ and $u$.  In this way the persistence modules $\mathbb{CH}(X,\lambda)$ obtain a $\Z$-grading whenever $H_1(X;\Z)=\{0\}$ and $c_1(TX)$ is torsion.  Moreover if $(X,\lambda)$ and $(Y,\mu)$ both satisfy these conditions and if $\phi\co X\to Y$ is a morphism in $\mathbf{LDom}_{2n}$ then the associated transfer map $\mathbb{CH}(Y,\mu)\to \mathbb{CH}(X,\lambda)$ preserves gradings, as follows directly from the construction of this map as described \emph{e.g.} in \cite[Section 7]{GH}; the point is that the transfer map is defined using solutions to the Floer continuation equation, and such solutions provide homotopies over which the relevant trivializations can be extended. 

Thus if we let $\mathbf{SLDom}_{2n}$ be the full subcategory of $\mathbf{LDom}_{2n}$ whose objects are nondegenerate $2n$-dimensional Liouville domains $(X,\lambda)$ having $H_1(X;\Z)=\{0\}$ and $c_1(TX)$ torsion, then the restriction of the functor $\mathbb{CH}$ to $\mathbf{SLDom}_{2n}$ lifts to a functor (still denoted $\mathbb{CH}$) \[ \mathbb{CH}\co \mathbf{SLDom}_{2n}^{op}\to \left(\mathbf{GrMod}_{\Q[T]}\right)^{\R_+},\] where $\mathbf{GrMod}_{\Q[T]}$ is the category of $\Z$-graded $\Q[T]$-modules, $\Q[T]$ being regarded as a graded ring with $T$ having degree $-2$.

\subsection{Open domains}

We now observe, following \cite{GU}, that the functor $\mathbb{CH}$ just described can be adapted to give a functor on the category $\mathbf{TE}_{2n}$ of $2n$-dimensional tamely exhausted exact symplectic manifolds (as defined in Definition \ref{tamedef}; the morphisms in this category are taken to be all Liouville embeddings), which by Proposition \ref{sstame} includes all open Liouville domains.  This functor is valued in $\left(\mathbf{Mod}_{\Q[T]}\right)^{\R_+}$; if we restrict to the full subcategory of $\mathbf{TE}_{2n}$ consisting of $(U,\lambda)$ with $H_1(U,\Z)=\{0\}$ and $c_1(TU)$ torsion then the functor lifts to $\left(\mathbf{GrMod}_{\Q[T]}\right)^{\R_+}$.

To define this functor on $\mathbf{TE}_{2n}$, on objects one puts $\mathring{\mathbb{CH}}(U,\lambda)=\varprojlim_X \mathbb{CH}(X,\lambda|_X)$ where the inverse limit\footnote{The inverse limit of a directed system $\{M(\alpha)\}$ of persistence modules in a category $\mathbf{D}$ that itself admits inverse limits can be constructed in the obvious way, setting $\varprojlim_{\alpha}M(\alpha)$ to be the persistence module given by setting $\left(\varprojlim_{\alpha}M(\alpha)\right)_t=\varprojlim_{\alpha}(M(\alpha)_t)$ and using the result of the obvious diagram chase to define the structure maps $\left(\varprojlim_{\alpha}M(\alpha)\right)_s\to \left(\varprojlim_{\alpha}M(\alpha)\right)_t $.} is over compact subsets $X\subset U$ with $(X,\lambda|_X)$ an object of $\mathbf{LDom}_{2n}^{op}$, partially ordered by saying that $X<X'$ provided that $X'\subset X^{\circ}$, with transition maps $\mathbb{CH}(X,\lambda|_X)\to \mathbb{CH}(X',\lambda|_{X'})$ when $X<X'$ given by applying the functor $\mathbb{CH}$ to the inclusion of $X'$ into $X$.  On morphisms, $\mathring{\mathbb{CH}}$ assigns to a Liouville embedding $\phi\co (U,\lambda)\to (V,\mu)$ the map $\mathring{\mathbb{CH}}(\phi)\co \mathring{\mathbb{CH}}(V,\mu)\to \mathring{\mathbb{CH}}(U,\lambda)$ characterized by requirement that the diagram \begin{equation}\label{chringdiag} \xymatrix{ \mathring{\mathbb{CH}}(V,\mu)\ar[r]^{\mathring{\mathbb{CH}}(\phi)}\ar[d] & \mathring{\mathbb{CH}}(U,\lambda) \ar[d] \\  \mathbb{CH}(Y,\mu|_Y) \ar[r]^{\mathbb{CH}(\phi|_X)} & \mathbb{CH}(X,\lambda|_X)   } \end{equation} commute whenever $X\subset U,Y\subset V$ with $(X,\lambda|_X),(Y,\mu|_Y)$ being objects of $\mathbf{LDom}_{2n}$ such that $\phi(X)\subset Y^{\circ}$.  Here of course the vertical arrows are the structure maps of the inverse limits.   While \cite[Lemma 2.18]{GU} is stated in a somewhat more specific context, its proof goes through without substantive change to show that this prescription indeed uniquely defines a functor $\mathring{\mathbb{CH}}$.  Note that the fact that $(V,\mu)$ is tamely exhausted guarantees that if one has $U,V,X$ as in the diagram (\ref{chringdiag}) then one can find a subset $Y\subset V$ that allows one to complete the diagram.  (Formally, the definition yields a Liouville domain $(Y,\mu|_Y)$ with $\phi(X)\subset Y^{\circ}\subset Y\subset V$ without a guarantee of nondegeneracy for $Y$, but a small perturbation of $Y$ will be nondegenerate and will satisfy the remaining requirements.)  If it additionally holds that $U,V$ both have vanishing first homology and torsion first Chern class, then by Proposition \ref{sstame} a cofinal system of the nondegenerate Liouville domains contained in $U$ (resp. in $V$) will have the same property, and so $\mathring{\mathbb{CH}}(\phi)$ will be a morphism of $\R_+$-persistence modules in the category of \emph{graded} $\mathbb{Q}[T]$-modules.

\subsection{Scalings}\label{scalesect}

 If $a\geq 1$ and $L>0$, one obtains as in the proof of \cite[Proposition 4.15]{G} a \emph{rescaling isomorphism} $\sigma_a\co CH^L(a^{-1}X,\lambda|_{a^{-1}X})\to CH^{aL}(X,\lambda)$ by using the Liouville flow $\mathcal{L}_{\log a}^{\lambda}$ to identify various objects associated to $a^{-1}X$ with objects associated to $X$.  To explain this detail, for a Liouville domain $(X,\lambda)$ with completion $(\hat{X},\hat{\lambda})$,  $\mathbb{CH}(X,\lambda)$ is formed as a direct limit of filtered equivariant positive\footnote{by positive we just mean that the underlying chain complex $CF^{S^1,N}(H,J)$ is obtained by quotienting out by a certain subcomplex  generated by orbits with action close to zero} Floer homologies $\mathbb{HF}^{S^1,N}(H,J)$ as $H$ varies over the class $\mathcal{H}_{adm}(X)$ of ``admissible parametrized Hamiltonians'' $H\co S^1\times \hat{X}\times S^{2N-1}\to \R$, as $N$ varies through $\N$, and where $J$ belongs to a suitable family of almost complex structures on $\hat{X}$.  (See, \emph{e.g.}, \cite[Definitions 5.8 and 6.1]{GH}.)  For $a\geq 1$, the time-$(\log a)$ flow $\mathcal{L}_{\log a}^{\lambda}$ defines a diffeomorphism $a^{-1}X\to X$  which extends trivially\footnote{by $\psi_a(x,s)=\left(\mathcal{L}_{\log a}^{\lambda}(x),s\right)$ for $(x,s)\in \partial(a^{-1}X)\times[1,\infty)$} to a diffeomorphism $\psi_a\co \widehat{a^{-1}X}\to \hat{X}$ of the symplectizations, such that \begin{equation}\label{confpull} \psi_{a}^{*}\hat{\lambda}=a\widehat{\lambda|_{a^{-1}X}}.\end{equation}  Now a parametrized Hamiltonian $H\co S^1\times \widehat{a^{-1}X}\times S^{2N-1}\to \R$ belongs to $\mathcal{H}_{adm}(a^{-1}X)$ if and only if the parametrized Hamiltonian $\delta_aH\co S^1\times \hat{X}\times S^{2N-1}\to\R$ defined by \begin{equation}\label{deltaa} \delta_aH(t,x,z)=aH(t,\psi_{a}^{-1}(x),z)\end{equation} belongs to $\mathcal{H}_{adm}(X)$; moreover $\psi_{a}$ pushes forward the Hamiltonian vector field of $H$ to the Hamiltonian vector field of $\delta_aH$.  Consequently the action of $\psi_a$ induces  isomorphisms of equivariant chain complexes $CF^{S^1,N}(H,J)\cong CF^{S^1,N}(\delta_aH,\psi_{a*}J)$ for each choice of $H,J,N$; also (\ref{confpull}) implies that the action of $\psi_a$ on generators multiplies their actions by $a$.  That we obtain the aforementioned isomorphisms $\sigma_a\co CH^{L}(a^{-1}X,\lambda|_{a^{-1}X})\to CH^{aL}(X,\lambda)$ in the direct limit follows from the fact that the continuation maps used to take the direct limit that defines $CH^{L}(a^{-1}X,\lambda|_{a^{-1}X})$ can likewise be transported to $X$ via the conformal symplectomorphism $\psi_a$.

Suppose now that $(X,\lambda)$ is a Liouville domain and that $Y\subset X^{\circ}$ has the property that $(Y,\lambda|_{Y})$ is a Liouville domain with $\dim Y=\dim X$.  Let us quickly outline the construction of the transfer map $\mathbb{CH}(X,\lambda)\to \mathbb{CH}(Y,\lambda|_Y)$ from \cite[Section 4.1 and 4.2]{G} and \cite[Sections 7.1 and 7.2]{GH}.  This map is obtained as a direct limit of compositions $HF^{S^1,N}(H_1,J_1)\to HF^{S^1,N}(H_2,J_2)\to HF^{S^1,N}(H_2^Y,J_2^Y)$ where: 
\begin{itemize} \item $H_1\in \mathcal{H}_{adm}(X)$, the function $H_2\co S^1\times \hat{X}\times S^{2N-1}\to \R$ is an ``admissible parametrized stair Hamiltonian'' (see \cite[Definition 7.1]{GH}) and $H_{2}^{Y}\co S^1\times \hat{Y}\times S^{2N-1}\to\R$ is a certain element of $\mathcal{H}_{adm}(Y)$ that coincides with $H_2$ on a neighborhood of $Y$ in $\hat{X}$;
\item the map $HF^{S^1,N}(H_1,J_1)\to HF^{S^1,N}(H_2,J_2)$ is a standard Floer continuation map in $\hat{X}$, while the map $HF^{S^1,N}(H_2,J_2)\to HF^{S^1,N}(H_2^Y,J_2^Y)$ is given by identifying the underlying chain complex $CF^{S^1,N}(H_{2}^{Y},J_{2}^{Y})$ as the quotient of $CF^{S^1,N}(H_2,J_2)$ by a canonical subcomplex (see \cite[(7.5)]{GH}), and then taking the map on homology induced by the quotient projection.\end{itemize}

If $a\geq 1$ the same constructions may be applied to $a^{-1}Y\subset a^{-1}X$.  Moreover the conformal symplectomorphism $\psi_a\co \widehat{a^{-1}X}\to \hat{X}$ has the property that if $H_2$ is an admissible stair Hamiltonian for the pair $(a^{-1}Y,a^{-1}X)$ then the map $\delta_aH_2$ (using the same notation as in (\ref{deltaa})) is an admissible stair Hamiltonian for $(Y,X)$, and the assignment $(H_2,J_2)\rightsquigarrow (H_{2}^{Y},J_{2}^{Y})$ and the identification of $CF(H_2^Y,J_2^Y)$ with a quotient of $CF(H_2,J_2)$ are likewise respected by the operation $\delta_a$ on parametrized Hamiltonians.  Since $\psi_a$ can be used to identify the continuation map $HF^{S^1,N}(H_1,J_1)\to HF^{S^1,N}(H_2,J_2)$ (for $H_1\in \mathcal{H}_{adm}(a^{-1}X)$) with the continuation map $HF^{S^1,N}(\delta_aH_1,\psi_{a*}J_1)\to HF^{S^1,N}(\delta_aH_2,\psi_{a*}J_2)$ we see by passing to direct limits that we have a commutative diagram \begin{equation}\label{refdiag} \xymatrix{ CH^{aL}(X,\lambda)\ar[r]& CH^{aL}(Y,\lambda|_Y) \\ CH^L(a^{-1}X,\lambda|_{a^{-1}X})\ar[r] \ar[u]_{\sigma_a} & CH^L(a^{-1}Y,\lambda|_{a^{-1}Y}) \ar[u]^{\sigma_a}   }\end{equation} whenever $a\geq 1$ and $(X,\lambda),(Y,\lambda|_Y)$ are Liouville domains with $Y\subset X^{\circ}$, where the horizontal maps are transfer maps associated to inclusions.

In a different vein, if $(X,\lambda)$ is a Liouville domain  and if $1\leq a\leq b$ there is an inclusion $b^{-1}X\subset a^{-1}X$, which the functor $\mathbb{CH}$ sends to a morphism of $\mathbb{R}^+$-persistence modules $\mathbb{CH}(a^{-1}X,\lambda|_{a^{-1}X})\to \mathbb{CH}(b^{-1}X,\lambda|_{b^{-1}X})$. 
Specializing this at any given value $L\in \R^+$ yields a morphism of $\Q[T]$-modules $r_{ba,L}\co CH^L(a^{-1}X,\lambda|_{a^{-1}X})\to CH^L(b^{-1}X,\lambda|_{b^{-1}X})$. If $H_1(X;\Z)=\{0\}$ and $c_1(TX)$ is torsion then this morphism preserves grading.


By \cite[Lemma 2.5]{GU}, these morphisms $r_{ba,L}$ are related to the rescaling isomorphisms as follows: for any $L\in \R_+$ the diagram \begin{equation} \xymatrix{ CH^L(a^{-1}X,\lambda|_{a^{-1}X})\ar[r]^{r_{ba,L}} \ar[d]^{\sigma_a} & CH^{L}(b^{-1}X,\lambda|_{b^{-1}X})\ar[d]_{\sigma_b} \\ CH^{aL}(X,\lambda) \ar[r]^{\imath_{bL,aL}} & CH^{bL}(X,\lambda) } \label{scaleclosed}\end{equation}  commutes. 

We now extend this to open Liouville domains $(U,\lambda)$. 
For $a\geq 1$ and $L>0$, taking the inverse limit of the rescaling isomorphisms $\sigma_a\co CH^{L}(a^{-1}X,\lambda|_{a^{-1}X})\to CH^{aL}(X,\lambda)$ over compact subsets $X\subset U$ that are Liouville domains with respect to the fixed one-form $\lambda$  (as is justified by the commutativity of (\ref{refdiag})) gives an isomorphism $\sigma_a\co CH^{L}(a^{-1}U,\lambda|_{a^{-1}U})\to CH^{aL}(U,\lambda)$.  Taking a limit of (\ref{scaleclosed}) gives a commutative diagram \begin{equation}\label{scaleopen} \xymatrix{  CH^L(a^{-1}U,\lambda|_{a^{-1}U})\ar[r]\ar[d]^{\sigma_a}& CH^{L}(b^{-1}U,\lambda|_{b^{-1}U}) \ar[d]^{\sigma_b} \\ CH^{aL}(U,\lambda)\ar[r] & CH^{bL}(U,\lambda)} \mbox{ for $1\leq a\leq b$}\end{equation} for any  open Liouville domain $(U,\lambda)$, where the top arrow is the transfer map associated to the inclusion $b^{-1}U\subset a^{-1}U$ and the bottom map is the structure map for the persistence module $\mathring{\mathbb{CH}}(U,\lambda)$.

\subsection{Implantations}

We will express the key relationship between $\delta_f$ and the equivariant symplectic homology persistence modules in terms of the following notion.
\begin{dfn}
Let $\mathbb{V}=\{V_s,\sigma_{ts}\}$ and $\mathbb{W}=\{W_s,\tau_{ts}\}$ be two $\mathbb{R}_+$-persistence modules in the same category $\mathcal{C}$ and let $a\geq 1$.  An \textbf{$a$-implantation} of $\mathbb{V}$ into $\mathbb{W}$ is a collection of $\mathcal{C}$-morphisms $\phi_s\co V_s\to W_{as}$ and $\psi_s\co W_{s}\to V_{as}$ for all $s\in \R_+$, such that each of the following diagrams commute for all $s,t\in \R_+$ with $s\leq t$:\begin{equation}\label{implantdiagrams} \xymatrix{ V_s\ar[r]^{\sigma_{ts}}\ar[d]_{\phi_s} & V_t \ar[d]^{\phi_t} \\ W_{as} \ar[r]_{\tau_{at,as}} & W_{at}},\quad \xymatrix{ W_s\ar[r]^{\tau_{ts}}\ar[d]_{\psi_s} & W_t \ar[d]^{\psi_t} \\ V_{as} \ar[r]_{\sigma_{at,as}} & V_{at}},\quad \xymatrix{V_s\ar[rr]^{\sigma_{a^2s,s}}\ar[dr]_{\phi_s} & & V_{a^2s} \\ & W_{as}\ar[ru]_{\psi_{as}}&}.\end{equation}
\end{dfn} 

Note that in contrast to the well-established notion of an $a$-interleaving of persistence modules (discussed, \emph{e.g.}, in \cite[Chapter 4]{CSGO}) we do not impose any requirement on the compositions $\phi_{as}\circ\psi_s\co W_{s}\to W_{a^2s}$.  Whereas an $a$-interleaving can be regarded as an ``approximate isomorphism'' between persistence modules, an $a$-implantation from $\mathbb{V}$ into $\mathbb{W}$ is a sort of approximate injection of $\mathbb{V}$ into $\mathbb{W}$.  The asymmetry between $\mathbb{V}$ and $\mathbb{W}$ in the definition will ultimately be what allows us to distinguish between $\delta_f((U,\lambda),(V,\mu))$ and $\delta_f((V,\mu),(U,\lambda))$ for certain open Liouville domains $(U,\lambda),(V,\mu)$.

Here is our main bridge connecting Banach-Mazur distances to persistence modules.

\begin{prop}\label{implantdelta}
Let $(U,\lambda)$ and $(V,\mu)$ be  open Liouville domains and suppose that $\delta_f\left((U,\lambda),(V,\mu)\right)< b$.  Then there is a $b^{1/2}$-implantation of the persistence module $\mathring{\mathbb{CH}}(V,\mu)$ into the persistence module $\mathring{\mathbb{CH}}(U,\lambda)$.  Here $\mathring{\mathbb{CH}}(V,\mu)$  and $\mathring{\mathbb{CH}}(U,\lambda)$ are regarded as persistence modules in $\mathbf{GrMod}_{\mathbb{Q}[T]}$ if both $U$ and $V$ have $H_1=0$ and $c_1$ torsion, and are regarded as persistence modules in $\mathbf{Mod}_{\mathbb{Q}[T]}$ otherwise.
\end{prop}

\begin{proof}
Since if $a\leq b$ a $a^{1/2}$-implantation gives rise in obvious fashion (by composing with the structure maps of the respective persistence modules) to a $b^{1/2}$-implantation, it suffices to show that if $h\co a^{-1/2}U\xhookrightarrow{L} V$ with $a^{-1}V\subset h(a^{-1/2}U)$ then there is a $a^{1/2}$-implantation  of $\mathring{\mathbb{CH}}(V,\mu)$ into $\mathring{\mathbb{CH}}(U,\lambda)$.

Such an embedding $h$ gives rise to a commutative diagram of Liouville embeddings  \begin{equation}\label{spacediag} \xymatrix{ & a^{-1/2}U\ar[ld]_{h} & \\ V  & & a^{-1}V \ar[lu]_{h^{-1}|_{a^{-1}V}} \ar[ll] }  \end{equation}  where the bottom map is the inclusion, and hence to a commutative diagram of persistence module morphisms  \[ 
\xymatrix{ & \mathring{\mathbb{CH}}(a^{-1/2}U,\lambda)\ar[dr]^{F}  & \\ \mathring{\mathbb{CH}} (V,\mu)\ar[rr]^{r_{a1}}\ar[ru]^{G} & &  \mathring{\mathbb{CH}} (a^{-1}V,\mu).
} \] Here each arrow is induced by the corresponding map in (\ref{spacediag}), and in particular $r_{a1}$ is the map associated to the inclusion $a^{-1}V\hookrightarrow V$.  

For $L\in \R_+$ we have specializations $F_L\co CH^L(a^{-1/2}U,\lambda)\to CH^L(a^{-1}V,\mu)$ and $G_L\co CH^L(V,\mu)\to CH^L(a^{-1/2}U,\lambda)$, and (\ref{scaleopen}) gives a commutative diagram \[ \xymatrix{  & CH^L(a^{-1/2}U,\lambda)\ar[rdd]^{F_L}\ar[d]^{\sigma_{\sqrt{a}}} & \\  & CH^{a^{1/2}L}(U,\lambda) \\ CH^{L}(V,\mu)\ar[rr] \ar[ruu]^{G_L} \ar[d]_{\sigma_1} & & CH^L(a^{-1}V,\mu) \ar[d]^{\sigma_a} \\ CH^L(V,\mu) \ar[rr]_{\imath_{aL,L}} & & CH^{aL}(V,\mu) }.\]  (Of course $\sigma_1$ is the identity.)  So assuming that the various $\sigma_a\circ F_L\circ\sigma_{\sqrt{a}}^{-1}\co CH^{\sqrt{a}L}(U,\lambda)\to CH^{aL}(V,\mu)$ have the property that the resulting diagrams \begin{equation}\label{sFs} \xymatrix{  CH^{\sqrt{a}L}(U,\lambda)\ar[r]^{\iota_{\sqrt{a}M,\sqrt{a}L}} \ar[d]_{\sigma_a\circ F_{L}\circ\sigma_{\sqrt{a}}^{-1}} & CH^{\sqrt{a}M}(U,\lambda)\ar[d]^{\sigma_a\circ F_{M}\circ\sigma_{\sqrt{a}}^{-1}} \\ CH^{aL}(V,\mu)\ar[r]_{\imath_{aM,aL}} & CH^{aM}(V,\mu) } \end{equation} commute, and that similar diagrams involving the various $\sigma_{\sqrt{a}}\circ G_L$ commute, our desired $a^{1/2}$-implantation will be given by the maps $\sigma_a\circ F_L\circ\sigma_{\sqrt{a}}^{-1}$  and $\sigma_{\sqrt{a}}\circ G_L$.

To check the required property of the $\sigma_a\circ F_L\circ\sigma_{\sqrt{a}}^{-1}$, we observe that, for $L\leq M$, we have a commutative diagram \[ \xymatrix{CH^{a^{1/2}L}(U,\lambda) \ar[r] \ar[d]^{\sigma_{\sqrt{a}}^{-1}} & CH^{a^{1/2}M}(U,\lambda) \ar[d]_{\sigma_{\sqrt{a}}^{-1}} \\ CH^L(a^{-1/2}U,\lambda)\ar[r]\ar[d]_{F_L}& CH^M(a^{-1/2}U,\lambda)\ar[d]^{F_M} \\ CH^L(a^{-1}V,\mu) \ar[r]\ar[d]_{\sigma_a}& CH^M(a^{-1}V,\mu)\ar[d]^{\sigma_a} \\ CH^{aL}(V,\mu)\ar[r] & CH^{aM}(V,\mu)} \] where the horizontal maps are the structure maps of the persistence modules $\mathring{\mathbb{CH}}(U,\lambda)$,$\mathring{\mathbb{CH}}(a^{-1/2}U,\lambda)$, $\mathring{\mathbb{CH}}(a^{-1}V,\mu)$ and $\mathring{\mathbb{CH}}(V,\mu)$, respectively. Indeed the commutativity of the middle square follows from $F$ being a persistence module morphism, and the commutativity of the top and bottom squares follows by using (\ref{refdiag}) to take an inverse limit of appropriate versions of the  diagram \cite[(2.3)]{GU}.  So (\ref{sFs}) indeed commutes. The same argument (with a slightly smaller diagram) shows that the relevant diagram involving the $\sigma_{\sqrt{a}}\circ G_L$ commutes, completing the proof.
\end{proof}

In order to apply Proposition \ref{implantdelta} it will of course be necessary to know something about the structure of the persistence modules $\mathring{\mathbb{CH}}(U,\lambda)$ for particular examples of open Liouville domains $(U,\lambda)$. For a Liouville domain (in the ordinary sense), the filtered positive equivariant symplectic homology is related to the good\footnote{As usual, a Reeb orbit is called good if it is not an even multiple cover of another Reeb orbit whose Conley--Zehnder index has opposite parity.} closed orbits of the Reeb vector field on the boundary of the domain.  The following result uses information about such orbits (``good Reeb orbits,'' for short) on a sequence of Liouville domains approximating a given open Liouville domain $(U,\lambda)$ in order to gain some information about $\mathring{\mathbb{CH}}(U,\lambda)$.  While this information is far from complete, we will see later that it suffices in the context of our main theorems to give obstructions to implantations that lead via Proposition \ref{implantdelta} to lower bounds on $\delta_f$. 

\begin{lemma}\label{exhaust}
Let $(U,\lambda)$ be an open Liouville domain, let $k,r\in\Z$ and $s,t\in \R$ with $0<s<t$, and suppose that we can write $U=\cup_{m=1}^{\infty} U_m$ where each $U_m\subset U_{m+1}$, and each $(U_m,\lambda)$ is the interior of a Liouville domain $(\bar{U}_m,\bar{\lambda}_m)$ with $H_1(\bar{U}_m;\Z)=\{0\}$ and $c_1(T\bar{U}_m)$ torsion whose periodic Reeb orbits satisfy the following properties, independently of $m$:
\begin{itemize} \item[(i)] Among the nondegenerate good Reeb orbits on $\partial \bar{U}_m$ with Conley--Zehnder index $k$, none have period exactly $s$, and exactly $r$ have period less than $s$.
\item[(ii)] Every Reeb orbit $\gamma$ on $\partial\bar{U}_m$ apart from those from part (i) either is nondegenerate with Conley--Zehnder index $CZ(\gamma)$ not belonging to the set $\{k-1,k+1\}$, or else has period greater than $t$.
\end{itemize}
Then $\dim_{\Q}CH^{s}_{k}(U,\lambda)=r$, and for all $u\in [s,t)$ the structure map $\imath_{us}\co CH^{s}_{k}(U,\lambda)\to CH^{u}_{k}(U,\lambda)$ is injective.
\end{lemma}

\begin{remark}
In the language of barcodes, the hypothesis of Lemma \ref{exhaust} is designed to imply that the degree-$k$ part of the barcode of $\mathbb{CH}(\bar{U}_m,\lambda_m)$ includes exactly $r$ bars which contain $s$, all of which in fact contain the whole interval $[s,t)$; the conclusion of the Lemma states that the barcode of $\mathring{\mathbb{CH}}(U,\lambda)$  has the same property.
\end{remark}

\begin{proof}[Proof of Lemma \ref{exhaust}]
First we will show that there is no loss of generality in assuming that the $\bar{U}_m$ are contained in $U$ with $\bar{\lambda}_m=\lambda$.  For any $\ep_m>0$ we will have $(1-\ep_m)\bar{U}_m\subset U_m\subset U$, so $\bar{\lambda}_m|_{(1-\ep_m)\bar{U}_m}=\lambda|_{(1-\ep_m)\bar{U}_m}$ since by assumption $\bar{\lambda}_m|_{U_m}=\lambda|_{U_m}$. Moreover rescaling gives a bijection between the Reeb orbits on $\partial \bar{U}_m$ and those on $\partial(1-\ep_m)\bar{U}_m$, under which non-degeneracy and Conley--Zehnder indices are unchanged while periods are multiplied by $1-\ep_m$.  So since the period spectrum of $\partial\bar{U}_m$ is a closed set, if $\ep_m$ is small enough the truth of conditions (i) and (ii) in the statement of the proposition will be unaffected by replacing $\bar{U}_m$ by $(1-\ep_m)\bar{U}_m$.  We claim moreover that if $\ep_m\searrow 0$ then we will still have $U=\cup_{m=1}^{\infty}(1-\ep_m)U_m$.  Indeed, if $K\subset U$ is compact then the hypotheses immediately imply that $K\subset U_M$ for some $M$, and then an easy covering argument shows that there is $\delta>0$ such that $K\subset (1-\delta)U_M$, so since $U_M\subset U_m$ for $m\geq M$, once $m$ is large enough that $\ep_m<\delta$ we will have $K\subset (1-\ep_m)U_m$.  So since $U$ is exhausted by its compact subsets this shows that $U=\cup_{m=1}^{\infty}(1-\ep_m)U_m$.  Putting together the facts in this paragraph shows that, by replacing each $U_m$ with $(1-\ep_m)U_m$ for a sufficiently small sequence $\ep_m\searrow 0$, we may indeed assume that $\bar{U}_m\subset U$ with $\bar{\lambda}_m=\lambda$.

Moreover there is no loss of generality in assuming that the Reeb flow on $\partial \bar{U}_m$ is nondegenerate, as this can be achieved by arbitrarily small perturbations which (if taken small enough) will not affect conditions (i) and (ii).

Assuming this, let $1<\zeta<t/s$.  The increasing family of open sets $U_m$ then comprise an open cover of the compact set $\overline{\zeta^{-1}U}$, so $\zeta^{-1}U\subset U_M$ for some $M$.  Hence for all $m\geq M$ we have a sequence of inclusions of Liouville domains (with respect to the same one-form $\lambda$) $\zeta^{-1}\bar{U}_m\subset \bar{U}_M\subset \bar{U}_m$.  This gives rise to a diagram, for any $m\geq M$, \begin{equation}\label{appint} \xymatrix{ CH^{s}_{k}(\bar{U}_m,\lambda)\ar[r] \ar[rrd] & CH^{s}_{k}(\bar{U}_M,\lambda) \ar[r] & CH^{s}_{k}(\zeta^{-1}\bar{U}_m,\lambda)\ar[d]^{\sigma_{\zeta}}_{\cong} \\ & & CH^{\zeta s}(\bar{U}_m,\lambda) } \end{equation} where the horizontal maps are transfer maps induced by inclusion and the diagonal map is induced by inclusion of filtered subcomplexes. This diagram commutes by \cite[Lemma 2.4]{GU}. By \cite[Lemma 2.1]{GU}, for any $m\geq 1$ and any $u\in \R_+$, $CH^{u}(\bar{U}_m,\lambda)$ is the $u$-filtered homology of an $\R$-filtered complex $\{CC_{*}(\bar{U}_m,\lambda)\}$ whose generators in degree $d$ are in bijection with good Reeb orbits of Conley--Zehnder index $d$, with filtration level given by the period of the orbit.  In particular our hypotheses imply that if $u<t$ then $CC_{k\pm 1}^{u}(\bar{U}_m,\lambda)=\{0\}$, so that $CH^{u}_{k}(\bar{U}_m,\lambda)=CC^{u}_{k}(\bar{U}_m,\lambda)$ with the inclusion-induced map $CH_{k}^{s}(\bar{U}_m,\lambda)\to CH^{u}(\bar{U}_m,\lambda)$ an injection for $s<u<t$.  Hence in (\ref{appint}), both $CH^{s}_{k}(\bar{U}_m,\lambda)$ and $CH^{s}_{k}(\bar{U}_M,\lambda)$ are $r$-dimensional and the diagonal map is an injection, whence the first map is also an injection, and hence an isomorphism by dimensional considerations.  

Since $CH_{k}^{s}(\bar{U}_m,\lambda)\to CH_{k}^{s}(\bar{U}_M,\lambda)$ is an isomorphism for all $m\geq M$, and since the $\bar{U}_m$ form a cofinal sequence in the inverse limit defining $CH_{k}^{s}(U,\lambda)$, we conclude that the canonical map $CH^{s}_{k}(U,\lambda)\to CH^{s}_{k}(\bar{U}_M,\lambda)$ is an isomorphism, and in particular that $\dim_{\Q}CH^{s}_{k}(U,\lambda)=r$.

Furthermore for $s<u<t$ there is a commutative diagram \[ \xymatrix{ CH^{s}_{k}(U,\lambda)\ar[r]\ar[d] & CH^{u}_{k}(U,\lambda) \ar[d] \\ CH^{s}_{k}(\bar{U}_M,\lambda)\ar[r] & CH^{u}_{k}(\bar{U}_M,\lambda) } \] where the horizontal maps are the structure maps of the  persistence modules $\mathring{\mathbb{CH}}(U,\lambda)$ and $\mathbb{CH}(\bar{U}_M,\lambda)$ and the vertical maps are the canonical maps for the inverse limit.  Since we have seen that the left map is an isomorphism and the bottom map is injective it follows that the top map is also injective, completing the proof.
\end{proof}

\section{Tubes}\label{tubesect}
This section will describe a rather general way of constructing  $(2n+2)$-dimensional Louville domains $W_H$ from certain autonomous Hamiltonian flows on $2n$-dimensional Liouville domains $W$, and will describe the periods and Conley--Zehnder indices  of the closed Reeb orbits on $\partial W_H$.  Later, this will  be applied to obtain the examples in our main theorems.

Suppose we are given:

\begin{itemize}
\item[(i)] A Liouville domain $(W,\lambda)$ and 
\item[(ii)] An autonomous Hamiltonian $H\co W\to \R$ such that $H|_{\partial W}\equiv 0$, and such that the Hamiltonian vector field $X_H$ (defined by $d\lambda(X_H,\cdot)=dH$) obeys \begin{equation}\label{taudef} \tau_{\lambda,H}:=\iota_{X_H}\lambda+H >0 \mbox{ everywhere on }W.\end{equation}
\end{itemize}

Note that (ii) implies that $H>0$ on $W^{\circ}$, for otherwise $H$  would attain a nonpositive global minimum on $W^{\circ}$, at which $\tau_{\lambda,H}$ would coincide with $H$ because $X_H$ would be zero, contradicting the positivity of $\tau_{\lambda,H}$.  (ii) also immediately implies that $\iota_{X_H}\lambda|_{\partial W}>0$ and that $d\lambda(X_H,\cdot)|_{T\partial W}\equiv 0$, and hence that $X_H|_{\partial W}$ is a positive (possibly nonconstant) multiple of the Reeb field $R_{\partial W}$ of the contact form $\lambda|_{\partial W}$ on $\partial W$.

On $\mathbb{C}=\{x+iy|x,y\in \R\}$, write $\rho=\frac{1}{2}(x^2+y^2)$ and (away from $0$) $d\theta=\frac{xdy-ydx}{x^2+y^2}$.  Thus $\rho d\theta$ extends smoothly by zero over the origin to give the standard primitive $\frac{1}{2}(xdy-ydx)$ for the standard symplectic form on $\mathbb{C}$.  The dual vector field $\partial_{\theta}=-y\partial_x+x\partial_y$ to $d\theta$ with respect to the cotangent frame $\{d\rho,d\theta\}$ over $\C\setminus\{0\}$ also extends smoothly as zero over the origin.  Using these smooth extensions we accordingly regard both $\rho d\theta$ and $\partial_{\theta}$ (but not $d\theta$ by itself) as objects defined over all of $\C$ and not just over $\C\setminus\{0\}$.

\begin{prop}\label{suspliouville}
Let $(W,\lambda)$ and $H\co W\to \R$ satisfy (i) and (ii), and define \[ W_H=\left\{(w,z)\in W\times \C\left| \pi|z|^2\leq H(w)\right.\right\}, \qquad \hat{\lambda}=\lambda+\rho d\theta\in \Omega^1(W\times\C).\]  Then $(W_H,\hat{\lambda})$ is a Liouville domain, and the Reeb field for the contact form $\hat{\lambda}|_{\partial W_H}$ on $\partial W_H=\{(w,z)|\pi|z|^2=H(w)\}$ is given by \[ R_{\partial W_H}=\frac{1}{\tau_{\lambda,H}}(X_H+2\pi\partial_{\theta}).\]
\end{prop}

 \begin{proof} Obviously $d\hat{\lambda}$ is a symplectic form on $W\times \C$. Let $\mathcal{L}$ denote the Liouville vector field of $(W,\lambda)$.  Then writing \[ \hat{\mathcal{L}}=\mathcal{L}+\frac{1}{2}(x\partial_x+y\partial_y),\] we have $\iota_{\hat{\mathcal{L}}}d\hat{\lambda}=\hat{\lambda}$.  Define $f_H\co W\times \C\to\R$ by $f_H(w,z)=\pi|z|^2-H(w)$, \emph{i.e.} $f_H=2\pi \rho-H$.  Then $W_H=f_{H}^{-1}((-\infty,0])$, so to show that $(W_H,\hat{\lambda})$ is a Liouville domain it suffices to show that $df_H(\hat{\mathcal{L}})>0$ everywhere on $f_{H}^{-1}(\{0\})$ (as this shows both that $0$ is a regular value so that $W_H$ is a smooth manifold with boundary $f_{H}^{-1}(\{0\})$, and that the Liouville vector field $\hat{\mathcal{L}}$ points outward along $\partial W_H$).  

We calculate \begin{align} df_H(\hat{\mathcal{L}})&=(2\pi d\rho-dH)\left(\frac{1}{2}(x\partial_x+y\partial_y)+\mathcal{L}\right)=2\pi \rho-(\iota_{X_H}d\lambda)(\mathcal{L}) \nonumber
\\ &= 2\pi\rho+d\lambda(\mathcal{L},X_H)=f_H+H+\iota_{X_H}\lambda = f_H+\tau_{\lambda,H}, \label{dfL}\end{align} which by condition (ii) is indeed positive on $f_{H}^{-1}(\{0\})$. Thus $(W_H,\hat{\lambda})$ is a Liouville domain.

Moreover we find that $df_H(X_H+2\pi\partial_{\theta})=-dH(X_H)+(2\pi)^2 d\rho(\partial_{\theta})=0$ and also that \begin{align*} \iota_{X_H+2\pi\partial_{\theta}}d\hat{\lambda}&=\iota_{X_H}d\lambda+2\pi\iota_{\partial_{\theta}}d(\rho d\theta)
\\ &= dH-2\pi d\rho = -df_H,\end{align*} so the vector field $X_H+2\pi\partial_{\theta}$ is tangent to the level sets of $f_H$ (in particular to $\partial W_H=\{f_H=0\}$) and lies in the kernel of the restriction of $d\lambda$ to each level set of $f_H$.  In particular, along $\partial W_H$, $X_H+2\pi\partial_{\theta}$ is proportional to the Reeb field $R_{\partial W_H}$; specifically \[ X_H+2\pi\partial_{\theta}=\hat{\lambda}(X_H+2\pi\partial_{\theta})R_{\partial W_H}.\]  So the last statement of the proposition follows from the calculation \[ \hat{\lambda}(X_H+2\pi\partial_{\theta})=\lambda(X_H)+2\pi\rho=\tau_{\lambda,H}+f_H \] and the fact that $f_H|_{\partial W_H}\equiv 0$.
\end{proof}

Thus on the contact manifold $(\partial W_H,\ker(\hat{\lambda}|_{\partial W_H}))$ the orbits of the Reeb vector field $R_{\partial W_H}$ are reparametrizations of the orbits of the vector field $X_H+2\pi \partial_{\theta}$.  To be more precise we have:

\begin{cor}\label{flowformula} In the situation of Proposition \ref{suspliouville},
let $\{\phi_{H}^{t}\co W\to W\}_{t\in \R}$ denote the flow of the Hamiltonian vector field $X_H$, and let $\{\psi^t\co \partial W_H\to\partial W_H\}_{t\in \R}$ denote the flow of $R_{\partial W_H}$.  Also define $h\co \R\times W\to \R$ implicitly\footnote{The fact that $\tau_{\lambda,H}$ is positive and (by compactness of $W$) bounded away from zero readily implies that this equation uniquely defines $h(t,w)$, and the implicit function theorem shows that $h$ is smooth.} by the equation \[ \int_{0}^{h(t,w)}\tau_{\lambda,H}(\phi_{H}^{s}(w))ds = t.\]  Then, for $(w,z)\in \partial W_H$, we have \[ \psi^t(w,z)=\left(\phi_{H}^{h(t,w)}(w),e^{2\pi i h(t,w)}z\right).\]
\end{cor}
\begin{proof} Since $h(0,w)=0$ and $\frac{\partial h}{\partial t}(t,w)=\left( \tau_{\lambda,H}(\phi_{H}^{h(t,w)}(w))\right)^{-1}$ this follows immediately from the formula for $R_{\partial W_H}$ in Proposition \ref{suspliouville}.
\end{proof}

\begin{cor}\label{neworbits} Up to time-translation, the periodic Reeb orbits for $(\partial W_H,\hat{\lambda})$ consist of:
\begin{itemize}\item[(i)] For each periodic orbit $c\co \R/T\Z\to\partial W$ for the Reeb field $R_{\partial W}$, the orbit $\gamma_c\co \R/T\Z\to \partial W_H$ defined by $\gamma_c(t)=(c(t),0)$.  
\item[(ii)] For each pair $(w,N)\in W^{\circ}\times \Z_+$ such that $\phi_{H}^{N}(w)=w$, an orbit $\gamma_{w,N}\co \R/T_{w,N}\Z\to \partial W_H$ defined by \[ \gamma_{w,N}(t)=\left(\phi_{H}^{h(t,w)}(w),\sqrt{\frac{H(w)}{\pi}}e^{2\pi i h(t,w)}\right).\]  Here \begin{equation}\label{tformula} T_{w,N}=\int_{0}^{N}\tau_{\lambda,H}(\phi_{H}^{s}(w))ds,\quad \mbox{ and }h(T_{w,N},w)=N.\end{equation}
\end{itemize}\end{cor}

\begin{proof}
The manifold $\partial W_H$ contains $\partial W\times\{0\}$, and (since the vector field $\partial_{\theta}$ on $\C$ vanishes at the origin) the Reeb vector field $R_{\partial W_H}=\frac{1}{\tau_{\lambda,H}}(X_H+2\pi\partial_{\theta})$ is tangent to $\partial W\times\{0\}$.  Hence the orbits of $R_{\partial W_H}$ are each either contained in or disjoint from $\partial W\times\{0\}$.

By construction, $\hat{\lambda}|_{\partial W\times \{0\}}=\lambda|_{\partial W}$, so the restriction of the Reeb field $R_{\partial W_H}$ to $\partial W\times\{0\}$ coincides with $R_{\partial W}$.  Thus the closed Reeb orbits that are contained in $\partial W\times\{0\}$ are as described in item (i) in the corollary.

On the other hand if $T>0$ and $\gamma\co \mathbb{R}/T\Z\to \partial W_H$ is a Reeb orbit with $\gamma(0)=(w_0,z_0)$ with $z_0\neq 0$ (and hence $H(w_0)>0$ and so $w_0\in W^{\circ}$), then it follows from Corollary \ref{flowformula} that $e^{2\pi ih(T,w_0)}z_0=z_0$ and hence that $h(T,w_0)=N$ where $N\in \Z_+$ is the number of times that $\gamma$ intersects $W\times \R_+$.  (We must have $N>0$ since $T>0$ and $\frac{\partial h}{\partial t}>0$.)  Since the image of $\gamma$ intersects $W\times\R_+$, we can reparametrize the domain $\R/T\Z$ by a time-translation so that $\gamma(0)\in W\times \R_+$.  After doing so, we will have $\gamma(0)=\gamma(T)=\left(w,\sqrt{\frac{H(w)}{\pi}}\right)$ for some $w\in W^{\circ}$.  Then Corollary \ref{flowformula} shows that $h(T,w)=N$ and $\phi_{H}^{N}(w)=w$, so $\gamma=\gamma_{w,N}$ as defined in item (ii) of the Proposition, and (by the formula defining $h(T,w)$ in Corollary \ref{flowformula}) we have \[ \int_{0}^{N}\tau_{\lambda,H}(\phi_{H}^{s}(w))ds = T,\] so the period $T$ is as indicated in (\ref{tformula}).  Conversely any $\gamma_{w,N}\co \R/T_{w,N}\Z\to\partial W_H$ is indeed a periodic Reeb orbit by Corollary \ref{flowformula}.  So the closed orbits for $R_{\partial W_H}$ not contained in $\partial W\times \{0\}$ are, up to time-translation, precisely those described in item (ii).
\end{proof}

\begin{ex}\label{introex}
A family of special cases of the construction of Proposition \ref{suspliouville} that will be relevant in Sections \ref{trunc} and \ref{sink} is given as follows.  Choose $a_1,\ldots,a_n>0$ such that $\frac{a_i}{a_j}\notin \Q$ for $i\neq j$, and take $W=E(a_1,\ldots,a_n)$.  Define $u\co \C^n\to \R$ by $u(z_1,\ldots,z_n)=\sum_{j=1}^{n}\frac{\pi|z_j|^2}{a_j}$; thus $W=\{\vec{z}\in \C^n|u(z)\leq 1\}$.  The function $H\co W\to \R$ will be taken to be of the form $h\circ u$ for some $h\co [0,1]\to[0,\infty)$ with $h(1)=0$.  We use $\lambda=\sum_j\rho_jd\theta_j$ where as usual $\rho_j=\frac{1}{2}|z_j|^2$ and $\theta_j$ is the usual angular polar coordinate on the $j$th copy of $\C$ in $\C^n$.

In this case the vector field $X_H$ is given by \[ X_H=-\sum_{j=1}^{n}\frac{2\pi h'(u)}{a_j}\partial_{\theta_j},\] and so \[ \tau_{\lambda,H}=-h'(u)\sum_j\frac{2\pi \rho_j}{a_j}+h(u) = h(u)-uh'(u).\]  The right hand side above has a geometric interpretation as the $y$-intercept of the tangent line to the graph of the function $h$ at $u$, and so for Proposition \ref{suspliouville} to apply $h$ must be chosen so that all of the tangent lines to its graph have positive $y$-intercept.

The Reeb orbits on $\partial W_H$ as described in Corollary \ref{neworbits} include the Reeb orbits on $\partial W\times\{0\}$, which are just circles in the $j$th copy of $\C$ in $\C^{n+1}$ for $j=1,\ldots,n$, as well as Reeb orbits which trace out circles in $\C^{n+1}$ that project to $\C^n$ as periodic points of $\phi_{H}^{1}\co W\to W$.  One such periodic point will always be the origin in $\C^n$, leading to a Reeb orbit which is  a circle in the $(n+1)$st copy of $\C$; other periodic points will depend on the specific function $h$. 
\end{ex}

\subsection{Gradings}

Recall that we regard the positive equivariant symplectic homology of an open Liouville domain $(X,\mu)$ as a persistence module in the category of \emph{graded} $\Q[T]$-modules under the topological hypotheses that $H_1(V;\Z)=\{0\}$ and $c_1(TV)$ is torsion. 

\begin{prop}\label{toppres}
Let $(W,\lambda)$ and $H$ be as in Proposition \ref{suspliouville}, with $\dim W>0$.  \begin{itemize}
\item[(i)] There are isomorphisms $H_1(W_H;\Z)\cong H_1(W;\Z)\cong H_1(\partial W_H;\Z)$.
\item[(ii)] If $c_1(TW)$ is torsion then $c_1(TW_H)$ is a torsion element of $H^2(W_H;\Z)$ and $c_1(\ker\hat{\lambda}|_{\partial W_H})$ is torsion as an element of $H^2(\partial W_H;\Z)$.
\end{itemize}
\end{prop} 

\begin{proof}
Let $\pi\co W_H\to W$ denote the restriction to $W_H$ of the projection $W\times\C\to W$.  Then $\pi$ is obviously a homotopy equivalence, giving an isomorphism on homology. Also there is an obvious  symplectic bundle isomorphism $TW_{H}\cong \pi^*TW\oplus\underline{\C}$ where $\underline{\C}$ denotes the trivial Hermitian line bundle.    So $c_1(TW_H)=\pi^*c_1(TW)$.  This proves the statements about $W_H$ in both (i) and (ii).

As for $\partial W_H$, denote as before $\hat{\mathcal{L}}$ and $R_{\partial W_H}$ the Liouville and Reeb vector fields for $\lambda$ along $\partial W_H$.  Then we have a direct sum decomposition of symplectic vector bundles \[ TW_H|_{\partial W_H}=\ker(\hat{\lambda}|_{\partial W_H})\oplus span\{\hat{\mathcal{L}},R_{\partial W_H}\},\] where the second summand is a trivial bundle.  Hence $c_1(\ker\hat{\lambda}|_{\partial W_H})=c_1(TW_H)|_{\partial W_H}$, which by what we have already shown is torsion provided that $c_1(TW)$ is torsion.

Finally note that $\partial W_H\setminus(\partial W\times\{0\})$ maps diffeomorphically to $W^{\circ}\times S^1$ via the map $(w,z)\mapsto (w,\frac{z}{|z|})$, while a neighborhood of $\partial W\times\{0\}$ in $\partial W_H$ can be identified (by using the Liouville flow of $(W,\lambda)$ to retract a neighborhood of $\partial W$ in $W$ to $\partial W$)  with $\partial W\times D^2$.  This gives rise to a Mayer-Vietoris sequence  \begin{equation}\label{mvseq} \xymatrix{
H_1(\partial W\times S^1;\Z)\ar[r] & H_1(W^{\circ}\times S^1;\Z)\oplus H_1(\partial  W;\Z)\ar[r] & H_1(\partial W_H;\Z)\ar[dll] \\  H_0(\partial W\times S^1;\Z)\ar[r] & H_0(W^{\circ}\times S^1;\Z)\oplus H_0(\partial W;\Z)  
}\end{equation}  The last map is injective, so the second map is surjective.  Let us analyze the first map in terms of the decompositions $H_1(\partial W\times S^1;\Z)=H_0(\partial W;\Z)\otimes[S^1]\oplus H_1(\partial W;\Z)\otimes [pt]$ and $H_1( W^{\circ}\times S^1;\Z)=H_0( W^{\circ};\Z)\otimes[S^1]\oplus H_1(W^{\circ};\Z)\otimes [pt]$.  Because every component of a Liouville domain has nonempty boundary, restricting the first map to 
$H_0(\partial W;\Z)\otimes[S^1]$ gives a surjection to $H_0(W^{\circ};\Z)\otimes [S^1]\oplus\{0\}\oplus\{0\}$. On the other hand the image of the restriction to $H_1(\partial W;\Z)\otimes [pt]$ consists of those elements of form $(0,i_*x,x)$ where $x$ varies through $H_1(\partial X;\Z)$ and $i_*$ is the inclusion-induced map $H_1(\partial W;\Z)\to H_1(W;\Z)$. Thus the full image of
the first map in (\ref{mvseq}) is complementary to the summand $H_1(W^{\circ};\Z)\otimes [\mathrm{pt}]$ of $H_1(W^{\circ}\times S^1;\Z)$.  Hence the second map sends $H_1(W^{\circ};\Z)\otimes [\mathrm{pt}]$ isomorphically to $H_1(\partial W_H;\Z)$. So we indeed have an isomorphism $H_1(W;\Z)\cong H_1(\partial W_H;\Z)$.
\end{proof}

Recall that, for $(V,\mu)$ a Liouville domain with $H_1(V;\Z)=\{0\}$ and $c_1(TV)$ torsion, a Reeb orbit $\gamma$ on $\partial V$ has an associated Conley--Zehnder index $CZ(\gamma)$ which is computed as follows. Use the Liouville vector field to construct a collar neighborhood $(\delta,1]\times \partial V$ of $\partial V$ on which $\mu$ is identified with $r\cdot(\mu|_{\partial V})$ where $r$ is the coordinate on $(\delta,1]$, and extend the Reeb flow $\psi^t$ to a flow $\Psi^t$ on this neighborhood as the Hamiltonian flow of the function $-r$. This flow then acts as the Reeb flow on each parallel copy $\{r\}\times\partial V$ and preserves both the Liouville vector field and the Reeb vector field.  Also choose a map $u\co \Sigma\to V$ with $u|_{\partial \Sigma}=\gamma$; then $CZ(\gamma)$ is defined to be the Maslov index as defined in \cite[Section 4]{RS} of the path of symplectic matrices given by expressing the linearization of the $\Psi^t$ in terms of a trivialization of $\gamma^*TV$ that extends over $u^*TV$.  The assumption on $c_1$ implies that this index is independent of the choices involved.

\begin{remark}\label{twxi} Write $\xi=\ker(\mu|_{\partial V})$. We then have a direct sum decomposition $TV|_{\partial V}=\xi\oplus \underline{\mathbb{C}}$ where the trivial summand $\underline{\mathbb{C}}$ is spanned by the Liouville and Reeb fields. In particular $c_1(\xi)$ is torsion when $c_1(TV)$ is torsion. If the Reeb orbit $\gamma$ happens to be homologically trivial in $\partial V$ and not just in $V$, then the map $u$ from the previous paragraph can be taken to have image in $\partial V$.  This allows us to first trivialize $u^*\xi$ and extend this trivially over the $\underline{\mathbb{C}}$ summand; hence our trivialization of $\gamma^*TV$ will be the direct sum of a trivialization of $\gamma^*\xi$ with a trivialization having frame given by the Liouville and Reeb fields.  Since $\Psi^t$ preserves the latter two vector fields, the linearized flow acts as the identity on the second summand, and so (by the product and zero axioms of \cite[Theorem 4.1]{RS}), in this case $CZ(\gamma)$ is equal to the Maslov index of the linearized Reeb flow acting on $\gamma^*\xi$, computed with respect to a trivialization of $\gamma^*\xi$ that extends over a bounding surface for $\gamma$ in $\partial V$.\end{remark}

\begin{prop} \label{czc}
Let $(W,\lambda)$ and $H$ be as in Proposition \ref{suspliouville}, and assume that $H_1(W;\Z)=\{0\}$ and that $c_1(TW)$ is torsion. Let $c\co \R/T\Z\to\partial W$ be a periodic orbit for $R_{\partial W}$, giving rise to the orbit $\gamma_{c}\co \R/T\Z\to \partial W_H$ as in Corollary \ref{neworbits}.  Assume that $c$ is nondegenerate as a Reeb orbit for $R_{\partial W}$.  Then $\gamma_c$ is nondegenerate as a Reeb orbit for $R_{\partial W_H}$ if and only if $h(T,c(0))\notin \Z$. In this case we have \[ CZ(\gamma_c)=CZ(c)+1+2\lfloor h(T,c(0))\rfloor.\]  
\end{prop}

\begin{proof}
If $u_0\co \Sigma\to W$ has $u_0|_{\partial\Sigma}=c$, then we get a map $u\co \Sigma\to W_H$ with $u|_{\partial \Sigma}=\gamma_c$ by simply putting $u(x)=(u_0(x),0)$.  Since $TW_H=\pi^*TW\oplus\underline{\mathbb{C}}$, a trivialization of $u_{0}^{*}TW$ extends in obvious fashion to a trivialization of $u^*TW_H$. Along $\partial W\times\{0\}$ we have $\ker\hat{\lambda}=(\ker\lambda)\times\C$, and the Liouville and Reeb fields for $\lambda$ along $\partial W\times\{0\}$ coincide with the Liouville and Reeb fields for $\hat{\lambda}$.  With respect to the resulting trivialization of $\gamma_{c}^{*}TW_H=c^*TW\oplus \underline{\mathbb{C}}$, Corollary \ref{flowformula} shows that extended version $\Psi^t$ of the time-$t$ linearized flow of $R_{\partial W_H}$ acts as the (extended) time-$t$-linearized flow of $R_{\partial W}$ on the $c^*TW$ summand, and as rotation by the angle $2\pi h(t,c(0))$ on the $\underline{\C}$ summand.  (To be more precise, Corollary \ref{flowformula} shows explicitly that the flows coincide on $c^*T\partial W$; that they coincide on all of $c^*TW$ then follows from the fact that they preserve the respective Liouville fields, which agree with each other along $\partial W\times\{0\}$.)   As $t$ increases from $0$ to $T$, the value $h(t,c(0))$ increases from $0$ to $h(T,c(0))$, so the conclusion follows immediately from standard properties of the Maslov index.  
\end{proof}

To find the Conley--Zehnder indices of the remaining Reeb orbits on $\partial W_H$, namely the $\gamma_{w,N}$ of Corollary \ref{neworbits}, we will first work out the relationship between two types of trivialization for $\gamma^*\xi_{H}$, where $\gamma\co S^1\to \partial W_H\setminus(\partial W\times\{0\})$, $\xi_H=\ker(\hat{\lambda}|_{\partial W_H})$, and we continue to assume that we are in the situation of the first sentence of Proposition \ref{czc}:
\begin{itemize}\item[(i)] Because $H_1(\partial W_H;\Z)=\{0\}$ and $c_1(\xi_H)$ is torsion by Proposition \ref{toppres} and Remark \ref{twxi}, a trivialization of $\gamma^*\xi_H$ is determined up to homotopy by the requirement that there be a map $u\co \Sigma\to \partial W_H$ where $\Sigma$ is a compact surface with boundary and $u|_{\partial \Sigma}=\gamma$ such that the trivialization $\gamma^*\xi_H$ extends to a trivialization of $u^*\gamma_H$.  Call one such trivialization $\mathcal{T}_{\gamma}\co \gamma^*\xi_H\to S^1\times \C^n$. (Of course, the Conley--Zehnder indices of our Reeb orbits are eventually to be calculated with respect to $\mathcal{T}_{\gamma}$.)
\item[(ii)]  Write $Y=\partial W_H\setminus(\partial W\times\{0\})$. (Equivalently, $Y=\{(w,z)\in \partial W_H|z\neq 0\}$.)  Also let $\pi\co Y\to W^{\circ}$ denote the projection to the $W$ factor. There is then a symplectic vector bundle isomorphism $p\co \pi^*TW|_Y\to \xi_H|_Y$ defined as follows.  For $v\in (\pi^*TW)_{(w,z)}=T_wW$, note that $v$ lifts to the tangent vector $\hat{v}:=v+\frac{dH(v)}{2\pi}\partial_{\rho}\in T_{(w,z)}\partial W_H$; we then let $p(v)$ be the projection of $\hat{v}$ to the summand $\xi_{\partial W_H}$ of the direct sum decomposition $T\partial W_H=  \xi_{\partial W_H}\oplus\langle R_{\partial W_H}\rangle$.  So if $\gamma\co S^1\to Y$, and hence $\pi\circ\gamma\co S^1\to W^{\circ}$, then $p$ induces an isomorphism $(\pi\circ\gamma)^*TW\cong \gamma^*\xi_H$ and so (since $H_1(W;\Z)=\{0\}$ and $c_1(TW)$ is torsion by assumption) we obtain a trivialization for $\gamma^*\xi_H$ by taking a trivialization of $(\pi\circ\gamma)^*TW$ that extends over a surface that bounds $\pi\circ\gamma$ and then applying $p$.  Denote one such trivialization by $\mathcal{U}_{\gamma}\co \gamma^*\xi_H\to S^1\times \C^{n}$.
\end{itemize}

\begin{prop}\label{relmas}
If $\gamma\co S^1\to \partial W_H\setminus(\partial W\times\{0\})$ projects via $(w,z)\mapsto z$ as a loop in $\C\setminus\{0\}$ having winding number $N$ around $0$, then the relative Maslov index of $\mathcal{T}_{\gamma}$ with respect to $\mathcal{U}_{\gamma}$ is $N$.  (In other words, the map $\mathcal{T}_{\gamma}\circ \mathcal{U}_{\gamma}^{-1}\co S^1\times \C^n\to S^1\times \C^n$ has form $(e^{i\theta},v)\mapsto(e^{i\theta}, a_{\gamma}(e^{i\theta})v)$ where the loop $a_{\gamma}\co S^1\to Sp(2n)$ has Maslov index $N$.)
\end{prop}

\begin{proof}
Homologous loops $\gamma$ in $Y= \partial W_H\setminus(\partial W\times\{0\})$ give rise to homologous transition maps $a_{\gamma}\co S^1\to Sp(2n)$, and the Maslov index gives rise to a well-defined homomorphism (in fact isomorphism) $H_1(Sp(2n);\Z)\to\Z$.  So it suffices to prove the proposition for a loop $\gamma$ whose homology class generates $H_1(Y;\Z)$. We have a diffeomorphism $Y\to  W^{\circ}\times S^1$ defined by $(w,z)\mapsto (w,z/|z|)$ and we are assuming that $H_1(W^{\circ};\Z)=\{0\}$, so we can take $\gamma(e^{i\theta})=\left(w,\sqrt{\frac{H(w)}{\pi}}e^{i\theta}\right)$ for any choice of $w\in W^{\circ}$ and we just need to show that $a_{\gamma}$ has Maslov index $1$ for this particular $\gamma$ (as the general loop in $Y$ that projects to a loop with winding number $N$ in $\C\setminus\{0\}$ will be homologous to an $N$-fold iterate of this loop).  For convenience choose $w=\eta_{-\delta}(w_0)$ for some $w_0\in \partial W$ and a small $\delta>0$, where for $t\leq 0$, $\eta_{t}\co W\to W$ denotes the time-$t$ flow of the Liouville vector field $\mathcal{L}$ of $(W,\lambda)$.

Since $\pi\circ\gamma$ is then the constant map to $w$, we can obtain a trivialization of the form $\mathcal{U}_{\gamma}$ from (ii) above by using a frame obtained by simply applying the bundle isomorphism $p$ to a symplectic basis for $T_wW$.  To choose a specific basis, note that $\eta_{-\delta}(\partial W)$ contains $w$ and is a contact manifold with respect to the contact form $\lambda|_{\eta_{-\delta}(\partial W)}$, so let us take $\{v_1,\ldots,v_{2n-2}\}$ to be a symplectic basis for the contact distribution at $w$, and extend it to a symplectic basis $\{v_1,\ldots,v_{2n-2},\mathcal{L}_w,R_w\}$ where $R_w$ is the value at $w$ of the Reeb vector field for this contact form and $\mathcal{L}_w$ is the value at $w$ of the Liouville field.  So $\mathcal{U}_{\gamma}$ is given by the frame $\{pv_1,\ldots,pv_{2n-2},p\mathcal{L}_w,pR_w\}$. (Actually since $v_i$ lie in the contact distribution one has $pv_i=\hat{v}_i$ in the notation used earlier.)

To construct $\mathcal{T}_{\gamma}$, for the disk $u\co D^2\to \partial W_H$ with boundary $\gamma$ let us use \[ u(re^{i\theta})=\left(\eta_{-r\delta}(w_0),\sqrt{\frac{H(\eta_{-r\delta}(w_0))}{\pi}}e^{i\theta}\right).\]  Now at $u(0)=(w_0,0)$, the contact distribution $\xi_{\partial W_H}$ is the direct sum of $(\ker\lambda|_{\partial W})\times \{0\}$ with $\{0\}\times \C$. Hence if $\delta$ has been chosen small enough, we can find a symplectic frame $\{e_1,\ldots,e_{2n}\}$ for $u^{*}\xi_{H}$ such that $e_1,\ldots,e_{2n-2}$ are each of the form $\hat{v}$ where $v$ lies in the image of $\ker\lambda|_{\partial W}$ under the Liouville flow, and such that $e_{2n-1},e_{2n}$ project to the $\C$ factor as positive multiples of, respectively $\partial_x$ and $\partial_y$. We use such a frame to construct $\mathcal{T}_{\gamma}$. If we do this in such a way that $e_1,\ldots,e_{2n-2}$ coincide along the boundary   with $\hat{v}_1,\ldots,\hat{v}_{2n-2}$ from the previous paragraph, then the transition matrix that expresses the frame for $\mathcal{U}_{\gamma}$ in terms of the frame for $\mathcal{T}_{\gamma}$ will act as the identity on the first $2n-2$ basis vectors, and its action on the remaining basis vectors will be given by the projections of positive multiples of the vector fields $p\mathcal{L}_w$ and $pR_w$ to $\C$.   Since $\lambda(\mathcal{L}_w)=0$ and $dH(\mathcal{L}_w)<0$, we see that $p\mathcal{L}_w$ projects to $\C$ as a negative multiple of $\partial_{\rho}$, while since along the boundary the Reeb field is positively proportional to $X_H$ (and hence $R_w$ will be almost positively proportional to $X_H$ if $\delta$ is small), $pR_w$ will, at least for small $\delta$, be close to a negative multiple of $\partial_{\theta}$.  Since the frame $\{-\partial_{\rho},-\partial_{\theta}\}$ makes one full counterclockwise turn around $S^1$, this shows that the basis change matrix from the frame for $\mathcal{U}_{\gamma}$ to the frame for $\mathcal{T}_{\gamma}$ has Maslov index one. As explained in the first paragraph of the proof this suffices to prove the proposition.
\end{proof}

We now consider the $\gamma_{w,N}$ described in  Corollary \ref{neworbits} (ii).  Here $w\in W^{\circ}$ and $N\in \Z_+$ with $\phi_H^N(w)=w$.  Provided that $H_1(W;\Z)=\{0\}$ and $c_1(TW)$ is torsion, we can associate a ``Hamiltonian Conley--Zehnder index'' $CZ_{Ham}(w,N)$ to any such pair $(w,N)$: find $u\co \Sigma\to W$ with $\partial\Sigma = \R/N\Z$ and $u|_{\partial\Sigma}(t)=\phi_{H}^{t}(w)$, trivialize $u^{*}TW$, and let $CZ_{Ham}(w,N)$ be the Conley--Zehnder index of the linearized flow $\{\phi_{H*}^{t}\}$ with respect to the restriction of the trivialization to the boundary.

\begin{prop}\label{czwn}
Let $(W,\lambda)$ and $H$ be as in Proposition \ref{suspliouville}, and assume that $H_1(W;\Z)=\{0\}$ and that $c_1(TW)$ is torsion. Let $w\in W^{\circ}$ and $N\in \Z_+$ be such that $\phi_{H}^{N}(w)=w$.  Then $\gamma_{w,N}$ is nondegenerate as a Reeb orbit provided that $\phi_{H*}^{N}\co T_wW\to T_wW$ does not have one as an eigenvalue.  Moreover \begin{equation}\label{czchange} CZ(\gamma_{w,N})=CZ_{Ham}(w,N)+2N.\end{equation}
\end{prop}

\begin{proof} We make use of the symplectic bundle isomorphism
$p\co \pi^*TW|_Y\to \xi_H|_Y$ from item (ii) before Proposition \ref{relmas}, where again $Y=\partial W_H\setminus (\partial W\times \{0\})$.  From the definition of this isomorphism and Corollary \ref{flowformula} it is easy to see that we have a commutative diagram, for $(w,z)\in Y$, \[ \xymatrix{ T_{w}W\ar[r]^{p}\ar[d]^{\phi_{H*}^{h(t,w)}} &  T_{(w,z)}\partial W_H  \ar[d]_{\psi^{t}_{*}} \\ T_{\phi_{H}^{h(t,w)}(w)}W \ar[r]^{p} & T_{\psi^t(w,z)}\partial W_H.
}\]  Under the assumptions of the proposition, applying this with $(w,z)=\left(w,\sqrt{\frac{H(w)}{\pi}}\right)$ and with $t=T_{w,N}$ (so that $h(t,w)=N$) makes it immediately clear that $\gamma_{w,N}$ is nondegenerate if and only if $1$ does not lie in the spectrum of $\phi_{H*}^{N}\co T_wW\to T_wW$.  Moreover the diagram shows that the Conley--Zehnder index of the linearized flow of $\psi^t$ along $\gamma_{w,N}$ \emph{with respect to the trivialization $\mathcal{U}_{\gamma}$} is equal to $CZ_{Ham}(w,N)$.  So (\ref{czchange}) follows from Proposition \ref{relmas} and the loop property \cite{Sal} of the Conley--Zehnder index. 
\end{proof}

\begin{remark}\label{zerodim}
Throughout the foregoing we have implicitly assumed that the dimension of the initial Liouville domain $(W,\lambda)$ is positive, as is typically done in the literature.  It is amusing to consider the effect of allowing $W$ to be zero-dimensional.  Indeed a strict reading of the definition shows that any compact zero-manifold $W$ becomes a Liouville domain with respect  to the unique element $0\in \Omega^1(W)$.  Since $\partial W=\varnothing$, the hypotheses of Proposition \ref{suspliouville} are satisfied by any positive function $H\co W\to \R$, and we will have $\tau_{\lambda,H}=H$.  If $W=\{w_1,\ldots,w_k\}$ and we write $H_j=H(w_j)$, then $W_H$ is just the disjoint union of disks of respective radii $\sqrt{\frac{H_j}{\pi}}$, and Proposition \ref{suspliouville} endows this with a Liouville form having Reeb vector field on the boundary of the $j$th disk given by $\frac{2\pi}{H_j}\partial_{\theta}$. Corollary \ref{neworbits}  identifies the $N$-fold cover of the boundary of the $j$th disk as the orbit denoted there by $\gamma_{w_j,N}$, and correctly gives its period as $NH_j$. 

The part of Proposition \ref{toppres} asserting that $H_1(W;\Z)\cong H_1(\partial W_H;\Z)$ is obviously no longer true when $\dim W=0$; indeed since now $\partial W=\varnothing$ it is no longer true that $H_0(\partial W;\Z)$ surjects to $H_0(W;\Z)$ and this was used in the proof.  Since the Reeb orbits $\gamma_{w_j,N}$ are now homologically essential in $\partial W_H$ the method of proof of Proposition \ref{czwn} no longer applies; we must use trivializations of $\gamma_{w_j,N}^{*}TW_H$ rather than of $\gamma_{w_j,N}^{*}\xi_H$ in order to compute $CZ(\gamma_{w_j,N})$.  However since the linearized Reeb flow along $\gamma_{w_j,N}$ is simply given by $N$ counterclockwise rotations, we see by direct inspection that $CZ(\gamma_{w_j,N})=2N$, consistently with Proposition \ref{czwn}.  So the conclusion of Proposition \ref{czwn} extends to the case that $\dim W=0$ even though the original proof does not.
\end{remark}

\begin{ex}\label{ellex}
Up to linear symplectomorphism, any closed ellipsoid in $\mathbb{C}^n$ can be expressed in the form \[ E(a_1,\ldots,a_n)=\left\{(z_1,\ldots,z_n)\in\C^n\left|\sum_{j=1}^{n}\frac{\pi|z_j|^{2}}{a_j}\leq 1\right.\right\}\] where each $a_j\in\R_+$.
Suppose that $n\geq 1$.  Write $\rho_j=\frac{1}{2}|z_j|^2$ for $j=1,\ldots,n$ and $\theta_j$ for the usual angular polar coordinate on the $j$th factor of $\C^n$, so that $\lambda_0=\sum_{j=1}^{n}\rho_jd\theta_j$ is a primitive for the standard symplectic form on $\C^n$.

If we let $(W,\lambda)=\left(E(a_1,\ldots,a_n),\lambda_0\right)$ and, for some $a_{n+1}>0$, define $H\co W\to \R$ by \[ H(z_1,\ldots,z_n)=a_{n+1}\left(1-\sum_{j=1}^{n}\frac{2\pi \rho_j}{a_j}\right),\] 
then the function $\tau_{\lambda,H}=\iota_{X_H}\lambda_0+H$ will be identically equal to $a_{n+1}$, and the Liouville domain $(W_H,\lambda_0+\rho_{n+1}d\theta_{n+1})$ produced by Proposition \ref{suspliouville} is evidently just $E(a_1,\ldots,a_{n+1})$, with its standard Liouville primitive.  Hence the function $h\co \R\times W\to \R$ of Proposition \ref{flowformula} is just given by $h(t,w)=\frac{t}{a_{n+1}}$.  

The Hamiltonian flow $\phi_{H}^{t}\co W\to W$ in this case is given by rotating the $j$th factor of $\C^n$ by the angle $2\pi a_{n+1}t/a_j$.  Assume that, for each $j\in\{1,\ldots,n\}$, $\frac{a_{n+1}}{a_j}\notin \Q$.  Then the only orbits $\gamma_{w,N}$ as in Proposition \ref{neworbits} (ii) will have $w=\vec{0}$, and their periods will be $Na_{n+1}$.  (More specifically, they are given by $\gamma_{\vec{0},N}(t)=\left(\vec{0},\sqrt{\frac{a_{n+1}}{\pi}}e^{2\pi it/a_{n+1}}\right)$.)  Moreover Proposition \ref{czwn} shows that in this case $\gamma_{\vec{0},N}$ is nondegenerate with  \[ CZ(\gamma_{\vec{0},N})=n+2N+2\sum_{j=1}^{n}\left\lfloor\frac{Na_{n+1}}{a_j}\right\rfloor = n+2\sum_{j=1}^{n+1}\left\lfloor\frac{Na_{n+1}}{a_j}\right\rfloor.\]

The other closed Reeb orbits on $\partial E(a_1,\ldots,a_{n+1})$ are, according to Proposition \ref{neworbits}, the orbits $\gamma_c(t)=(c(t),0)$ as $c$ varies through closed Reeb orbits on $\partial E(a_1,\ldots,a_n)$.  By Proposition \ref{czc} such an orbit is nondegenerate provided that $c$ is nondegenerate and the period $T_c$ of $c$ does not lie in $a_{n+1}\Z$, in which case the Conley--Zehnder indices are related by $CZ(\gamma_c)=CZ(c)+1+2\left\lfloor \frac{T_c}{a_{n+1}}\right\rfloor$.  

From the previous two paragraphs it follows by induction (with the base step given by setting $W=\{0\}$ and $H=a_1$ as in Remark \ref{zerodim}) that if every $\frac{a_j}{a_k}$ for $j\neq k$ is irrational, then the closed Reeb orbits on $\partial E(a_1,\ldots,a_{n+1})$ are all nondegenerate and consist of the orbits $\gamma(k,N)$ which wrap $N$ times counterclockwise around the circle of radius $\sqrt{\frac{a_k}{\pi}}$ in the $k$th factor of $\C^{n+1}$, as $k$ varies through $\{1,\ldots,n+1\}$ and $N$ varies through $\Z_+$, with the period of $\gamma(k,N)$ equal to $Na_k$ and the Conley--Zehnder index equal to $n+2\sum_{j=1}^{n+1}\left\lfloor\frac{Na_k}{a_j}\right\rfloor$.  This agrees with the standard calculation found \emph{e.g.} in \cite[Section 2.1]{GH}.
\end{ex}


\section{Truncated ellipsoids}\label{trunc}

Throughout this section fix $a_1,\ldots,a_{n+1}\in (0,\infty)$, corresponding to ellipsoids \[ E=E(a_1,\ldots,a_n)\subset \C^n, \qquad \hat{E}=E(a_1,\ldots,a_{n+1})\subset \C^{n+1}.\]  We will assume moreover that \begin{equation}\label{irrat1} \mbox{For each $j\in\{1,\ldots,n\}$, }\frac{a_{n+1}}{a_j}\notin \Q.\end{equation}

For $j=1,\ldots,n+1$ write $\rho_j=\frac{1}{2}|z_j|^2$ and $\theta_j$ for the standard polar coordinate on the $j$th factor of $\C^{n+1}$.  So the standard Liouville form on $\C^n$ is $\lambda_0=\sum_{j=1}^{n}\rho_j d\theta_j$, and by definition \[ E=\left\{\sum_j\frac{2\pi \rho_j}{a_j}\leq 1\right\}.\] As in Example \ref{introex} let us abbreviate \[ u:=\sum_{j=1}^{n}\frac{2\pi \rho_j}{a_j},\qquad \mbox{so that } E=\{u\leq 1\}.\]

For any continuous function $H\co E\to \R$ let us write \[ E_H=\left\{(w,z)\in E\times\C\left|\pi|z|^2\leq H(w)\right.\right\},\] which evidently has interior given by $E_{H}^{\circ}=\{(w,z)|\pi|z|^2<H(w)\}$.  According to Proposition \ref{suspliouville}, if $H$ is smooth with $H|_{\partial E}\equiv 0$ and if $\iota_{\lambda_0}X_H+H>0$ then $E_H$ is a Liouville domain with respect to the usual Liouville form $\hat{\lambda}_0:=\lambda_0+\rho_{n+1}d\theta_{n+1}$ on $\C^{n+1}$.  As a special case, as noted in Example \ref{ellex}, we have $\hat{E}=E_{a_{n+1}(1-u)}$.

For $0<\ep<1<\beta$ we define $H_{\ep,\beta}\co E\to \R$ by  \[ H_{\ep,\beta}=\min\{a_{n+1}(\ep+\beta u),a_{n+1}(1-u)\}.\]  Thus $H$ is continuous, and smooth everywhere except the locus where $u=\frac{1-\ep}{1+\beta}$.  Evidently $E_{H_{\ep,\beta}}$ is not a Liouville domain since its boundary is not even a smooth manifold.  We do however have:

\begin{prop} For any $0<\ep<1<\beta$, $(E_{H_{\ep,\beta}}^{\circ},\hat{\lambda}_0)$ is an open Liouville domain.
\end{prop}

\begin{proof} The Liouville vector field of $(E_{H_{\ep,\beta}}^{\circ},\hat{\lambda}_0)$ is $\hat{\mathcal{L}}=\frac{1}{2}\sum_{j=1}^{n+1}(x_j\partial_{x_j}+y_j\partial_{y_j})$; if this vector field is regarded as a vector field on $\C^{n+1}$ instead of just on $E_{H_{\ep,\beta}}^{\circ}$ its flow is complete.  Moreover $du(\hat{\mathcal{L}})=u$ and $d\rho_{n+1}(\hat{\mathcal{L}})=\rho_{n+1}$.  Now we can write \[ E_{H_{\ep,\beta}^{\circ}}=\left\{\max\{f,g\}<0\right\} \mbox{ where }f=2\pi\rho_{n+1}-a_{n+1}(1-u)\mbox{ and }g=2\pi\rho_{n+1}-a_{n+1}(\ep+\beta u).\]  We then have \[ df(\hat{\mathcal{L}})=f+a_{n+1},\qquad dg(\hat{\mathcal{L}})=g+a_{n+1}\ep.\]  Thus the time-$t$ flow of $\hat{\mathcal{L}}$ maps the sublevel set $\{f<0\}$ to the sublevel set $\{f<a_{n+1}(e^{t}-1)\}$ and the sublevel set $\{g<0\}$ to the sublevel set $\{g<a_{n+1}\ep(e^{t}-1)\}$; hence if $t<0$ this flow maps $E_{H_{\ep,\beta}}^{\circ}=\{\max\{f,g\}<0\}$ inside $\{\max\{f,g\}<a_{n+1}\ep(e^{t}-1)\}$ which (since $e^t-1<0$) has compact closure inside of $E_{H_{\ep,\beta}}^{\circ}$.
\end{proof}

When $\beta$ is large, we will see that $(E^{\circ}_{H_{\ep,\beta}},\lambda_0)$ is close to $(\hat{E}^{\circ},\lambda)$ with respect to the coarse symplectic Banach-Mazur distance $\delta_c$, whereas at least for certain choices of large $\beta$ and small $\ep$ the distance with respect to the fine symplectic Banach-Mazur distance is large.  The first of these statements follows relatively easily from the following lemma:

\begin{lemma}\label{translate}
Given $\beta>0$ and $0<\ep<1$ there is a Hamiltonian isotopy supported in $\hat{E}^{\circ}$ whose time-one map $\psi$ maps $\frac{\beta}{1+\beta}\hat{E}$ into the region $\hat{E}^{\circ}\cap \{u\geq \frac{1-\ep}{1+\beta}\}$.
\end{lemma}

\begin{proof} Let us abbreviate $\alpha=\frac{\beta}{1+\beta}$, so that $\frac{1}{1+\beta}=1-\alpha$ and the lemma is equivalent to the statement that  if $0<\alpha<1$ and $0<\gamma<1-\alpha$ then there is a Hamiltonian isotopy supported in $\hat{E}^{\circ}$ whose time-one map sends $\alpha \hat{E}$ into $E^{\circ}\cap \{u\geq \gamma\}$.

Choose a number $\delta$ with $0<\delta<1-\alpha-\gamma$ and let $K\co \C\times \R\to \R$ be a smooth function such that for each $v\in \R$ the function $K(\cdot,v)\co \C\to\R$ has support inside $\left\{\frac{\pi|z|^2}{a_1}<v+\delta\right\}$ (so in particular $K(\cdot,v)\equiv 0$ for $v\leq -\delta$) and such that the time-one map of the Hamiltonian flow of $K(\cdot,v)$ maps $\{\frac{\pi|z|^2}{a_1}\leq v\}$ into the complement of $\{\mathrm{Re}(z)\geq 0\}\cap \R$ in 
$\left\{\frac{\pi|z|^2}{a_1}<v+\delta\right\}$.
Define $G\co \C^{n+1}\to \R$ by \[ G\left(z_1,z_2,\ldots,z_{n+1}\right)=K\left(z_1,\alpha-\sum_{j=2}^{n+1}\frac{\pi|z_j|^2}{a_j}\right).\] Then:\begin{itemize} \item The value of $\sum_{j=2}^{n+1}\frac{\pi |z_j|^{2}}{a_j}$ is conserved under the Hamiltonian flow of $G$.
\item The support of $G$ is contained in $\{\sum_{j=1}^{n+1}\frac{\pi|z_{j}|^{2}}{a_j}<\alpha+\delta\}=(\alpha+\delta)\hat{E}^{\circ}$.
\item The time-one map $\phi_{G}^{1}$ sends $\alpha\hat{E}$ to the complement of $\{\mathrm{Re}(z_1)\in [0,\infty)\}$ in $(\alpha+\delta)\hat{E}^{\circ}$.\end{itemize}

In particular the polar coordinate $\theta_1$ gives a well-defined smooth function valued in $(0,2\pi)$ on the compact subset $\phi_{G}^{1}(\alpha\hat{E})$ of $(\alpha+\delta)\hat{E}^{\circ}$.  Since $\alpha+\gamma+\delta<1$, we can then use a cutoff version of $\frac{a_1\gamma}{2\pi}\theta_1$ as a Hamiltonian supported in $\hat{E}^{\circ}$ whose time-one map, say $\psi_1$, restricts to $\phi_{G}^{1}(\alpha\hat{E})$ as a map which increases $\frac{2\pi \rho_1}{a_1}$ by $\gamma$ and leaves $\rho_2,\ldots,\rho_{n+1}$ unchanged.  Thus \[ \psi_1(\phi_{G}^{1}(\alpha\hat{E}))\subset \hat{E}^{\circ}\cap \left\{\frac{2\pi \rho_1}{a_1}\geq \gamma\right\}\subset \hat{E}^{\circ}\cap \{u\geq \gamma\} \] and the lemma holds with $\psi=\psi_1\circ\phi_{G}^{1}$.
\end{proof}

\begin{cor}\label{coarsecvg}
For $0<\ep<1<\beta$ we have \[ d_c\left((E_{H_{\ep,\beta}}^{\circ},\lambda_0),(\hat{E}^{\circ},\lambda_0)\right)   \leq \delta_f\left((E_{H_{\ep,\beta}}^{\circ},\lambda_0),(\hat{E}^{\circ},\lambda_0)\right)\leq \left(\frac{1+\beta}{\beta}\right)^2.\]
\end{cor}

\begin{proof}
The first inequality is given by Proposition \ref{easyineq}.  For the second, observe firstly that $E_{H_{\ep,\beta}}^{\circ}\subset \hat{E}^{\circ}$, and secondly that for $u\geq \frac{1-\ep}{1+\beta}$ we have $1-u\leq \ep+\beta u$, in view of which $\hat{E}^{\circ}\cap\left\{u\geq \frac{1-\ep}{1+\beta}\right\}\subset E_{H_{\ep,\beta}}^{\circ}$.  Thus the map $\psi$ from Lemma \ref{translate} obeys $\psi\left(\frac{\beta}{1+\beta}\hat{E}^{\circ}\right)\subset E_{H_{\ep,\beta}}^{\circ}$.   Let $g\co \C^{n+1}\to\C^{n+1}$ denote the time-$\left(\log\frac{\beta}{1+\beta}\right)$ flow of the Liouville vector field $\hat{\mathcal{L}}=\sum_{j=1}^{n+1}\frac{1}{2}(x_j\partial_{x_j}+y_j\partial_{y_j})$ for $\lambda_0$ on $\mathbb{C}^n$, so that whenever $V\subset\C^n$ is an open subset with $(V,\lambda_0)$ an open Liouville domain we have $g(V)=\frac{\beta}{1+\beta}V$.  The Hamiltonian isotopy of Lemma \ref{translate} is conjugated by $g$ to a Hamiltonian isotopy with support in $\frac{\beta}{1+\beta}\hat{E}^{\circ}$ whose time-one map is $g\psi g^{-1}\co \frac{\beta}{1+\beta}\hat{E}^{\circ}\to \frac{\beta}{1+\beta}\hat{E}^{\circ}$.  Then \[ g\psi g^{-1}\left(\left(\frac{\beta}{1+\beta}\right)^2\hat{E}^{\circ}\right)\subset \frac{\beta}{1+\beta}E_{H_{\ep,\beta}}^{\circ}\subset \frac{\beta}{1+\beta}\hat{E}^{\circ}.\]  So the  embedding \[ h:=(g\psi g^{-1})^{-1}|_{\frac{\beta}{1+\beta}E_{H_{\ep,\beta}}^{\circ}}\co \frac{\beta}{1+\beta}E_{H_{\ep,\beta}}^{\circ}\xhookrightarrow{L} \frac{\beta}{1+\beta}\hat{E}^{\circ}\subset \hat{E}^{\circ} \] has image containing $\left(\frac{\beta}{1+\beta}\right)^2\hat{E}^{\circ}$, whence the conclusion is immediate from the definition of $\delta_f$.
\end{proof}

While the above shows that $\delta_f\left((E_{H_{\ep,\beta}}^{\circ},\lambda_0),(\hat{E}^{\circ},\lambda_0)\right)$ is quite small when $\beta$ is large, the remainder of the section will be devoted to proving results (culminating in Corollary \ref{dellu}) that show that $\delta_f\left((\hat{E}^{\circ},\lambda_0),(E_{H_{\ep,\beta}}^{\circ},\lambda_0)\right)$ becomes quite large for at least  some large $\beta$.   

For each sufficiently small $\delta>0$ choose a smooth function $h_{\delta}\co [0,1]\to [0,\infty)$ such that $h_{\delta}(s)=a_{n+1}(\ep+\beta s)$ for $s<\frac{1-\ep}{1+\beta}-\delta$, $h_{\delta}(s)=a_{n+1}(1-s)$ for $s>\frac{1-\ep}{1+\beta}+\delta$, and $h''_{\delta}(s)\leq 0$ for all $s$.  Also arrange that $\delta_1<\delta_2\Rightarrow h_{\delta_1}\geq h_{\delta_2}$.  Evidently $E_{H_{\ep,\beta}}^{\circ}=\cup_{m=1}^{\infty}E_{h_{\delta_m}\circ u}^{\circ}$ for any sequence $\delta_m\searrow 0$, with $E_{h_{\delta_m}\circ u}^{\circ}\subset E_{H_{\ep,\beta}}^{\circ}$.
 
Now for any smooth $h\co [0,1]\to [0,\infty)$ We have  $X_{h\circ u}=(h'\circ u)X_u$, and $X_u=-\sum_{j=1}^{n}\frac{2\pi}{a_j}\partial_{\theta_j}$, so \[ \iota_{X_{h\circ u}}\lambda_0+(h\circ u) = h\circ u-uh'\circ u.\]  For $h=h_{\delta}$ as in the previous paragraph, observe that $\frac{d}{ds}(h_{\delta}(s)-sh_{\delta}'(s))=-sh''_{\delta}(s)\geq 0$, so since $h_{\delta}(0)=a_{n+1}\ep$ we have $\iota_{X_{h_{\delta}\circ u}}\lambda_0+h_{\delta}\circ u\geq a_{n+1}\ep$ everywhere.  Thus by Proposition \ref{suspliouville} each $(E_{h_{\delta}\circ u},\lambda_0)$ is a Liouville domain, and for any sequence $\delta_m\searrow 0$ we can use Lemma \ref{exhaust} with $\bar{U}_m=E_{h_{\delta_m}\circ u}$ to gain information about $\mathring{\mathbb{CH}}(E_{H_{\ep,\beta}}^{\circ},\lambda_0)$.  

\begin{lemma}\label{betaorbits} Assume that $0<\delta<\frac{1-\ep}{1+\beta}<1<\beta$ and that, for each $j\in\{1,\ldots,n\}$, both $\frac{a_{n+1}}{a_j}$ and $\frac{\beta a_{n+1}}{a_j}$ are irrational.  Then the closed Reeb orbits of the contact form $\lambda_0$ on $\partial E_{h_{\delta}\circ u}$ consist only of:
\begin{itemize}\item[(i)] orbits contained in $\partial E\times\{0\}$, each having period bounded below by $\min_{1\leq j\leq n}a_j$;
\item[(ii)] for $N\in\Z_+$, the orbits $\gamma_{0,N}\co \R/Na_{n+1}\ep \Z\to  \partial E_{H_{m,\delta}}$ defined up to time translation by $\gamma_{0,N}(t)=(0,\ldots,0,\sqrt{a_{n+1}\ep/\pi}e^{\frac{2\pi i t}{a_{n+1}\ep}})$, which have period $Na_{n+1}\ep$ and are nondegenerate with Conley--Zehnder index\begin{equation}\label{cz0n} CZ(\gamma_{0,N})=n+2N+2\sum_{j=1}^{n}\left\lfloor \frac{-N\beta a_{n+1}}{a_j}\right\rfloor;\end{equation}
\item[(iii)] For $k\in \Z$ and $N\in \Z_+$, $j\in\{1,\ldots,n\}$ with $-Na_{n+1}< ka_j<N\beta a_{n+1}$, a collection of orbits $\gamma_{w,N}$ with periods $T_{w,N}$ obeying \begin{equation}\label{twnbound} T_{w,N}\geq Na_{n+1}\left(\ep-2\beta\delta\right)+(N\beta a_{n+1}-ka_j)\frac{1-\ep}{1+\beta}.\end{equation}
\end{itemize}
\end{lemma}

\begin{proof}
We apply Corollary \ref{neworbits}, which asserts that the closed Reeb orbits on $\partial E_{h_{\delta}\circ u}$ comprise orbits $\gamma_c$ associated to the closed Reeb orbits $c$ on $\partial E$ together with orbits $\gamma_{w,N}$ for each $w\in E,N\in \Z_+$ such that $\phi_{h_{\delta}\circ u}^{N}(w)=w$.   Orbits of the first type are accounted for by (i) in the statement of the lemma.  As for the others, we find that, for $w=(w_1,\ldots,w_n)\in E$, \begin{equation}\label{hdeltaflow} \phi_{h_{\delta}\circ u}^{t}(w)=\left(e^{-2\pi i th'_{\delta}(u(w))/a_1}w_1,\ldots,
e^{-2\pi i th'_{\delta}(u(w))t/a_n}w_n\right).\end{equation} Thus the point $w=0$ is fixed under the flow, giving orbits $\gamma_{0,N}$ for each $N$.  We have $\tau_{\lambda_0,h_{\delta}\circ u}(0)=a_{n+1}\ep$, so these orbits have period $Na_{n+1}\ep$ and are parametrized as described in item (ii) in the statement of the lemma; moreover their nondegeneracy and the formula for their Conley--Zehnder indices in (\ref{cz0n}) follows using Lemma \ref{czwn} from the facts that $h'_{\delta}(s)=\beta a_{n+1}$ for $s$ near $0$ and that the various $\frac{\beta a_{n+1}}{a_j}$ were assumed irrational.

It remains to consider the case that $\phi_{h_{\delta}\circ u}^{N}(w)=w$ with $w\neq 0$.  Based on (\ref{hdeltaflow}) this forces $h'_{\delta}(u(w))/a_j=k/N$ for some $j\in \{1,\ldots,n\}$ and $k\in \Z$.  Now for $s<\frac{1-\ep}{1+\beta}-\delta$ we have $h'_{\delta}(s)=\beta a_{n+1}$, while for $s>\frac{1-\ep}{1+\beta}+\delta$ we have $h'_{\delta}(s)=-a_{n+1}$, so since $\frac{\beta a_{n+1}}{a_j},\frac{a_{n+1}}{a_j}$ are all assumed irrational for $j\in\{1,\ldots,n\}$ we can have  $h'_{\delta}(u(w))/a_j=k/N$ only if $\left|u(w)-\frac{1-\ep}{1+\beta}\right|\leq \delta$.  Furthermore since $h''_{\delta}\leq 0$, in this case we will have $\beta a_{n+1}>\frac{k a_j}{N}>-a_{n+1}$, as stated in item (iii) of the lemma.  

It remains only to check the lower bound (\ref{twnbound}) on the period $T_{w,N}$. The function $\tau_{\lambda_0,h_{\delta}\circ u}$ from (\ref{taudef}) is given by $\tau_{\lambda_0,h_{\delta}\circ u}(z)=h_{\delta}(u(z))-u(z)h'_{\delta}(u(z))$ and in particular is constant along the Hamiltonian flow of $h_{\delta}\circ u$, so  $T_{w,N}=N\tau_{\lambda_0,h_{\delta}\circ u}(w)$.  As noted in the previous paragraph, 
we have $\left|u(w)-\frac{1-\ep}{1+\beta}\right|\leq \delta$.  It's easy to check that the assumption that $\beta>1$ implies that $h_{\delta}\left(\frac{1-\ep}{1+\beta}-\delta\right)<h_{\delta}\left(\frac{1-\ep}{1+\beta}+\delta\right)$, and so the fact that $h''_{\delta}\leq 0$ implies that $h_{\delta}$ is minimized on the interval $\left[\frac{1-\ep}{1+\beta}-\delta,\frac{1-\ep}{1+\beta}+\delta\right]$ at $\frac{1-\ep}{1+\beta}-\delta$.  So since $h'_{\delta}(u(w))=\frac{ka_j}{N}\in \left(-a_{n+1},\beta a_{n+1}\right)$, if $h_{\delta}'(u(w))\geq 0$  we have \begin{align*} T_{w,N} &\geq N\left(h_{\delta}\left(\frac{1-\ep}{1+\beta}-\delta\right)-\left(\frac{1-\ep}{1+\beta}+\delta\right)\frac{ka_j}{N}\right) 
\\ &= Na_{n+1}\ep+N\beta a_{n+1}\left(\frac{1-\ep}{1+\beta}-\delta\right)-ka_j\left(\frac{1-\ep}{1+\beta}+\delta\right) 
\\ & \geq N\left(a_{n+1}\ep-\beta a_{n+1}\delta-\frac{ka_j}{N}\delta\right) + (N\beta a_{n+1}-ka_j)\frac{1-\ep}{1+\beta}, \end{align*} which is bounded below by the right-hand side of (\ref{twnbound}) since $\frac{ka_j}{N}<\beta a_{n+1}$.  On the other hand if $h_{\delta}'(u(w))\leq 0$ we obtain \begin{align*} T_{w,N} & \geq N\left(h_{\delta}\left(\frac{1-\ep}{1+\beta}-\delta\right)-\left(\frac{1-\ep}{1+\beta}-\delta\right)\frac{ka_j}{N}\right) \\ &= Na_{n+1}(\ep-\beta \delta)+ka_j\delta + (N\beta a_{n+1}-ka_j)\frac{1-\ep}{1+\beta} \end{align*} which is greater than the right-hand side of (\ref{twnbound}) since $\frac{ka_j}{N}>-a_{n+1}>-\beta a_{n+1}$.
\end{proof}

Provided that $\delta<\frac{\ep}{2\beta}$, the expression (\ref{twnbound}) is, for every $N$ and $k$ under consideration, a sum of two nonnegative terms; our intention is to choose our parameters $\beta,\ep$ in such a way that (for very small $\delta$) one or the other of these terms is large in comparison to the period of the orbit $\gamma_{0,1}$, namely $a_{n+1}\ep$. In particular our parameters $\beta,\ep$ will satisfy, among other properties, $\beta \gg 1$ and $\ep<\frac{1}{\beta^2}$, which implies that \begin{equation}\label{ube} \frac{1-\ep}{1+\beta}>\frac{1}{\beta}\frac{1-\beta^{-2}}{1+\beta^{-1}}=\frac{1-\beta^{-1}}{\beta}>(\beta-1)\ep,\end{equation} so that a lower bound on $(N\beta a_{n+1}-ka_j)$ will make the second term in (\ref{twnbound}) much larger than $a_{n+1}\ep$.  While it is not possible to give a universal positive lower bound for $N\beta a_{n+1}-ka_j$ for all $N,k\in \Z$ under consideration (namely all those with $-Na_{n+1}<ka_j<N\beta a_{n+1}$), we will see below that a judicious choice of $\beta$ does make it possible to give such a lower bound when $N$ is not too large, while in the case that $N$ is large the first term in (\ref{twnbound}) will be large.  This is easiest when $a_1=\cdots=a_{n}$ (\emph{e.g.} when $n=1$) in which case we can choose $\beta$ so that $\frac{\beta a_{n+1}}{a_j}$ is very close to but smaller than some integer, and then inequalities $0<N\beta a_{n+1}-ka_j<\frac{a_j}{2}$ with $N,k\in \Z$ can readily be seen to force $N$ to be large.  A little more effort is required when we have distinct values $a_1,\ldots,a_n$, but we carry this out presently.

\begin{lemma}\label{dirichlet}
Fix $n\in \Z_+$ and $a_1,\ldots,a_{n+1}\in (0,\infty)$. Then there is a constant $c>0$ (depending on $a_1,\ldots,a_{n+1}$) and an unbounded open set $B$ of positive real numbers $\beta$ with the following property.   For each $\beta\in B$ there are $p_1,\ldots,p_n\in \Z$ such that, for each $j\in\{1,\ldots,n\}$, \[ 0<p_j\frac{a_j}{a_{n+1}} - \beta < \left\{\begin{array}{ll} c\beta^{-2} & \mbox{if }n=1 \\ c\beta^{-1/(n-1)} & \mbox{if }n\geq 2\end{array}\right..\]
\end{lemma}

\begin{proof} For the $n=1$ case we can simply take
\[ B=\bigcup_{r\in\Z_+}\left(\frac{ra_1}{a_2}-\left(\frac{ra_1}{a_2}\right)^{-2},\frac{ra_1}{a_2}\right),\] since if $\frac{ra_1}{a_2}-\left(\frac{ra_1}{a_2}\right)^{-2}<\beta<\frac{ra_1}{a_2}$ the desired inequality holds with $p_1=r$ and $c=1$.  So assume that $n\geq 2$.

By Dirichlet's theorem on simultaneous Diophantine approximations (see \emph{e.g.} \cite[Theorem II.1A]{Sch}), for every integer $Q>1$ there are integers $p_1,\ldots,p_n$ such that \begin{equation}\label{schmidt} 1\leq p_n < Q^{n-1}\mbox{ and } \left|p_n\frac{a_n}{a_j}-p_j\right|\leq \frac{1}{Q}\mbox{ for all }j\in\{1,\ldots,n-1\}.\end{equation} Thus there are tuples $(p_1,\ldots,p_n)\in \Z_{+}^{n}$ with $p_n$ arbitrarily large and \begin{equation}\label{approx} \left|p_n\frac{a_n}{a_j}-p_j\right|\leq p_{n}^{-1/(n-1)}\mbox{ for each }j\in\{1,\ldots,n-1\}:\end{equation} indeed if any of the $\frac{a_n}{a_j}$ are irrational any upper bound $P$ on $p_n$ would induce a positive lower bound for the values $\left|p_n\frac{a_n}{a_j}-p_j\right|$ for natural numbers $p_j,p_n$ with $p_n<P$, making (\ref{schmidt}) impossible to satisfy for large $Q$; on the other hand if the $\frac{a_n}{a_j}$ are all rational then we could find a tuple $(p_1,\ldots,p_n)\in \Z_{+}^{n}$ with each $p_n\frac{a_n}{a_j}-p_j=0$, and then arbitrarily large integer multiples of $(p_1,\ldots,p_n)$ would satisfy this same property.

For any tuple $(p_1,\ldots,p_n)$ obeying (\ref{approx}) and each $j\in\{1,
\ldots,n-1\}$, we therefore have \[ \left|p_j\frac{a_j}{a_{n+1}}-p_n\frac{a_n
}{a_{n+1}}\right|=\frac{a_j}{a_{n+1}}\left|p_j-p_n\frac{a_n}{a_j}\right|<\frac
{a_j}{a_{n+1}}p_{n}^{-1/(n-1)}.\]  So if $A=\max\left\{\frac{a_1}{a_{n+1}},
\ldots,\frac{a_n}{a_{n+1}}\right\}$ we have \[ \left|p_j\frac{a_j}{a_{n+1}}-p_
k\frac{a_k}{a_{n+1}}\right|<2Ap_{n}^{-1/(n-1)}\mbox{ for }j,k\in\{1,\ldots,n\}
.\]  Consequently if we choose $\beta$ with the property that \begin{equation}
\label{betaint} 0<\min_{1\leq j\leq n}\frac{p_j a_j}{a_{n+1}} -\beta < A p_{n}
^{-1/(n-1)}\end{equation} then it will hold that, for each $j\in\{1,\ldots,n\}
$, $0<p_j\frac{a_j}{a_{n+1}}-\beta<3Ap_{n}^{-1/(n-1)}$.  The fact that $\beta<
p_n\frac{a_n}{a_{n+1}}$ implies that $3Ap_{n}^{-1/(n-1)}<3A\left(\frac{a_n}{a_
{n+1}}\right)^{1/(n-1)}\beta^{-1/(n-1)}$.  

So setting $c=3A\left(\frac{a_n}{a_{n+1}}\right)^{1/(n-1)}$ (which depends only on $a_1,\ldots,a_{n+1}$),  for every tuple $(p_1,\ldots,p_n)$ obeying (\ref{approx}), any number $\beta$ obeying (\ref{betaint}) has $0<p_j\frac{a_j}{a_{n+1}}-\beta<c\beta^{-1/(n-1)}$ for each $j\in\{1,\ldots,n\}$.  Since the numbers $p_n$ arising in such tuples are unbounded above, so too (by (\ref{betaint})) are the corresponding values of $\beta$.
\end{proof}

\begin{lemma} \label{strongTbound}
There is a constant $C>0$, depending only on $a_1,\ldots,a_{n+1}$, with the following property.  Suppose that $\beta>2$ lies in the set $B$ from Lemma \ref{dirichlet}, that $0<\ep<\beta^{-2}$, and $0<\delta<\frac{\ep}{3\beta}$. Then all of the orbits $\gamma_{w,N}$ from item (iii) in Lemma \ref{betaorbits} have period $T_{w,N}$ obeying \[ T_{w,N}>\left\{\begin{array}{ll} C\beta \ep & \mbox{if }n=1 \\ C\beta^{1/(n-1)}\ep & \mbox{if }n\geq 2 \end{array}\right.  .\]
\end{lemma}

\begin{proof}
As noted in (\ref{ube}), the assumption that $0<\ep<\beta^{-2}$ implies that $\frac{1-\ep}{1+\beta}>(\beta-1)\ep$.  Also the assumption that $0<\delta<\frac{\ep}{3\beta}$ implies that $\ep-2\beta\delta>\frac{\ep}{3}$.  So (\ref{twnbound}) immediately gives \begin{equation}\label{twnbound2} T_{w,N}> \ep\left[\frac{a_{n+1}}{3}N+(\beta-1)(N\beta a_{n+1}-ka_j)\right].\end{equation}  Recall that $N\in \Z_+$ and $k\in \Z$ have $N\beta a_{n+1}-ka_j>0$. In particular both of the terms in parentheses are positive; we shall show that one or the other of them is always greater than $C\beta^{1/(n-1)}$ (or than $C\beta$ if $n=1$) for an appropriate constant $C$.

For this purpose, we use Lemma \ref{dirichlet} to write $\beta=\frac{p_ja_j}{a_{n+1}}-\zeta$ where \[ p_j\in \Z\mbox{ and } 0<\zeta<\left\{\begin{array}{cc} c\beta^{-2} & \mbox{if }n=1 \\ c\beta^{-1/(n-1)} & \mbox{if }n\geq 2\end{array}\right..\]  So \[ N\beta a_{n+1}-ka_j=(Np_j-k)a_j-N\zeta a_{n+1},\] so the fact that $N\beta a_{n+1}-ka_j>0$ implies that $Np_j-k>0$.  But each of $N,p_j,k$ is an integer, so this forces $Np_j-k\geq 1$ and hence proves that \[ N\beta a_{n+1}-ka_j\geq a_j-N\zeta a_{n+1}.\] 

So if $N\leq \frac{a_j}{2\zeta a_{n+1}}$ then we obtain $N\beta a_{n+1}-ka_j\geq \frac{a_j}{2}$ and so (\ref{twnbound2}) immediately gives \begin{equation}\label{twnbound3} T_{w,N}>\ep(\beta-1)\frac{a_j}{2}>\frac{a_j}{4}\beta\ep \quad\mbox{when }N\leq \frac{a_j}{2\zeta a_{n+1}}.\end{equation} (The last inequality follows from the assumption $\beta>2$.) On the other hand if $N>\frac{a_j}{2\zeta a_{n+1}}$ then our bound on $\zeta$ shows that \[ \frac{a_{n+1}}{3}N\geq \left\{\begin{array}{cc} \frac{a_j}{6c}\beta^2 & \mbox{if }n=1 \\ \frac{a_j}{6c}\beta^{1/(n-1)} & \mbox{if }n\geq 2\end{array}\right. \] and hence, by (\ref{twnbound2}), \begin{equation}\label{twnbound4} T_{w,N}>\left\{\begin{array}{cc} \frac{a_j}{6c}\beta^2\ep &\mbox{if }n=1 \\ \frac{a_j}{6c}\beta^{1/(n-1)} \ep &\mbox{if }n\geq 2\end{array}\right. \quad \mbox{when }N>\frac{a_j}{2\zeta a_{n+1}}.\end{equation} The lemma now follows directly from (\ref{twnbound3}) and (\ref{twnbound4}).
\end{proof}

\begin{prop} \label{eme} Assume that $\frac{a_{n+1}}{a_j}$ is irrational for each $j\in\{1,\ldots,n\}$. The  unbounded open set $B$ of Lemma \ref{dirichlet} and the constant $C$ of Lemma \ref{strongTbound} have the following property.  Let $M_{\beta}=\left\{\begin{array}{ll} C\beta & \mbox{if }n=1 \\ C\beta^{1/(n-1)} & \mbox{if }n\geq 2 \end{array}\right.$.  For all sufficiently large $\beta\in B$ such that each $\frac{\beta a_j}{a_{n+1}}$ is irrational, there is a grading $k_{\beta}\leq 1$ such that whenever $0<\ep<\beta^{-2}$ and $a_{n+1}\ep<s<t< M_{\beta}\ep$, the canonical map $\imath_{t,s}\co \mathring{\mathbb{CH}}_{k_{\beta}}^{s}(E_{H_{\ep,\beta,}}^{\circ},\lambda_0)\to \mathring{\mathbb{CH}}_{k_{\beta}}^{t}(E_{H_{\ep,\beta}}^{\circ},\lambda_0)$ has rank $1$.
\end{prop}

\begin{proof}
As mentioned just before Lemma \ref{betaorbits}, we can apply Lemma \ref{exhaust} with $U=E_{H_{\ep,\beta}}^{\circ}$ and $\bar{U}_m=E_{h_{\delta_m}\circ u}$ for any sequence $\delta_m\searrow 0$.  Consider the Reeb orbits on $\partial    E_{h_{\delta_m}\circ u}$ where we assume that $\delta_m<\frac{\ep}{3\beta}$.  
We have $M_{\beta}\ep=O(\beta^{-1})$, so for $\beta\in B$ sufficiently large all of the orbits in item (i) of Lemma \ref{betaorbits} will have period larger than $M_{\beta}\ep$.  The preceding lemma shows that all of the orbits in item (iii) also have period larger than $M_{\beta}\ep$.  As for those in item (ii), the orbit $\gamma_{0,1}$ has period $a_{n+1}\ep$ and Conley--Zehnder index $n+2+2\sum_{j=1}^{n}\left\lfloor\frac{-\beta a_{n+1}}{a_j}\right\rfloor$, which we define to be $k_{\beta}$.  Since each $\left\lfloor \frac{-\beta a_{n+1}}{a_j}\right\rfloor \leq -1$, we have $k_{\beta}\leq 2-n\leq 1$; if, as we henceforth assume, $\beta$ is large enough that there is some $j$ so that $\frac{\beta a_{n+1}}{a_j}>2$ then $N+\left\lfloor \frac{-N\beta a_{n+1}}{a_j}\right\rfloor$ will be a strictly decreasing function of the integer $N$ and so $CZ(\gamma_{0,N})\leq CZ(\gamma_{0,N-1})-2\leq k_{\beta}-2$ for all $N\geq 2$.  In particular in this case $\gamma_{0,1}$ is the only orbit from item (ii) whose Conley--Zehnder index lies in the  set $\{k_{\beta}-1,k_{\beta},k_{\beta}+1\}$.  Given this, the proposition follows directly from Lemma \ref{exhaust}.
\end{proof}

\begin{cor}\label{dellu}
Choose $a_1,\ldots,a_{n+1},\beta$ such that each $\frac{a_{n+1}}{a_j},\frac{\beta a_{n+1}}{a_j}$ is irrational, $\beta\in B$, and $\beta$ is large enough for the conclusion of Proposition \ref{eme} to hold.  Let $0<\ep<\beta^{-2}$. If $V\subset \C^{n+1}$ is the interior of any ellipsoid $E(b_1,\ldots,b_{n+1})$ we have $\delta_f((V,\lambda_0),(E_{H_{\ep,\beta}}^{\circ},\lambda_0))\geq M_{\beta}/a_{n+1}$.
\end{cor}

\begin{proof} If the corollary were false,  Proposition \ref{implantdelta} shows that, for some $a<(M_{\beta}/a_{n+1})^{1/2}$, there would be a $a^{1/2}$-implantation of $\mathring{\mathbb{CH}}(E_{H_{\ep,\beta}},\lambda_0)$ into $\mathring{\mathbb{CH}}(V,\lambda_0)$.  Choosing $s$ such that $a_{n+1}\ep<s<a^2s<M_{\beta}\ep$.  We then have a commutative diagram \[ \xymatrix{ CH_{k_{\beta}}^{s}(E_{H_{\ep,\beta}}^{\circ},\lambda_0) \ar[rr]\ar[rd] & & CH^{a^2s}_{k_{\beta}}(E_{H_{\ep,\beta}}^{\circ},\lambda_0) \\ & CH_{k_{\beta}}^{as}(V,\lambda_0)\ar[ru] & } \] where the top map has rank $1$ by Corollary \ref{eme}.  But $k_{\beta}\leq 1$, and the fact that $V$ is the interior of an ellipsoid implies that $CH_{k}^{L}(V,\lambda_0)=\{0\}$ for all $L>0$ whenever $k\leq 1$.  (Indeed, we can exhaust $V$  by ellipsoids $E(c_1,\ldots,c_{n+1})$ with all $\frac{c_j}{c_m}$ irrational $j\neq m$, and then Example \ref{ellex} shows that all Reeb orbits on the boundaries of these ellipsoids have index at least $n+2$.  Hence the ellipsoids being used to approximate $V$ have $CH_{k}^{L}=\{0\}$ for all $k\leq 1$ and $L>0$, and passing to the inverse limit shows that the same holds for $V$.) This contradiction proves the corollary.
\end{proof}
We have now completed the proof of Theorem \ref{elldist}, as all of the results mentioned in the proof outline in Section \ref{examples} have been established.
Truncated ellipsoids also give examples for which the first inequality in Proposition \ref{easyineq} can be strengthened to a strict inequality \begin{equation}\label{dcnotmin} d_c((U,\lambda),(V,\mu))<\min\{\delta_f((U,\lambda),(V,\mu)),\delta_f((V,\mu),(U,\lambda))\}.\end{equation}   For the remainder of this section choose a positive irrational number $\ep\ll 1$ and assume that the parameters $a_1,\ldots,a_n,a_{n+1}$ that were fixed at the start of the section are given by $a_1=\cdots=a_n=\pi$ and $a_{n+1}=\pi(1-\ep)$.  For $p\in \Z_+$ let $V_p=W_{H_{\ep,p}}$.  We will show that (\ref{dcnotmin}) holds with the domains $U,V$ given by $V_3$ and $V_4$ and with $\lambda=\mu=\lambda_0$.  

To see this, first note that Corollary \ref{coarsecvg} and the multiplicative triangle inequality show that \[ d_c((V_3,\lambda_0),(V_4,\lambda_0))\leq  \left(\frac{4}{3}\right)^2\cdot\left(\frac{5}{4}\right)^2=\frac{25}{9}.\]  We will show however that both $\delta_f((V_3,\lambda_0),(V_4,\lambda_0))$ and $\delta_f((V_4,\lambda_0),(V_3,\lambda_0))$ can be made arbitrarily large by taking $\ep$ small.

The proof of this is similar to that of Corollary \ref{dellu}, though the assumption that $a_1=\cdots=a_n$ simplifies matters somewhat.  Here is the key ingredient:

\begin{prop}\label{doubleknot} If $\ep$ is sufficiently small  then:
\begin{itemize} \item[(i)] The canonical maps $CH_{2-5n}^{s}(V_3,\lambda_0)\to CH_{2-5n}^{t}(V_3,\lambda_0)$ and $CH_{2-7n}^{s}(V_4,\lambda_0)\to CH_{2-7n}^{t}(V_4,\lambda_0)$ are each nonzero whenever $\ep\pi<s<t<\frac{\pi}{14}$; and \item[(ii)] $CH^{u}_{2-5n}(V_4,\lambda_0)=CH^{u}_{2-7n}(V_3,\lambda_0)=\{0\}$ for all $u<\frac{\pi}{14}$.
\end{itemize}
\end{prop}

\begin{proof}
As was done with the general domains $W_{H_{\ep,\beta}}$ we use Lemma \ref{betaorbits} to control the periods and indices of the closed Reeb orbits on smooth Liouville domains that approximate $V_3$ and $V_4$, and then appeal to Lemma \ref{exhaust}.  We first claim that, for the smoothings of either $V_3$ or $V_4$, the periods of orbits of type (i) and type (iii) in Lemma \ref{betaorbits} always have period greater than $\frac{\pi}{14}$ provided that our parameter $\ep$, and also the smoothing parameter $\delta$, are sufficiently small. In what follows we always assume that $0<\delta\ll \ep$. For orbits of type (i) this is clear: indeed since $a_1=\cdots=a_n=\pi$ these orbits have period an integer multiple of $\pi$.  For those of type (iii), the bound on $T_{w,N}$ in Lemma \ref{betaorbits} becomes in our case \begin{equation}\label{twnv} T_{w,N}\geq \pi(1-\ep)\left[N(\ep-2\beta\delta)+\frac{N\beta(1-\ep)-k}{1+\beta}\right], \end{equation}  where $\beta=3$ for $V_3$ and $\beta=4$ for $V_4$, and $N\beta(1-\ep)-k>0$.  In particular both of the summands in brackets are positive since we assume $\delta\ll \ep$. If $\ep$ is sufficiently small, the term $\pi(1-\ep)N(\ep-2\beta\delta)$ may be bounded below by $\frac{N\ep\pi}{2}$.  Meanwhile, the fact that $N,\beta,k$ are integers with $N\beta(1-\ep)-k>0$ implies that $N\beta-k\geq 1$.  So (\ref{twnv}) implies, for $0<\delta\ll \ep\ll 1$, \[ T_{w,N}>\max\left\{\frac{N\ep\pi}{2},\frac{\pi(1-\ep)(1-N\beta\ep)}{1+\beta}\right\}.\]  If $N\leq\frac{1}{7\ep}$ then (since $\beta\in\{3,4\}$), the second expression in the maximum is bounded below by $\frac{3\pi(1-\ep)}{35}>\frac{\pi}{14}$ (for small $\ep$), while if $N>\frac{1}{7\ep}$ then the first expression is greater than $\frac{\pi}{14}$.  So (for appropriately small $\ep,\delta$) the only orbits of period at most $\frac{\pi}{14}$ are those of type (ii) in Proposition \ref{betaorbits}, namely the orbits $\gamma_{0,N}$ for $N\in\Z_+$.

Now the orbit $\gamma_{0,1}$ has period $\ep\pi(1-\ep)<\ep\pi$ and Conley--Zehnder index \[ CZ(\gamma_{0,1})=n+2+2n\left\lfloor -\beta(1-\ep)\right\rfloor = \left\{\begin{array}{ll} 2-5n & \beta=3 \\ 2-7n & \beta=4\end{array}\right. \] (we assume here that $\ep<\frac{1}{4}$).  In view of Lemma \ref{exhaust} and the fact that the $CZ(\gamma_{0,N})$ all have the same parity this is sufficient to prove item (i) in the proposition, as we have a nondegenerate orbit of the appropriate index with period less than $\ep\pi$, and the other orbits all either have period greater than $\frac{\pi}{14}$ or are nondegenerate with Conley--Zehnder index of the same parity.  Moreover to prove item (ii) in the proposition it is sufficient to check that if $\beta=4$ then for all $N\geq 1$ we have $CZ(\gamma_{0,N})\neq 2-5n$, and similarly that if $\beta=3$ then for all $N\geq 1$ we have $CZ(\gamma_{0,N})\neq 2-7n$.  

Assume from now on that $\ep<\frac{1}{6}$.  Then \[ CZ(\gamma_{0,N})=n+2N+2n\lfloor -\beta N(1-\ep)\rfloor \] is, for $\beta\in\{3,4\}$, a strictly decreasing function of $N$.  For $\beta=4$ we have $CZ(\gamma_{0,1})=2-7n<2-5n$, so since $N\mapsto CZ(\gamma_{0,N})$ is decreasing $CZ(\gamma_{0,N})$ indeed never takes the value $2-5n$.  For $\beta=3$, we have $CZ(\gamma_{0,1})=2-5n>2-7n$ and $CZ(\gamma_{0,2})=4-11n<2-7n$, so in this case the decreasing function $CZ(\gamma_{0,N})$ never takes the value $2-7n$. As noted earlier this suffices to establish (ii).
\end{proof}

\begin{cor}\label{v34} Assume that $\ep$ is sufficiently small.  Then
$\delta_f((V_3,\lambda_0),(V_4,\lambda_0))\geq \frac{1}{14\ep}$ and $\delta_f((V_4,\lambda_0),(V_3,\lambda_0))\geq \frac{1}{14\ep}$. Hence  $\min\{\delta_f((V_1,\lambda_0),(V_2,\lambda_0)),\delta_f((V_2,\lambda_0),(V_1,\lambda_0))\}>d_c((V_1,\lambda_0),(V_2,\lambda_0))$.
\end{cor}

\begin{proof}
If first statement of the corollary were false then Proposition \ref{implantdelta} would allow us to find $a<\sqrt{\frac{1}{14\ep}}$ such that for each $s$ the canonical map $CH^{s}_{2-7n}(V_4,\lambda_0)\to CH^{a^2s}_{2-7n}(V_4,\lambda_0)$ factors through $CH^{as}_{2-7n}(V_3,\lambda_0)$.  If we choose $s$ such that $\ep\pi<s<a^2s<\frac{\pi}{14}$, then this is impossible by Proposition \ref{doubleknot} since the map in question is nontrivial but $CH^{as}_{2-7n}(V_3,\lambda_0)=\{0\}$.  Switching the roles of $V_3$ and $V_4$ in this argument and using grading $2-5n$ rather than $2-7n$ proves the second statement.  The last statement follows from the fact that, as noted earlier, $d_c((V_3,\lambda_0),(V_4,\lambda_0))\leq \frac{25}{9}$.
\end{proof}

\section{Sinkholes}\label{sink}

In this section we give the construction yielding Theorem \ref{quasiembed}, asserting that for any $D\in \N$ and $n\geq 1$ there is an embedding of $\Delta_D=\{x_1\geq\cdots\geq x_D\}\subset\R^D$ into $\mathcal{S}_{2n+2}$ that is quasi-isometric  with respect to the fine symplectic Banach-Mazur distance.

Fix $n,D\in \Z_+$.  Fix also $\eta>0$ and collection of Liouville embeddings $\phi_m\co B^{2n}(1+3\eta)\xhookrightarrow{L} \R^{2n}$ having disjoint images $B_m$ for $m=1,\ldots,D$.  Here for any $c>0$ we write \[ B^{2n}(c)=\left\{(z_1,\ldots,z_n)\in \C^n\left|\sum_{j=1}^{n
}\pi|z_j|^2\leq c\right.\right\} \] for the standard $2n$-dimensional ball of capacity $c$.  Additionally, fix an ellipsoid $W=E(a_1,\ldots,a_n)$ such that $\sqcup_{m=1}^{D}B_m\subset (1-\eta)W$. 
 
Let $\chi\co B^{2n}(1+3\eta)\to [0,1]$ be a smooth function such that $\chi|_{B^{2n}(1+\eta)}\equiv 0$ and $\chi|_{B^{2n}(1+3\eta)\setminus B^{2n}(1+2\eta)}\equiv 1$.  For each $m\in\{1,\ldots,D\}$ we have $\phi_{m}^{*}\lambda_0-\lambda_0=df_m$ for some smooth function $f_m$; the form $\lambda_0+d(\chi f_m)$ then coincides with $\lambda_0$ on $B^{2n}(1+\eta)$ and with $\phi_{m}^{*}\lambda_0$ outside of $B^{2n}(1+2\eta)$.  We accordingly define $\lambda\in \Omega^1(W)$ by \[ \lambda=\left\{\begin{array}{ll} \phi_{m}^{-1*}(\lambda_0+d(\chi f_m)) & \mbox{on }B_m\,(m=1,\ldots,D) \\ \lambda_0 &\mbox{elsewhere}\end{array}\right.\]  Thus $\lambda$ is a primitive for the standard symplectic form on $W\subset \C^n$, and coincides with the usual primitive $\lambda_0$ outside of $\cup_mB_m\subset (1-\eta)W$. In particular $(W,\lambda)$ is a Liouville domain.

Having fixed these data $W,\phi_m,\lambda$, we will now associate to every $\vec{\ep}=(\ep_1,\ldots,\ep_D)$ with $0<\ep_1\leq \ep_2\leq\cdots\leq \ep_D\leq 2$ a continuous function $H_{\vec{\ep}}\co W\to \R$ and hence a subset $W_{H_{\vec{\ep}}}=\{(z_1,\ldots,z_{n+1})|\pi|z_{n+1}|^{2}\leq H_{\vec{\ep}}(z_1,\ldots,z_n)\}\subset \C^{n+1}$.  Let $\xi\co [0,1]\to [0,2]$ be a smooth function such that $\xi(s)=2$ for $s\leq 1-\eta$, $\xi(1)=0$, and $\xi''(s)\leq 0$ for all $s$.  If $0<t\leq 2$ write $K_t\co B^{2n}(1+3\eta)\to \R$ for the function \[ K_t(z_1,\ldots,z_n)=\left\{\begin{array}{ll}  t+(2-t)\sum_{j=1}^{n}\pi|z_j|^2 & \mbox{if }\sum_{j=1}^{n}\pi|z_j|^2\leq 1 \\ 2 & \mbox{otherwise}\end{array}\right.\]  The promised function $H_{\vec{\ep}}$ is then given by \begin{equation}\label{hepdef} H_{\vec{\ep}}=\left\{\begin{array}{ll} K_{\ep_m}\circ\phi_{m}^{-1}& \mbox{on }B_m=\phi_m(B^{2n}(1+3\eta)) \\ \xi\circ u &\mbox{elsewhere} \end{array}\right.\end{equation} where as in Section \ref{trunc} the function $u\co W\to \R$ is defined by $u(z_1,\ldots,z_n)=\sum_{j=1}^{n}\frac{\pi|z_j|^2}{a_j}$.  Thus $H\equiv 2$ throughout $(1-\eta)W\setminus\cup_mB_m$, while at the origin of the $m$th ``sinkhole'' $B_m$ the value of $H_{\vec{\ep}}$ falls to the value $\ep_m$.

Evidently $H_{\vec{\ep}}$ is smooth away from $\cup_m\phi_m(\partial B^{2n}(1))$.  It will sometimes be useful to approximate $H_{\vec{\ep}}$ by smooth functions $H_{\vec{\ep},\delta}$ where $0<\delta<\eta$ constructed similarly to the smooth approximations used in Section \ref{trunc}: we define $H_{\vec{\ep},\delta}$ by replacing the function $K_{\ep_m}\co B^{2n}(1+3\eta)\to\R$ in (\ref{hepdef}) by a function $K_{\ep_m,\delta}(z_1\ldots,z_n)=k_{\ep_m,\delta}\left(\sum_j\pi|z_j|^2\right)$, where $k_{\ep_m,\delta}\co [0,\infty)\to [0,2]$ is smooth with $k_{\ep_m,\delta}(s)=2$ for $s\geq 1+\delta$, $k_{\ep_m,\delta}(s)=\ep_m+(2-\ep_m)s$ for $s\leq 1-\delta$, and $k''_{\ep_m,\delta}(s)\leq 0$ everywhere. (This automatically implies that $K_{\ep_m,\delta}\leq K_{\ep_m}$ everywhere.) We may and do also require that $k_{\ep_m,\delta_0}\geq k_{\ep_m,\delta_1}$ when $\delta_0<\delta_1$.  

\begin{prop}\label{hepstar}
Let $\hat{\lambda}=\lambda+\rho_{n+1}d\theta_{n+1}\in \Omega^1(\C^{n+1})$. For any choice of $\vec{\ep}=(\ep_1,\ldots,\ep_m)\in (0,2]^{D}$ and $\delta\in (0,\eta)$, $(W_{H_{\vec{\ep},\delta}},\hat{\lambda})$ is a Liouville domain, and $(W_{H_{\vec{\ep}}}^{\circ},\hat{\lambda})$ is an open Liouville domain.  Moreover if $\vec{\ep},\vec{\zeta}\in (0,2]^D$ and if $C:=\max_{1\leq m\leq D}\frac{\ep_m}{\zeta_m}\geq 1$ we have $C^{-1}W_{H_{\vec{\ep}}}^{\circ}\subset W_{H_{\vec{\zeta}}}^{\circ}$.
\end{prop} 

\begin{proof} Let $\hat{\mathcal{L}}$ be the Liouville vector field associated to $\hat{\lambda}$, and define $ f_{\vec{\ep},\delta}\co \C^{n+1}\to\R$ by \[ f_{\vec{\ep},\delta}(z_1,\ldots,z_{n+1})=\pi|z_{n+1}|^{2}-H_{\vec{\ep},\delta}(z_1,\ldots,z_n)\] and similarly $f_{\vec{\ep}}(z_1,\ldots,z_{n+1})=\pi|z_{n+1}|^2-H_{\vec{\ep}}(z_1,\ldots,z_n)$. Thus $W_{H_{\vec{\ep},\delta}}=\{f_{\vec{\ep},\delta}\leq 0\}$, and $W_{H_{\vec{\ep}}}^{\circ}=\{f_{\vec{\ep}}<0\}$.

Just as in (\ref{dfL}) we have \[ (df_{\vec{\ep},\delta})_{(z_1,\ldots,z_{n+1})}(\hat{\mathcal{L}})=f_{\vec{\ep},\delta}(z_1,\ldots,z_{n+1})+\tau_{\lambda,H_{\vec{\ep},\delta}}(z_1,\ldots,z_n)\] where $\tau_{\lambda,H_{\vec{\ep},\delta}}=H_{\vec{\ep},\delta}+\iota_{X_{H_{\vec{\ep},\delta}}}\lambda$.  Now on $(1-\eta)W\setminus \cup_m\phi_m(B^{2n}(1+\eta))$, the function  $H_{\vec{\ep},\delta}$ is identically equal to $2$, and so $\tau_{\lambda,H_{\vec{\ep},\delta}}=2$.  On $W\setminus (1-\eta)W$ one easily checks that (because $\xi''\leq 0$) we have $\tau_{\lambda,H_{\vec{\ep},\delta}}\geq 2$.

On $\phi_m(B^{2n}(1+\eta))$, the form $\lambda$ coincides with $\phi_{m}^{-1*}\lambda_0$, and so if $z=\phi_{m}(w_1,\ldots,w_n)$ with $\sum_{j=1}^{n}\pi|w_j|^{2}=s$ we readily find that \[ \tau_{\lambda,H_{\vec{\ep},\delta}}(z)=k_{\ep_m,\delta}(s)-sk'_{\ep_m,\delta}(s)\geq k_{\ep_m,\delta}(0)=\ep_m \] where the  inequality follows from the fact that $\frac{d}{ds}\left(k_{\ep_m,\delta}(s)-sk'_{\ep_m,\delta}(s)\right)=-sk''_{\ep_m,\delta}(s)\geq 0$.  

The above calculations show that \begin{equation}\label{dfLbound} df_{\vec{\ep},\delta}(\hat{\mathcal{L}})\left\{\begin{array}{ll} \geq f_{\vec{\ep},\delta}+ 2 & \mbox{on }\left(W\setminus\cup_m\phi_m(B^{2n}(1+\eta))\right) \\ \geq f_{\vec{\ep},\delta}+ \ep_m & \mbox{on }\phi_m(B^{2n}(1+\eta))\end{array}\right.. \end{equation}

In particular this shows that $\hat{\mathcal{L}}$ points outward along the boundary of $W_{H_{\vec{\ep},\delta}}$, whence $(W_{H_{\vec{\ep},\delta}},\hat{\lambda})$ is a Liouville domain.  For the other statements, for $T\in [0,\infty)$ let us consider $e^{-T}W_{H_{\vec{\ep}},\delta}^{\circ}$.  By definition this set consists of values $\gamma(T)$ where $\gamma\co [0,T]\to \C^{n+1}$ is a flowline of $-\hat{\mathcal{L}}$ having $f_{\vec{\ep}}(\gamma(0))<0$.  Now $f_{\vec{\ep},\delta}\searrow f_{\vec{\ep}}$ as $\delta\searrow 0$, so for some $\delta>0$ we also have $f_{\vec{\ep},\delta}(\gamma(0))<0$.      Writing $g=f_{\vec{\ep},\delta}\circ\gamma$, we see that \[ g'(t)=-(df_{\vec{\ep},\delta})_{\gamma(t)}(\hat{\mathcal{L}})\leq -\left(f_{\vec{\ep},\delta}(\gamma(t))+\min_m\ep_m\right)=-(g+\min_m\ep_m).\]  Thus writing $\ep=\min_m\ep_m$, the function $g+\ep$ obeys a differential inequality $(g+\ep)'\leq -(g+\ep)$, which (by an easy case of Gronwall's inequality) implies that $g(t)+\ep\leq (g(0)+\ep)e^{- t}$ for all $t\geq 0$.  Thus \[ f_{\vec{\ep}}(\gamma(T))\leq f_{\vec{\ep},\delta}(\gamma(T))\leq f_{\vec{\ep},\delta}(\gamma(0))e^{-T}-(1-e^{-T})\ep.\]

Thus for $T>0$, $e^{-T}W_{H_{\vec{\ep}}}$ is contained in $\{f_{\vec{\ep}}<-(1-e^{-T})\ep\}$, which has compact closure in  $W_{H_{\vec{\ep}}}^{\circ}=\{f_{\vec{\ep}}<0\}$.  This implies that $(W^{\circ}_{H_{\vec{\ep}}},\hat{\lambda})$ is an open Liouville domain.

Finally we compare $W_{H_{\vec{\ep}}}^{\circ}$ to $W_{H_{\vec{\zeta}}}^{\circ}$ for $\vec{\ep},\vec{\zeta}\in (0,2]^D$.  We are to show that if $T\geq 0$ with $e^T\geq \frac{\ep_m}{\zeta_m}$ for every $m$, then every flowline $\gamma\co [0,T]\to \C^{n+1}$ for $-\hat{\mathcal{L}}$ such that $f_{\vec{\ep}}(\gamma(0))<0$ has $f_{\vec{\zeta}}(\gamma(T))<0$.  As before choose $\delta>0$ so that $f_{\vec{\ep},\delta}(\gamma(0))<0$, and write $g=f_{\vec{\ep},\delta}\circ\gamma$.  

On $\phi_m(B^{2n}(1+\eta))\times \C$, the Liouville vector field $\hat{\mathcal{L}}$ is given by $\phi_{m*}\left(\sum_{j=1}^{n}\rho_j\partial_{\rho_j}\right)+\rho_{n+1}\partial_{\rho_{n+1}}$, in view of which the set $\phi_m(B^{2n}(1+\eta))\times \C$ is invariant under the negative-time flow of $\hat{\mathcal{L}}$.  Hence if our flowline $\gamma$ for $-\hat{\mathcal{L}}$ ever enters $\phi_m(B^{2n}(1+\eta))\times\C$, then it holds that $\gamma(T)\in \phi_m(B^{2n}(1+\eta))\times \C$; in particular the image of $\gamma$ intersects at most one of the $\phi_m(B^{2n}(1+\eta))\times\C$.  If the image of $\gamma$ does not intersect \emph{any} of the $\phi_m(B^{2n}(1+\eta))$, then clearly $\gamma(T)\in W_{H_{\vec{\zeta}}}^{\circ}$ since we have already seen that $W_{H_{\vec{\ep}}}^{\circ}$ is an open Liouville domain and, in this case, $\gamma$ maps to a region whose intersection with 
$W_{H_{\vec{\ep}}}^{\circ}$ is the same as its intersection with $W_{H_{\vec{\zeta}}}^{\circ}$.  

The remaining case is that the image of $\gamma$ intersects some (necessarily unique) $\phi_m(B^{2n}(1+\eta))\times\C$, in which case $\gamma(T)\in \phi_m(B^{2n}(1+\eta))\times \C$.  Clearly if $\zeta_m\geq \ep_m$ then $W_{H_{\vec{\ep}}}^{\circ}\cap\left(\phi_m(B^{2n}(1+\eta))\times \C\right) \subset W_{H_{\vec{\zeta}}}^{\circ}\cap \left(\phi_m(B^{2n}(1+\eta))\times \C\right)$, so the conclusion that $\gamma(T)\in  W_{H_{\vec{\zeta}}}^{\circ}$ follows from what we have already done.  So assume that $\zeta_m<\ep_m$.  Then (\ref{dfLbound}) shows that we have $g'(t)\leq -(g(t)+\ep_m)$, and so \begin{equation}\label{epzeta} f_{\vec{\ep}}(\gamma(T))\leq f_{\vec{\ep},\delta}(\gamma(T))\leq f_{\vec{\ep},\delta}(\gamma(0))e^{-T}-(1-e^{-T})\ep_m<\zeta_m-\ep_m \end{equation} since $f_{\vec{\ep},\delta}(\gamma(0))<0$ and $e^T\geq \frac{\ep_m}{\zeta_m}$.
But (since $\zeta_m<\ep_m$) it is clear from the definitions of $H_{\vec{\ep}},H_{\vec{\zeta}}$ that $H_{\vec{\ep}}-H_{\vec{\zeta}}\leq \ep_m-\zeta_m$ on $\phi_m(B^{2n}(1+\eta))$, and hence $f_{\vec{\zeta}}-f_{\vec{\ep}}\leq \ep_m-\zeta_m$ on $\phi_m(B^{2n}(1+\eta))\times\C$.  Thus (\ref{epzeta}) implies that $f_{\vec{\zeta}}(\gamma(T))<0$, as desired.
\end{proof}

\begin{cor}\label{makess}
For each $\vec{\ep}=(\vec{\ep}_1,\ldots,\vec{\ep}_D)$ with each $0<\ep_j\leq 2$, there is a symplectomorphism $F\co \C^{n+1}\to \C^{n+1}$ such that $\left(F(W_{H_{\vec{\ep}}}^{\circ}),\hat{\lambda}_0\right)$ is an open Liouville domain.
\end{cor}

\begin{proof}
Recalling the definition of $\xi$ and $u$ from the beginning of the section, given $\delta>0$ consider $W_{\delta\xi\circ u}$.  Provided that $\delta\leq \frac{1}{2}\min_m\ep_m$, we have $\delta\xi\circ u\leq H_{\vec{\ep}}$ everywhere on $W$, and hence $W_{\delta\xi\circ u}\subset W_{H_{\vec{\ep}}}$.  Let $\lambda_t=\lambda_0+t(\lambda-\lambda_0)\in\Omega^1(W)$.  In particular $\lambda_t$ is independent of $t$ outside of $(1-\eta)W$, while on $(1-\eta)W$ it holds that $\delta\xi\circ u$ is constant.  Thus $\tau_{\lambda_t,\delta\xi\circ u}=\iota_{X_{\delta\xi\circ u}}\lambda_t+\delta\xi\circ u$ is independent of $t$; moreover for $t=0$ we can see that this function is positive, in view of which Propositon \ref{suspliouville} shows that $(W_{\delta\xi\circ u},\hat{\lambda}_t)$ is a Liouville domain for all $t$.  Again because $\lambda_t$ coincides with $\lambda_0$ outside of $(1-\eta)W$, the other hypotheses of Lemma \ref{symplecto} are easily seen to be satisfied with $X=\C^{n+1}$, $W=W_{\delta\xi\circ u}$, and $\hat{\lambda}_{1-t}$ in the role of the one-forms denoted there by $\lambda_t$. So we obtain a symplectomorphism $F\co \C^{n+1}\to\C^{n+1}$ such that $F^*\hat{\lambda}_0-\hat{\lambda}$ is supported inside $W^{\circ}_{\delta\xi\circ u}$ and hence also inside $W_{H_{\vec{\ep}}}^{\circ}$.  It immediately follows that $\left(F(W_{H_{\vec{\ep}}}^{\circ}),F^{-1*}\hat{\lambda}\right)$ is an open Liouville domain.  Moreover it holds quite generally that if $(V,\mu)$ is an open Liouville domain then so too is $(V,\mu')$ for any $\mu'$ with $d\mu'=d\mu$ such that $\mu'-\mu$ has support contained in a compact subset of $V$, whence it follows that $\left(F(W_{H_{\vec{\ep}}}^{\circ}),\hat{\lambda}_0\right)$ is an open Liouville domain.
\end{proof}

\begin{cor}\label{uppersink}
If $\vec{\ep},\vec{\zeta}\in (0,2]^D$ then \[ \log d_f\left((W_{H_{\vec{\ep}}}^{\circ},\hat{\lambda}),( W_{H_{\vec{\zeta}}}^{\circ},\hat{\lambda})\right) \leq \max_m\left|\log\left(\frac{\ep_m}{\zeta_m}\right)^2\right|.\]
\end{cor}

\begin{proof} Indeed if $r=r_{\vec{\ep},\vec{\zeta}}$ denotes the right-hand side of the above inequality, the value $C$ of Proposition \ref{hepstar} has $C\leq e^{r/2}$ and so we have an inclusion $e^{-r/2}W_{H_{\vec{\ep}}}^{\circ}\subset W_{H_{\vec{\zeta}}}^{\circ}$.  But by construction $r_{\vec{\zeta},\vec{\ep}}=r_{\vec{\ep},\vec{\zeta}}$, so we likewise have an inclusion $e^{-r/2}W_{H_{\vec{\zeta}}}^{\circ}\subset W_{H_{\vec{\ep}}}^{\circ}$.  This leads to chains of inclusions \[ e^{-r}W_{H_{\vec{\zeta}}}^{\circ}\subset e^{-r/2}W_{H_{\vec{\ep}}}^{\circ}\subset W_{H_{\vec{\zeta}}}^{\circ}  \] and \[  e^{-r}W_{H_{\vec{\ep}}}^{\circ}\subset e^{-r/2}W_{H_{\vec{\zeta}}}^{\circ}\subset W_{H_{\vec{\ep}}}^{\circ}   \] which demonstrate that $d_f\left((W_{H_{\vec{\ep}}}^{\circ},\hat{\lambda}),( W_{H_{\vec{\zeta}}}^{\circ},\hat{\lambda})\right) \leq e^r$.
\end{proof}

\begin{prop}\label{hepper}
Let $\frac{1}{2}<b<1$, and assume that $0\leq \ep_1\leq \ep_2\leq\cdots\leq \ep_D<\frac{1}{2}$, that $\delta\in (0,\eta)$ is sufficiently small (depending on $b$ and the $\ep_m$), and that each $\ep_m$ is irrational.  Then the collection of closed Reeb orbits for $\hat{\lambda}$ on $\partial W_{H_{\vec{\ep},\delta}}$ has the following properties:
\begin{itemize} 
\item[(i)] Every such orbit having period at most $b$ is nondegenerate, and has Conley--Zehnder index congruent to $n$ modulo $2$.
\item[(ii)] There are precisely $D$ such orbits that have period at most $b$ and Conley--Zehnder index $2-3n$; these orbits are each simple and their periods are $\ep_1,\ldots,\ep_D$. \end{itemize}
\end{prop}

\begin{proof} The closed Reeb orbits on $\partial W_{H_{\vec{\ep},\delta}}$ are the orbits $\gamma_c,\gamma_{w,N}$ as described in Corollary \ref{neworbits}, as $c$ ranges over Reeb orbits on $\partial W$ and $(w,N)\in W^{\circ}\times\Z_+$ ranges over pairs with $\phi_{H_{\vec{\ep},\delta}}^{N}(w)=w$.  Now $W$ is an ellipsoid $E(a_1,\ldots,a_n)$ that contains a nonempty union of embedded balls of capacity $1+3\eta$, so by the non-squeezing theorem we automatically have each $a_j\geq 1+3\eta>b$.  So since the Reeb orbits on $\partial W$ have periods given by positive integer multiples of $a_j$, none of the $\gamma_c$ have period at most $b$.  So we can restrict attention to the $\gamma_{w,N}$.

The Hamiltonian flow of $H_{\vec{\ep},\delta}$ is easily seen to preserve the function $\tau_{\lambda,H_{\vec{\ep},\delta}}=H_{\vec{\ep},\delta}+\iota_{X_{H_{\vec{\ep},\delta}}}\lambda$, so by Corollary \ref{neworbits}, the orbit $\gamma_{w,N}$ has period $N\tau_{\lambda,H_{\vec{\ep},\delta}}(w)$.  As noted in the proof of Proposition \ref{hepstar}, we have $\tau_{\lambda,H_{\vec{\ep},\delta}}(w)\geq 2$ for $w\notin \cup_m\phi_m(B^{2n}(1+\eta))$, so it suffices to consider those $\gamma_{w,N}$ for $w\in \phi_m(B^{2n}(1+\eta))$ for some $m$.

Now on $\phi_m(B^{2n}(1+\eta))$, the Hamiltonian flow of $H_{\vec{\ep},\delta}$ is conjugated by $\phi_m$ to the Hamiltonian flow of the function $K_{\ep_m,\delta}(z_1,\ldots,z_m)=k_{\ep_m,\delta}(\pi\sum_{j=1}^{n}|z_j|^2)$ that was introduced just before Proposition \ref{hepstar}, and we have $\tau_{\lambda,H_{\vec{\ep},\delta}}=\tau_{\lambda_0,K_{\ep_m,\delta}}\circ\phi_{m}^{-1}$.  The Hamiltonian flow of $K_{\ep_m,\delta}$ is given by \[ \phi_{K_{\ep_m,\delta}}^{t}(\vec{z})=e^{-2\pi i k'_{\ep_m,\delta}(\pi\|\vec{z}\|^2)t}\vec{z}.\]  So for $\phi_{K_{\ep_m,\delta}}^{N}(\vec{z})=\vec{z}$ it must hold either that $\vec{z}=\vec{0}$ or that $k'_{\ep_m,\delta}(\pi\|\vec{z}\|^2)=\frac{r}{N}$ for some $r\in \Z$.  The key fact that we we will now demonstrate is that, provided that $\delta$ is sufficiently small, in the latter case we will always have $N\tau_{\lambda_0,K_{\ep_m,\delta}}(\vec{z})>b$.

To prove this fact, write $s_0=\pi\|\vec{z}\|^2$, and recall the now-familiar formula $\tau_{\lambda_0,K_{\ep_m,\delta}}(\vec{z})=k_{\ep_m,\delta}(s_0)-s_0k'_{\ep_m,\delta}(s_0)$.  Suppose that $k'_{\ep_m,\delta}(s_0)=\frac{r}{N}$ where $r\in \Z,N\in \Z_+$.  Recall that $k''_{\ep_m,\delta}(s)\leq 0$, and that $k'_{\ep_m,\delta}(s)=2-\ep_m$ for $s\leq 1-\delta$ and $k'_{\ep_m,\delta}(s)=0$ for $s\geq 1+\delta$. Recall also that we are assuming that $\ep_m$ is irrational, so our hypothesis implies that $s_0\geq 1-\delta$.  Also if $r=0$ then since $k_{\ep_m,\delta}(s)=2$ whenever $k'_{\ep_m,\delta}(s)=0$ we have $\tau_{\lambda_0,K_{\ep_m,\delta}}(\vec{z})=2>b$, so we may as well assume that $r\neq 0$, which forces $s_0<1+\delta$ and $k'_{\ep_m,\delta}(s_0)>0$.  So since $k'_{\ep_m,\delta}$ is a monotone function it follows that \begin{align*} N\tau_{\lambda_0,K_{\ep_m,\delta}}(\vec{z})&=N\left(k_{\ep_m,\delta}(s_0)-s_0k'_{\ep_m,\delta}(s_0)\right)\geq N\left(\ep_m+(2-\ep_m)(1-\delta)-\frac{r}{N}(1+\delta)\right)
\\&= N\ep_m-\delta(N(2-\ep_m)+r)+\left((2-\ep_m)N-r\right)>N(\ep_m-4\delta)+\left((2-\ep_m)N-r\right). \end{align*} Here the last inequality follows from the fact that $k'_{\ep_m,\delta}$ is monotone decreasing, so that $0<\frac{r}{N}<2-\ep_m$ and $N(2-\ep_m)+r< 4N$.  So provided that we choose $\delta<(1-b)\min_m\ep_m/4$, we obtain   \[ N\tau_{\lambda_0,K_{\ep_m,\delta}}(\vec{z})>Nb\ep_m+\left((2-\ep_m)N-r\right),\] where $(2-\ep_m)N-r>0$.  If $N\geq \frac{1}{\ep_m}$ this immediately yields $N\tau_{\lambda_0,K_{\ep_m,\delta}}(\vec{z})>b$.  On the other hand if $N<\frac{1}{\ep_m}$, we observe that since $2N>r$ where $N$ and $r$ are both integers we must have $2N-r\geq 1$, and so the above yields \[ N\tau_{\lambda_0,K_{\ep_m,\delta}}(\vec{z})>1+(b-1)N\ep_m>b.\] So the inequality $N\tau_{\lambda_0,K_{\ep_m,\delta}}(\vec{z})>b$ holds in any case.

The foregoing implies that (provided that $\delta<(1-b)\min_m\ep_m/4$) any closed Reeb orbit on $\partial W_{H_{\vec{\ep},\delta}}$ having period at most $b$ must be of the form $\gamma_{\phi_m(\vec{0}),N}$ for some $m\in\{1,\ldots,D\}$ and $N\in \Z_+$.  The period of such an orbit is $N\ep_m$.  The linearized Hamiltonian flow of $H_{\vec{\ep},\delta}$ along such an orbit acts (in terms of the obvious trivialization of $TW$) by rotation of each factor of $\C$ through an angle $-2\pi N(2-\ep_m)$.  So since $\ep_m$ is irrational, Proposition \ref{czwn} shows that $\gamma_{\phi_m(\vec{0}),N}$ is nondegenerate with Conley--Zehnder index \[ 
CZ(\gamma_{\phi_m(\vec{0}),N})=n+2N+2n\left\lfloor -N(2-\ep_m)\right\rfloor=n+2\left\lfloor -N(1-\ep_m)\right\rfloor+2(n-1)\left\lfloor -N(2-\ep_m)\right\rfloor.\]  This is evidently always congruent to $n$ modulo $2$.  We also clearly have $CZ(\gamma_{\phi_m(\vec{0}),N+1})\leq CZ(\gamma_{\phi_m(\vec{0}),N})$, with \[ CZ(\gamma_{\phi_m(\vec{0}),1})=n+2-4n=2-3n \] and (because $0<\ep_m<\frac{1}{2}$, so that $3<2(2-\ep_m)<4$) \[ CZ(\gamma_{\phi_m(\vec{0}),2})=n+4-8n=4-7n<CZ(\gamma_{\phi_m(\vec{0}),1}).\]  So the closed Reeb orbits having period at most $b$ and Conley--Zehnder index $2-3n$ are precisely the $CZ(\gamma_{\phi_m(\vec{0}),1})$, each of which is obviously simple.
\end{proof}

\begin{cor}\label{sinkrank}
Let $0<\ep_1\leq \ep_2\leq \cdots\leq \ep_D\leq \frac{1}{2}$ and suppose that $s<t<1$ and  either that $d\in\{1,\ldots,D-1\}$ with $\ep_d<s<\ep_{d+1}$, or that $d=0$ and $s<\ep_1$, or that $d=D$ and $s>\ep_D$.  Then $CH^{s}_{2-3n}(W_{H_{\vec{\ep}}}^{\circ},\hat{\lambda})$ has dimension $d$, and the canonical map $CH^{s}_{2-3n}(W_{H_{\vec{\ep}}}^{\circ},\hat{\lambda})\to CH^{t}_{2-3n}(W_{H_{\vec{\ep}}}^{\circ},\hat{\lambda})$ is injective.
\end{cor}

\begin{proof}
If the $\ep_m$ are each irrational this follows directly from Proposition \ref{hepper} and Lemma \ref{exhaust}, using the sequence of Liouville domains $\bar{U}_k=W_{H_{\vec{\ep},\delta_k}}$ for a sequence $\delta_k\searrow 0$.  If some of the $\ep_m$ are rational we can instead use $\bar{U}_k=W_{H_{\vec{\ep}^{(k)},\delta_k}}$ for a suitable sequence $\vec{\ep}^{(k)}$ of tuples of irrational numbers that converges to $\vec{\ep}$ with all $\vec{\ep}^{(k)}_{m}\leq\vec{\ep}^{(k+1)}_{m}$.
\end{proof}

The final ingredient necessary to complete the outline of the proof of Theorem \ref{quasiembed} in Section \ref{examples} is now the following:
\begin{cor}\label{quasicor}
If $0<\ep_1\leq \cdots\leq \ep_D\leq \frac{1}{2}$ and $0\leq\zeta_1\leq\cdots\leq \zeta_D\leq \frac{1}{2}$ then \[ \delta_f\left((W_{H_{\vec{\zeta}}}^{\circ},\hat{\lambda}),(W_{H_{\vec{\ep}}}^{\circ},\hat{\lambda})\right)\geq \max_{1\leq m\leq D}\left(\min\left\{\frac{1}{\ep_m},\left(\frac{\zeta_m}{\ep_m}\right)^2\right\}\right).\]
\end{cor}

\begin{proof}
If the corollary were false, then we could find $m\in\{1,\ldots D\}$, $s>\ep_m$ and $b>\delta_f\left((W_{H_{\vec{\zeta}}}^{\circ},\hat{\lambda}),(W_{H_{\vec{\ep}}}^{\circ},\hat{\lambda})\right)$ such that $b^{1/2}s<\zeta_m$ and $bs<1$. Without loss of generality we can assume that $s$ is distinct from each $\ep_j,\zeta_j$. By Proposition \ref{implantdelta} this would yield a commutative diagram \begin{equation}\label{sinkdiag}\xymatrix{ CH_{2-3n}^{s}(W_{H_{\vec{\ep}}}^{\circ},\hat{\lambda})\ar[rr]\ar[rd] & & CH_{2-3n}^{bs}(W_{H_{\vec{\ep}}}^{\circ},\hat{\lambda}) \\ & CH_{2-3n}^{b^{1/2}s}(W_{H_{\vec{\zeta}}}^{\circ},\hat{\lambda})\ar[ru] & } \end{equation}  Since $s>\ep_m$ and $bs<1$, the top map is injective with rank at least $m$.  But since $b^{1/2}s<\zeta_m$, another application of Corollary \ref{sinkrank} shows that $\dim CH_{2-3n}^{b^{1/2}s}(W_{H_{\vec{\zeta}}}^{\circ},\hat{\lambda})<m$, and so the commutativity of (\ref{sinkdiag}) leads to a contradiction.
\end{proof}

\end{document}